\documentclass[a4paper, 11pt, reqno, oneside]{amsart}
\usepackage{amsmath, amsthm, amssymb, amsfonts}
\usepackage{mathtools}
\usepackage{bbm}
\usepackage{enumerate}
\usepackage{tikz}
\usepackage{xcolor}
\usepackage{MnSymbol}
\usepackage[margin=1in]{geometry}
\usepackage{enumitem}
\usepackage{lipsum}
\usepackage{subcaption}

\usepackage{duckuments}

\usepackage[normalem]{ulem}

\definecolor{mutedCyan}{RGB}{20,110,150}
\usepackage[hypertexnames=false]{hyperref}
\hypersetup{
    colorlinks=true,
    citecolor=cyan,
    linkcolor=black,
}

\usepackage{thmtools}
\renewcommand{\thmcontinues}[1]{continued}
\usepackage{thm-restate}

\allowdisplaybreaks[2]

\let\originalleft\left
\let\originalright\right
\renewcommand{\left}{\mathopen{}\mathclose\bgroup\originalleft}
\renewcommand{\right}{\aftergroup\egroup\originalright}

\newcommand{\ve}{\varepsilon}

\newcommand\norm[1]{\left\lVert#1\right\rVert}
\newcommand\absv[1]{\left\lvert#1\right\rvert}
\renewcommand\brack[1]{\left (#1\right)}
\newcommand\abrack[1]{\left \langle #1 \right \rangle}
\newcommand\sbrack[1]{\left [ #1 \right ]}
\newcommand\cbrack[1]{\left \{ #1 \right \}}

\newcommand{\dd}{\mathop{}\!\mathrm{d}}

\renewcommand\P[1]{\mathbb{P}\brack{#1}}
\newcommand{\bbE}{\mathbb{E}}
\newcommand\E[1]{\bbE\sbrack{#1}}

\newcommand\V[1]{\mathrm{Var}\brack{#1}}

\newcommand{\vbar}{\overline{v}}
\newcommand{\btd}{\mathbb{T}^d}
\newcommand{\bt}{\mathbb{T}}

\newcommand{\bx}{\boldsymbol{x}}
\newcommand{\bv}{\boldsymbol{v}}
\newcommand{\bvbar}{\overline{\boldsymbol{v}}}
\newcommand{\bB}{\boldsymbol{B}}
\newcommand{\bj}{\boldsymbol{j}}
\newcommand{\by}{\boldsymbol{y}}
\newcommand{\bk}{\boldsymbol{k}}
\newcommand{\bxi}{\boldsymbol{\xi}}
\newcommand{\bn}{\boldsymbol{n}}
\newcommand{\bN}{\boldsymbol{\mathcal{N}}}

\newcommand{\xpde}{X_{\mathrm{SPDE}}}
\newcommand{\xcs}{X_{\mathrm{CS}}}

\newcommand{\ip}[2]{\left\langle #1,#2\right\rangle}
\newcommand{\hmtwo}{\boldsymbol{H}^{-2}}

\newcommand{\sav}{\overline{\mathcal{S}}}

\newtheorem{theorem}{Theorem}[section]
\newtheorem{proposition}[theorem]{Proposition}
\newtheorem{lemma}[theorem]{Lemma}
\newtheorem{corollary}[theorem]{Corollary}

\theoremstyle{definition}
\newtheorem{definition}[theorem]{Definition}

\newtheorem{remark}[theorem]{Remark}
\newtheorem{assumption}[theorem]{Assumption}
\newtheorem{inftheorem}[theorem]{Informal Theorem}

\numberwithin{equation}{section}
\numberwithin{theorem}{section}

\usepackage{amssymb}
\newcommand{\Tr}{\texttt{Tr}}

\usepackage{comment}

\title{On reduced SPDE models for stochastic Cucker--Smale flocking dynamics}
\author{Kamran Arora}
\address{Kamran Arora, \newline Department of Mathematics, University of Bath BA2 7AY, UK}
\email{ka679@bath.ac.uk}
\author{Federico Cornalba}
\address{Federico Cornalba, \newline Department of Mathematics, University of Bath BA2 7AY, UK}
\email{fc402@bath.ac.uk}
\author{Ana Djurdjevac}
\address{Ana Djurdjevac, \newline Mathematical Institute, University of Oxford OX2 6GG, UK}
\email{ana.djurdjevac@maths.ox.ac.uk}

\begin{document}

\begin{abstract}

Flocking refers to a phenomenon in collective motion where particles align their velocities with each other, commonly modelled using Cucker--Smale dynamics. In this work, we introduce and analyse two general reduced SPDE models for stochastic Cucker--Smale particle dynamics that encompass environmental and i.i.d.-type noises. Our models are computationally cheaper alternatives to the underlying particle system and preserve the flocking behaviour. To derive our reduced models, we consider suitable empirical particle densities which depend on position and time, but feature velocity dependency in a multiplicative fashion: this setup crucially allows for a very intuitive and convenient notion of flocking at the SPDE level. In $d=1$ we prove flocking of the reduced SPDE models, and, for the environmental noise, we also quantify the discrepancy between the reduced SPDE and mollified particle system. We also discuss the technical modifications required to extend the results to higher dimensions. We then provide extensive numerical experiments to verify our theoretical results in $d=1$ and to explore the behaviour of our models in $d=2$.

\end{abstract}

\maketitle

\section{Introduction}

This paper is devoted to modelling, analysis, and simulation of efficient continuous SPDE models which describe stochastic fluctuations for a vast and comprehensive set of interacting particle systems exhibiting \emph{flocking} phenomena. Flocking refers to a prominent phenomenon of collective motion, in which particle systems acquire ordered state by having specific features aligning over time. The particle systems we consider are of inertial type (in position $\bx$ and velocity $\bv$). Their general form (for $N$ particles) reads
\begin{equation}\label{general_part_sys}
\begin{cases}\tag{g-PS}
    \dd \bx_i(t) = \bv_i(t) \dd t, \\
    \dd \bv_i(t) = N^{-1}\sum_{j=1}^N a(\bx_i(t) - \bx_j(t))(\bv_j(t) - \bv_i(t)) \dd t + \dd \bn_i, \qquad i \in \{1,\dots, N\},
\end{cases}
\end{equation}
where $a$ is an interaction kernel, and $\bn_i$ are microscopic noises specific to the application. Here, flocking has to do with the velocities $\bv_i(t)$ converging (i.e.\ aligning) to the average velocity $\overline{\bv}(t)$ over time: informally, this means that
\begin{align}\label{informal_flocking}
    \bv_i(t) \approx \overline{\bv}(t) := N^{-1}\sum_{i=1}^{N}{\bv_i}(t)\qquad \mbox{for }t\rightarrow \infty
\end{align}
in ways that will be made precise later.
The system \eqref{general_part_sys} is quite general, and encompasses applications to collective animal behaviour  (bird flocks, school of fish), consensus phenomena (opinion dynamics), dissipative particle dynamics, among others. 

Although the noise-less dynamics $\bn_i\equiv 0$ in \eqref{general_part_sys} provides a solid starting point for describing flocking phenomena, a faithful and realistic description can often be better achieved by taking into account the intrinsic random fluctuations of the particles.

Furthermore, simulating \eqref{general_part_sys} on the microscopic scale is a computationally demanding task when $N\gg 1$.
In order to address the aforementioned two aspects (the need to account for the fluctuations and the need to reduce the computational burden of the simulations) continuous models (usually in the form of SPDEs) are being proposed more and more as computationally tractable, accurate descriptors of \eqref{general_part_sys}. 
In essence, such models augment the `average' mean-field PDE dynamics of the particle system using suitable space-time noises designed to provide an aggregate description of the particle's ensemble fluctuations. 

Though several stochastic noises $\bn_i$ in \eqref{general_part_sys} may be considered, we will restrict our considerations to just two physically relevant noises.

\textbf{Environmental noise.}
The first noise is
    \begin{align}\label{env_micr_noise}\tag{env-noise}
        \bn_i = \sigma(\bv_i - \overline{\bv})\bB
    \end{align}
    for $\sigma\in\mathbb{R}$ and a single $d$-dimensional Brownian motion $\bB$, and where the product $\sigma(\bv_i - \bvbar)\bB$ is to be read component-wise. This noise, introduced in \cite{ahn2010stochastic} and further studied in \cite{ton2014flocking,cattiaux2018stochastic}, models a shared stochastic perturbation from the environment: all particles are driven by the same Brownian motion $\bB$, and the noise intensity is proportional to the deviation from the mean velocity $\bvbar$. By construction, the average velocity is preserved over time, and the noise vanishes as the velocities align. Physically, this is well-suited to collective motion in fluctuating environments, such as bird flocks subject to turbulent wind or fish schools in oscillating water currents, as well as opinion dynamics models driven by a common external stochastic influence. 
    
  \textbf{Independent noise:} The second type of noise is 
   \begin{align}\label{indep_micr_noise}\tag{iid-noise}
        \bn_i = \sigma(\bv_i - \overline{\bv})\bB_i,
    \end{align}
    for $\sigma\in \mathbb{R}$ and independent $d$-dimensional Brownian motions $(\bB_i)_{i=1}^{N}$. This noise, studied in \cite{erban2016cucker,pedeches2018asymptotic,sun2015positive,cattiaux2018stochastic}, assigns each particle its own independent random fluctuation, modelling individual variability while retaining the velocity-aligning structure. Direct applications of our theory include pedestrian and agent-based dynamics where each individual carries its own intrinsic uncertainty. Furthermore, our theory covers specific instances of mesoscopic fluid models of dissipative particle dynamics (DPD) \cite{hoogerbrugge1992simulating,espanol1995statistical} in zero-temperature regimes. We envisage that the results of this paper may also be extended to encompass more general DPD cases.

Our approach produces SPDE models of reduced type, which track the particle density and momentum density 
\begin{align}
\rho(\bx,t) & = N^{-1}\sum_{i=1}^{N}{\delta(\bx - \bx_i(t))},\label{rho} \\
 \bj(\bx,t) & = N^{-1}\sum_{i=1}^{N}{\bv_i(t)\delta(\bx - \bx_i(t))},\label{j}
\end{align}
thus only considering position variables $\bx_i$ as independent variables. 
The formulation \eqref{rho}--\eqref{j} is extremely convenient for reading off flocking: specifically, on a formal level, one can `pull out' the velocities $\bv_i(t)$ in the $\bj$ density in \eqref{j} in light of the flocking condition \eqref{informal_flocking}, thus identifying flocking at the continuous level via the condition 
\begin{align}\label{informal_flock_continuous}
\bj(\bx,t) \approx \overline{\bv}(t)\rho(\bx, t),
\end{align}
where the approximation will be made precise later. The model reduction entails several advantages, above all a cost reduction over full-density models corresponding to the full empirical density $N^{-1}\sum_{i=1}^{N}{\delta(\bx - \bx_i(t),\bv - \bv_i(t))}$. In contrast to most reduced models, ours does not hinge on the full mean-field limit, but highlights the dependency on the velocity directly at the molecular level in \eqref{j}.

Our SPDE, in skeleton form, reads
    \begin{equation}\label{skeleton_spde}\tag{red-SPDE}
    \begin{cases}
        \dd \rho = -\nabla \cdot \bj \dd t, \\
        \dd \bj = \boldsymbol{\mathcal{D}}_{conv}\dd t + \boldsymbol{\mathcal{D}}_{kin} \dd t + \dd \bN,
    \end{cases}
    \end{equation}
where $\boldsymbol{\mathcal{D}}_{conv}$ treats convolutional terms arising from the kernel $a$, while $\boldsymbol{\mathcal{D}}_{kin}$  accounts for kinetic-type approximations. The skeleton equation \eqref{skeleton_spde} will be substantiated depending on the noise type, the space dimension, and the nature of the potential $a$ (constant or non-constant).
Furthermore, depending on the specific setting, we may need to complement the general equation \eqref{rho}--\eqref{j} with additional equations for reduced quantities, in a kinetic-type object $\boldsymbol{k}$. We defer the specific discussion to later stages of the paper.

Python implementations of the stochastic interacting particle systems and reduced SPDEs in arbitrary spatial dimension $d$ can be found in the open-source GitHub repository \url{https://github.com/kamran-arora/reduced-SPDE-stochastic-cucker-smale}.

\begin{remark}[Dean--Kawasaki noise]\label{remark: dk-noise} Another common choice of microscopic noise would be $\bn_i = \boldsymbol{B}_i$ for independent Brownian motions $(\boldsymbol{B}_i)_{i=1}^{N}$. This type of noise, while not feauturing the velocity-aligning factor $\bv_i - \overline{\bv}$ of \eqref{env_micr_noise}--\eqref{indep_micr_noise}, gives non-trivial flocking behaviour \cite{pedeches2018asymptotic}, and leads to a spatial Dean--Kawasaki-type noise \cite{Dean1996, kawasaki1994stochastic}. However, the absence of the factor $\bv_i - \overline{\bv}$ prevents a direct analysis under the continuous alignment ansatz \eqref{informal_flock_continuous}, and we therefore defer its analysis to subsequent works.

\end{remark}

\subsection{Overarching mathematical  strategy and main results}

Regardless of the specific nature of the particle system \eqref{general_part_sys} and corresponding reduced model \eqref{skeleton_spde}, our approach to proving flocking estimates in the continuous reduced SPDE models follows a unified pattern. 
The main mathematical ingredients are as follows:
\begin{itemize}
\item a quantity $X_t$ codifying flocking at the continuum level. In light of \eqref{informal_flock_continuous}, this is usually 
\begin{align}\label{informal_SPDE_flocking_quantity}
X_t := \|\bj(\bx,t)-\overline{\bv}(t)\rho(\bx,t)\|_{L^2}^2
\end{align}
or a closely related quantity.
\item a weighted $L^2$-energy norm $\eta$ (typically $\eta := \|\bj(t)\|_{L^2}^2 + |\overline{\bv}(t)|^2\|\rho(t)\|_{L^2}^2$, or closely related quantity).
\item a stopping time $\tau$ incorporating two pieces of information, namely: the exit time of the energy norm $\eta(t)$ from a ball of radius $M$, and the first time the average velocity $\overline{\bv}(t)$ becomes `degenerate' (i.e., its absolute value gets too close to zero, or too large). 
\end{itemize}
Furthermore, crucial quantities related to the skeleton equation \eqref{skeleton_spde} and feeding directly into our results are the scaling (alternatively, `magnitude', `amplitude') of the noise $\bN$ (which we denote by $\mathcal{S}$) and of its spatial average $\int{\bN}$ (which we denote by $\sav$) with respect to relevant parameters. The specifics of $\mathcal{S},\sav$ are given in the informal statements of our main results, Theorem \ref{thm_informal_env} (environmental noise case) and Theorem \ref{thm_informal_iid} (i.i.d. noise case) below. Before we state these results, we provide a skeleton of the strategy behind our proof of SPDE flocking.
\begin{enumerate}[label=\emph{Step \arabic*:}, leftmargin=*]
    \item \emph{Moment bounds for flocking quantity $X_t$ up to stopping time $\tau$}. Here, we show exponential decay for the flocking quantity $X^p_{t\wedge \tau}$, for $p\in \mathbb{N}$ up to
    \begin{align}
        p \lesssim \mathcal{S}^{-2}.
    \end{align}
    The `strength' of notion of flocking which we consider differs depending on the setting, but we can always obtain an estimate of the type 
    \begin{align}\label{step1_flocking}
    \mathbb{E}[e^{\kappa p(t \wedge \tau)}X^p(t \wedge \tau)]\leq \E{X^p(0)}.
    \end{align}
    for some $\kappa>0$.
    Informally, flocking occurs in expectation up to the stopping time $\tau$, i.e., so long as the energy $\eta$ of the system does not get too large and the average velocity does not become degenerate.
    We therefore need to study the properties of the stopping time $\tau$.
    
    \item \emph{Bounding the probability of the average velocity $\overline{\bv}(t)$ ever becoming degenerate}. 
    We show that 
    \begin{align}\label{exp_bound_av_vel_informal}
    \P{\overline{\bv}(t)\mbox{ is degenerate for some }t>0}\lesssim \exp\{-C\sav^{-\alpha}\}
    \end{align}
    for some $C,\alpha >0$.
    \item \emph{Bounding the moments of the system's energy $\eta_t$}. Here we show that, up to the stopping time $\tau$, the expected value of the energy moment $\eta^p(t\wedge \tau)$ is bounded uniformly-in-time by a constant $K^p$. Under suitable and realistic smallness assumptions on initial data and other relevant parameters,  we can guarantee $K<M$, (we recall that $M$ is part of the definition of the stopping time $\tau$).
    \item \emph{Bounding the probability of the energy $\eta$ exceeding the threshold $M$}. Using Steps 2 and 3, and a Chebyshev inequality-type argument, we get 
$$
\P{\mbox{energy }\eta(s)\mbox{ exceeds threshold }M\mbox{ for some }s<t}\leq \frac{\E{\eta_{t\wedge \tau}^p}}{M^p} \leq \left(\frac{K}{M}\right)^p.
$$
By taking the maximum value of $p$ for which flocking occurs (so $p
\lesssim \mathcal{S}^{-2}$) and considering that $K<M$, this readily implies
\begin{align}\label{exp_bound_eta_informal}
\P{\mbox{energy }\eta(s)\mbox{ exceeds threshold }M\mbox{ for some }s<t} \lesssim \exp\{-\mathcal{S}^{-2}\}.
\end{align}
\item \emph{Finalising the argument}.
Combining \eqref{exp_bound_av_vel_informal} and \eqref{exp_bound_eta_informal} (i.e.\ addressing the dual nature of the stopping time) gives
    \begin{align}\label{exp_decaying_bound_prob_tau}
    \mathbb{P}(\tau < \infty) \lesssim \exp\{-C\overline{\mathcal{S}}^{-\alpha}\} + \exp(-\mathcal{S}^{-2}).
    \end{align} 
    Combining \eqref{exp_decaying_bound_prob_tau} (which decays exponentially as the size of the noise goes to zero) and \eqref{step1_flocking} we get the flocking result
    \begin{align}\label{step3_flocking}
        \mathbb{E}[\mathbf{1}_{\tau = \infty}X(t)]\leq e^{- \kappa t}\E{X(0)}.
    \end{align}
\end{enumerate}

\begin{remark}
    \emph{Steps 1--5} above amount to the most `complete' version of the flocking proof: some of these steps are actually not necessary in certain cases (see the environmental noise case in Theorem \ref{thm_informal_env}). Furthermore, due to technicalities, it may be necessary to intertwine Steps 2 and 3 (see the independent noise case in Theorem \ref{thm_informal_iid}).
\end{remark}
\begin{remark}
    We note that, in the limit $\sigma\rightarrow 0$, $\lim_{\sigma\rightarrow 0}\P{\tau<\infty} = 0$, and this is consistent with flocking results for the deterministic flocking models in \cite{zimper2025reduced}.
\end{remark}

\subsubsection{Main Results}

Our rigorous mathematical results are formulated for $d=1$. Adaptations for the case $d>1$ are discussed later. 

\begin{inftheorem}[SPDE flocking for environmental noise case and error analysis, $d=1$]\label{thm_informal_env} In dimension $d=1$, consider the environmental noise \eqref{env_micr_noise}. The associated SPDE \eqref{skeleton_spde} is 
\begin{align*}
        \dd \rho & = -\partial_x j \dd t, \\
        \dd j & = [\rho(a*j) - j(a*\rho) - \vbar^2 \partial_x \rho] \dd t + \sigma(j-\vbar \rho) \dd B_t,
\end{align*}
and the associated noise scaling is $\mathcal{S} = \sav = \sigma$. For either a constant kernel $a = C_a$ or sufficiently regular kernel $a$ not significantly deviating from a suitable reference value $C_a$, and upon suitable smallness assumptions on the initial data $\rho_0$ and $j_0$ and of the noise intensity $\sigma$ so that $\sigma^2 - C_a < 0$, we get the bounds \eqref{step3_flocking} and \eqref{exp_decaying_bound_prob_tau} with $\kappa = C_a - \sigma^2$, see Theorem \ref{thm: stopping-time-exp-decay}.  
Additionally, we quantify the discrepancy between the SPDE densities $(\rho,j)$ and the associated particle dynamics in Theorem \ref{thm: error}. 
\end{inftheorem}

\begin{inftheorem}[SPDE flocking for independent noise case, $d=1$]\label{thm_informal_iid}
In dimension $d=1$, consider the independent noise \eqref{indep_micr_noise}. The associated SPDE \eqref{skeleton_spde} (complemented by a kinetic term $k$) is 
\begin{align*}
    d \rho & = - \nabla \cdot j dt ,\\
    d j & = \left\{\rho (a\ast j)  - j (a \ast \rho) - \overline{v}^2(t) \nabla \rho\right\}dt + \sigma N^{-1/2}\sqrt{k}\,d\xi
,\\
    d k & =
    2(j-\overline{v}(t)\rho)a\ast (j-\overline{v}(t)\rho)dt + \left(\sigma^2(1-2N^{-1})- 2 a\ast \rho\right)k dt - \overline{v}(t)\nabla kdt, 
\end{align*}
where $\xi$ is a coloured noise with $L^2(\mathbb{T})$-trace $\Tr(\xi)$, and where the associated noise scaling is $\mathcal{S} = \sigma N^{-1/2}\sqrt{\Tr(\xi)}, \sav = \sigma N^{-1/2}$. 
In the constant kernel case $a=C_a$, upon suitable smallness assumptions on the initial data $\rho_0$, $j_0$, and $k_0$, and of the noise intensity $\mathcal{S} = \sigma N^{-1/2}\sqrt{\Tr(\xi)}$, we prove well-posedness of the systems (Theorem \ref{thm:global_H1}) and get the bounds \eqref{step3_flocking} and \eqref{exp_decaying_bound_prob_tau} with $\kappa = \delta C_a$, for some small $\delta >0$ (see Theorem \ref{thm: constant-kernel-flocking-iid}). An analogous flocking result can be produced for the general kernel case (see Remark \ref{thm: general-kernel-flocking-iid-exp}), although in a discretised setting.
\end{inftheorem}

\begin{remark}\label{remark_about_multi_d}
In this paper, we conduct the full rigorous analysis in dimension $d=1$. We expect the results to be generalised to $d>1$ (see the SPDEs \eqref{eq:spde-vanishing-common-nd} and \eqref{reduced_CS_rho_iid_B_gen_multidim}--\eqref{reduced_CS_k_iid_B_gen-multidim}) with some minor modelling modifications: these modifications have the goal of mitigating the loss of integration-by-parts formulas which can only be used in $d=1$, and typically involve including diffusivity mechanisms in the reduced SPDEs. However, these diffusive mechanisms only kick-in in extremely singular (and  experimentally rarely observed) settings, see \cite{zimper2025reduced} for a similar discussion. This observation reinforces the use of our reduced SPDE models also in $d>1$: in particular, we conduct extensive simulations in the case $d=2$. More generally, we view the formalisation of continuous SPDE flocking provided in this paper to be a baseline upon which to build more tailored and custom-designed methods.  
\end{remark}

\subsection{Literature review}\label{subsec_lit_review}

This work falls under the broad label of collective motion, a 
ubiquitous phenomenon observed in interacting particle systems in which particles self-organise from an unordered to an ordered state. See, for example, the review paper \cite{vicsek2012collective} for an introduction to this topic. The basis of our work is the Cucker--Smale (CS) particle system \cite{cucker2007emergent, cucker2007mathematics} where particles adjust their velocities by averaging their velocity differences with all others, weighted by a communication kernel dependent on their relative positions. This is a well-studied model, with various analytical results characterizing the conditions under which flocking occurs and the rate at which it emerges, depending on factors such as the system parameters and initial configurations \cite{cucker2007emergent, cucker2007mathematics, ha2008particle, ha2009simple, ahn2012collision}.

In many realistic settings, however, purely deterministic dynamics are insufficient to capture the behaviour observed in real life. For example, particles may experience environmental noise, have incomplete information or there may be an intrinsic randomness in their actions. This has led to the introduction of stochastic models for collective motion, notably stochastic Cucker--Smale (sCS) models where randomness is (typically) injected into the velocity dynamics of the particles. Once again, there is a range of quantitative flocking results for sCS models \cite{cucker2008flocking, ahn2010stochastic, ton2014flocking, sun2015positive, cattiaux2018stochastic} where one also has to account for different notions of alignment, for example pathwise or in expectation, among others.

Simulating interacting particle systems becomes computationally expensive for large numbers of particles, a difficulty that is exacerbated in stochastic systems where multiple realisations are required to accurately estimate statistical observables. To overcome these problems, the mean-field limit (MFL) is commonly used to provide a macroscopic description of the system. For deterministic CS models, the MFL is a kinetic PDE and this is a well-studied object, see \cite{ha2008particle, ha2009simple}. In the presence of noise, the MFL can be deterministic or stochastic. Specifically, independent noises on each particle at the microscopic scale lead to deterministic MFLs since, roughly speaking, the noise will get averaged out. Conversely, a common environmental noise in the particle system will persist in the limit so the MFL is a stochastic PDE. For quantitative results regarding the derivation of such MFLs and associated flocking results, see \cite{bolley2011stochastic, choi2017cucker, pedeches2018asymptotic, rosello2019weak}.

For large particle systems, direct simulation of the particles can be computationally expensive, since evaluating all pairwise interactions typically requires  $O(N^2)$ operations per time step. In some settings, solving the corresponding mean-field equation can offer a more efficient alternative, where the actual cost depends on the numerical method and the dimension of the phase space. In particular, traditional MFL models are still $(2d+1)$-dimensional (S)PDEs in $d$-dimensional space ($d$ position variables, $d$ velocity variables and time) leading to high computational costs in $d=2,3$. This issue has motivated the derivation, simulation and analysis of reduced models \cite{carrillo2010asymptotic, carrillo2010particle, tadmor2014critical, choi2017emergent} which, roughly speaking, trade some accuracy in exchange for computational savings while still exhibiting flocking. Specifically, these models are often obtained via the MFL. That is, the key quantities in the reduced model (e.g. local mass, momentum and energy densities) are derived from the MFL. The resulting hydrodynamic equations in these quantities are not closed (due to the reduction) but can be closed under suitable assumptions at the microscopic level, see for example, the mono-kinetic approach  \cite{carrillo2010particle}. In the stochastic setting, however, the situation regarding reduced models is far less developed. In \cite{zimper2025reduced} a reduced SPDE model for stochastic flocking dynamics in the presence of environmental noise was introduced and simulated, however the authors did not perform a theoretical analysis. Moreover, to the best of our knowledge the present work is the first to consider a reduced continuum model for stochastic flocking in the presence of an i.i.d noise.

In contrast to model reduction performed at the level of the MFL, it is becoming increasingly common to derive reduced models directly from appropriate empirical measures defined at the level of the particles. The models studied in this paper are of the latter type and build upon the reduced PDE framework developed for the deterministic CS setting in \cite{zimper2025reduced}. More generally, our work is closely related to the rigorous framework for reduced inertial models of fluctuating hydrodynamics introduced in \cite{cornalba2019regularized, cornalba2020weakly, cornalba2021well}. This reduction technique is also being applied in other contexts that fall into the category of collective motion. For example, in \cite{zimper2026clustering} the authors use reduced SPDE models to study clustering in opinion dynamics.

It is also worth noting that our work is naturally connected to the Dean--Kawasaki equation \cite{Dean1996, kawasaki1994stochastic}, a formal SPDE for the empirical measure of overdamped Langevin particles. In recent years, it has become popular to use Dean--Kawasaki type SPDEs as (full) continuum models to study features of interacting particles systems. Applications include agent-based dynamics \cite{helfmann2021interacting}, opinion dynamics \cite{djurdjevac2022feedback}, and clustering \cite{wehlitz2025approximating} to name a few.

\subsection{Notation}

We work on the torus $\mathbb{T}^d$ and omit the domain when specifying function spaces. The $L^2$ norm will be denoted by $\norm{\cdot}$. We equip $L^2 \times L^2$ with the weighted norm $\norm{(\cdot, \cdot)}$ defined via 
$$
\norm{(f, g)}^2 \coloneqq \vbar^2(t)\norm{f}^2 + \norm{g}^2.
$$
The weight $\vbar(t)$ is the alignment velocity of the particle system being studied and is defined in \eqref{eq: vbar-definition}. 
The weighted norm is time-dependent (in the i.i.d.\ case) and requires $\vbar(t) \neq 0$, for all $t$  to be a norm (and not a semi-norm).

Throughout, we fix a probability space $(\Omega, \mathcal{F}, \mathbb{P})$ with a complete and right continuous filtration $(\mathcal{F}_t)_{t \geq 0}$, and we denote expectation by $\mathbb{E}$, and $L^p$ spaces over $\Omega$ by $\mathbb{L}^p$. 
\section{The reduced SPDE model for common noise}\label{sec: models}

\subsection{Standard Cucker--Smale model}

The Cucker--Smale system \cite{cucker2007emergent, cucker2007mathematics} is a well-established model for flocking dynamics that describes how particles interact with each other in order to align their movements. Each particle is represented by a position $\bx_i \in \mathbb{T}^d$ and a velocity $\bv_i \in \mathbb{R}^d$. These evolve according to the following dynamics
\begin{equation}\label{eq: standard-cs}
\begin{cases}
    \dd \bx_i(t) = \bv_i(t) \dd t, \\
    \dd \bv_i(t) = N^{-1}\sum_{j=1}^N a(\bx_i(t) - \bx_j(t))(\bv_j(t) - \bv_i(t)) \dd t,
\end{cases}
\end{equation}
where $a$ is the interaction kernel representing how particles communicate with each other based on their relative distance. Throughout, we work with a standard choice of interaction kernel given by
\begin{equation}\label{eq: kernel-definition}
a(\bx_i - \bx_j) \coloneqq \frac{\lambda}{(1+ \norm{\bx_i - \bx_j}^2)^r},
\end{equation}
for some $\lambda, r \geq 0$. Define the average velocity $\bvbar(t)$ by 
\begin{equation}\label{eq: vbar-definition}
\bvbar(t) \coloneqq N^{-1}\sum_{i=1}^N \bv_i(t).
\end{equation}
By symmetry of the interaction kernel $a$, it is clear that $\bvbar(t)$ is conserved under the dynamics of \eqref{eq: standard-cs} therefore we write $\bvbar \coloneqq \bvbar(0)$. Flocking is then defined as the alignment of the individual velocities $\bv_i(t)$ with $\bvbar$, as well as group formation 
as made precise by the following definition:
\begin{definition}\label{def: std-flocking}
    We say a solution $\{(\bx_i, \bv_i)\}_{i=1}^N$ to \eqref{eq: standard-cs} is flocking if
    $$
    \lim_{t \rightarrow \infty}\norm{\bv_i - \bvbar} = 0, \quad \text{and} \quad         \sup_{t \geq 0} \norm{\bx_i(t) - \bx_j(t)} < +\infty
    $$
    for all $i, j = 1, \ldots, N$. That is, we have velocity alignment and group formation.
\end{definition}

\begin{remark}\label{rmk: group-formation-1}
    Note that the group formation requirement exists since, classically, the problem is posed for positions in $\mathbb{R}^d$. For simplicity in this paper, we work on the torus where this condition is automatically satisfied.
\end{remark}

It should also be observed that the flocked state is invariant under the dynamics of the system, that is, once the particles are aligned, they stay aligned (see \eqref{eq: standard-cs}).
We will see that particular choices of sCS systems have this property and we will design our reduced SPDE model to preserve this. It is well-known that $\eqref{eq: standard-cs}$ flocks unconditionally provided $r \leq 1/2$ \cite{ha2008particle}, and for other values of $r$ it can be shown to flock under restrictions on the initial configuration (for example, \cite{cucker2007mathematics}). Despite \eqref{eq: standard-cs} being a prototypical model for collective motion with desirable theoretical guarantees, its applicability to real-world systems is limited by its deterministic nature motivating the introduction of stochastic variants.

\subsection{Stochastic Cucker--Smale model}\label{subsec: stochastic-cs-particle}

We now introduce a stochastic variant of the CS model \eqref{eq: standard-cs}. Consider perturbing the interaction term between particles by a common white noise. This models an environmental influence that affects the communication between all particles equally and doesn't depend on space. Under this perturbation, the velocity dynamics become
$$
\dd \bv_i(t) = N^{-1}\sum_{j=1}^N a(\bx_i(t) - \bx_j(t))(\bv_j(t) - \bv_i(t)) \dd t - \sigma N^{-1} \sum_{j=1}^N (\bv_j(t) - \bv_i(t)) \dd \bB(t)
$$
where $\sigma$ is a constant, and $\bB(t)$ is a standard $d$-dimensional Brownian motion adapted to the filtration $(\mathcal{F}_t)_{t \geq 0}$.  
Using the definition of $\bvbar(t)$, the noise term can be separated from the deterministic component leading to the following dynamics:
\begin{equation}\label{eq: stochastic-cs-vc}
\begin{cases}
    \dd \bx_i(t) = \bv_i(t) \dd t, \\
    \dd \bv_i(t) = N^{-1}\sum_{j=1}^N a(\bx_i(t) - \bx_j(t))(\bv_j(t) - \bv_i(t)) \dd t + \sigma(\bv_i(t) - \bvbar(t)) \dd \bB(t).
\end{cases}
\end{equation}
This model is (up to notational differences) precisely the sCS model introduced by \cite{ton2014flocking}
and later studied in the review paper \cite[Equation 1.8]{cattiaux2018stochastic}.
Crucially, observe that symmetry of the interaction kernel $a$, and the structure of the noise means $\bvbar(t)$ is conserved under the dynamics, as with the CS model \eqref{eq: standard-cs}: this can be easily seen from the identity $d \overline{\bv}(t) = \frac{1}{N^2}\sum_{i=1}^{N} \sum_{j=1}^{N} a(\bx_i-\bx_j)(\bv_j-\bv_i)dt +\frac{\sigma}{N}\sum_{i=1}^{N} (\bv_i - \overline{\bv})d\bB_t = 0$. 
We therefore write $\bvbar \coloneqq \bvbar(0)$. 
We can now define flocking for \eqref{eq: stochastic-cs-vc} as alignment of the individual velocities with the average initial velocity. The stochastic nature of the system means there is no single way to define such a property but we consider two natural choices used in \cite{cattiaux2018stochastic}:
\begin{definition}\label{def: flock-as}
    We say a solution $\{(\bx_i, \bv_i)\}_{i=1}^N$ to \eqref{eq: stochastic-cs-vc} is flocking almost surely if
    $$
    \lim_{t \rightarrow \infty}\norm{\bv_i - \bvbar} = 0, \quad \text{and} \quad         \sup_{t \geq 0} \norm{\bx_i(t) - \bx_j(t)} < +\infty,
    $$
    holds $\mathbb{P}$-a.s. for all $i, j = 1, \ldots, N$. 
\end{definition}

\begin{definition}\label{def: flock-expect}
    We say a solution $\{(\bx_i, \bv_i)\}_{i=1}^N$ to \eqref{eq: stochastic-cs-vc} is flocking in $\mathbb{L}^{p, q}$ ($p,q\geq 1$) if
    $$
    \lim_{t \rightarrow \infty}\E{\norm{\bv_i - \bvbar}^p} = 0, \quad \text{and} \quad \sup_{t \geq 0} \E{\norm{\bx_i(t) - \bx_j(t)}^q} < +\infty.
    $$
\end{definition}

Observe that in the setting of a conserved flocking velocity, these are natural extensions of the deterministic definition given in Definition \ref{def: std-flocking}.

In \cite{cattiaux2018stochastic}, flocking results are obtained for \eqref{eq: stochastic-cs-vc} both in the simple setting of a constant interaction term, and in the general case. When the interaction term is constant, that is, $a(\bx_i - \bx_j) \equiv C_a \geq 0$ for all $i, j$, they prove that the system is always almost surely flocking and flocking in $L^{2, 1}$ if and only if $\sigma^2 < 2C_a$. For details, see \cite[Theorem 2.1]{cattiaux2018stochastic}. For the general kernel, it is sufficient to consider their results that assume the kernel $a$ is bounded from below since this incorporates our $\btd$ setting. In this case they once again obtain almost sure flocking for any $\sigma>0$ and $\mathbb{L}^{2,2}$ flocking provided $\sigma^2 < 2a_{\mathrm{min}}$ where $a_{\mathrm{min}} \coloneqq \inf_{i, j} a(\bx_i - \bx_j)$. Full details with appropriate adjustment can be found in \cite[Theorem 6.3]{cattiaux2018stochastic}. 

\subsection{Derivation of the reduced SPDE}\label{subsec: derivation-env-noise}

In order to derive a reduced SPDE, consider the empirical mass density $\rho$ \eqref{rho} and empirical momentum density $\bj$ of the particle system \eqref{j}, which we recall here for simplicity:
\begin{align*}
\rho(\bx, t) \coloneqq N^{-1}\sum_{i=1}^N \delta(\bx - \bx_i(t)), \quad \bj(\bx, t) \coloneqq N^{-1}\sum_{i=1}^N \bv_i(t) \delta(\bx - \bx_i(t)).
\end{align*}
We give a sketch derivation with full details given in \cite{zimper2025reduced} and references therein. A formal application of It\^o's formula gives
$$
\dd \rho(\bx, t) = -N^{-1}\sum_{i=1}^N \partial_{x_k}\delta(\bx - \bx_i(t))v_{i, k}(t) = - \nabla \cdot \bj(\bx, t) \dd t.
$$
Similarly, for the $k$th component of $\bj$ we get
$$
\dd j_k(\bx, t) = N^{-1}\sum_{i=1}^N \dd v_{i, k}\delta(\bx - \bx_i(t)) + N^{-1}\sum_{i=1}^N v_{i, k} \dd \delta(\bx - \bx_i(t)) := T_1 + T_2.
$$
To deal with $T_1$, we insert the definition of $\dd v_{i, k}$ which gives
\begin{multline*}
T_1 = N^{-1}\sum_{i=1}^N \delta(\bx - \bx_i(t)) N^{-1}\sum_{j=1}^N a(\bx_i(t) - \bx_j(t))(v_{j, k}(t) - v_{i, k}(t)) \dd t \\ + N^{-1}\sum_{i=1}^N \sigma (v_{i, k}(t) - \vbar_k)\delta(\bx - \bx_i(t)) \dd B_k(t).
\end{multline*}
The second term is equal to $\sigma(j_k-\vbar_k \rho) \dd B_k(t)$ by definition. Using properties of the Dirac delta, it can be shown that the first term is $\rho(a*j_k) - j_k(a*\rho)$ (see \cite{zimper2025reduced}). For $T_2$, we need to make an approximation since it is not closable in the densities $\rho$ and $\bj$. Specifically, we make an equilibrium assumption that we are in the flocking regime, that is, $\bv_i(t) \equiv \bvbar(t)$ for all $i\in \{1,\dots,N\}$ and $t$. It is then possible to write
$$
T_2 = 
\begin{cases}
    -\vbar^2 \partial_x \rho(x, t) \dd t, \quad &d = 1, \\
    -\bvbar \nabla \cdot \bj(\bx, t) \dd t, \quad &d>1.
\end{cases}
$$
Again, we refer to \cite{zimper2025reduced} and references therein for full details of the model reduction and closure but mention that the specific choice in $d=1$ allows for stronger analytical results. The key takeaway is that we use an equilibrium assumption, and so expect our model to perform better the closer we are to the flocking regime. Overall, we have the reduced SPDEs
\begin{enumerate}
    \item $d=1$
    \begin{equation}\label{eq:spde-vanishing-common-1d}
    \begin{cases}
        \dd \rho = -\partial_x j \dd t, \\
        \dd j = [\rho(a*j) - j(a*\rho) - \vbar^2 \partial_x \rho] \dd t + \sigma(j-\vbar \rho) \dd B_t,
    \end{cases}
    \end{equation}
    \item $d>1$
    \begin{equation}\label{eq:spde-vanishing-common-nd}
    \begin{cases}
        \dd \rho = -\nabla \cdot \bj \dd t, \\
        \dd \bj = [\rho(a*\bj) - \bj(a*\rho) - \bvbar \nabla \cdot \bj] \dd t + \sigma(\bj-\bvbar \rho) \dd \bB_t.
    \end{cases}
    \end{equation}
\end{enumerate}

For the analytical considerations of this section we restrict to $d=1$ (i.e., we study \eqref{eq:spde-vanishing-common-1d}), and drop the boldface lettering for clarity. However, numerical simulations providing evidence of flocking will be given in dimensions $d=1,2$. This allows us to avoid theoretical complications arising due to the integration-by-parts formulas, as already mentioned in Remark \ref{remark_about_multi_d}.

\subsection{Choice of initial condition and conservation of mass}

The empirical mass and momentum densities \eqref{rho}--\eqref{j} have masses $1$ and $\vbar$ respectively, and these masses are conserved for all times under the particle dynamics. It is therefore natural to require the same for our reduced SPDE model. We do this by choosing initial data $\rho_0$ and $j_0$ with masses $1$ and $\vbar$ respectively. Moreover, the symmetry of the interaction kernel $a$, and the noise structure means these masses are conserved for all times. We prove this in the following proposition:

\begin{proposition}\label{prop: conservation-of-mass}
    Consider the reduced SPDE system \eqref{eq:spde-vanishing-common-1d} with initial data $(\rho_0, j_0)$ such that
    $$
    \int_\bt \rho_0(x) \dd x = 1, \quad \int_\bt j_0(x) \dd x = \vbar.
    $$
    Then the total mass and total momentum are conserved pathwise, i.e.,
    $$
    \int_\bt \rho(x, t) \dd x = 1, \quad \int_\bt j(x, t) \dd x = \vbar, \quad t \geq 0.
    $$
\end{proposition}

\begin{proof}
    Define the masses
    $$
    m_\rho(t) \coloneqq \int_\bt \rho(x, t) \dd x, \quad m_j(t) \coloneqq \int_\bt j(x, t) \dd x.
    $$
    We in fact prove the more general result that if $m_j(0) = \vbar m_\rho(0)$, we have conservation of mass. First compute
    $$
    \dd m_\rho(t) = \int_\bt \dd \rho(x, t) \dd x = - \int_\bt \partial_x j(x, t) \dd x \dd t = 0,
    $$
    by integration by parts. Therefore, $m_\rho(t) = m_\rho(0)$. Similarly, we compute
    $$
    \dd m_j(t) = \int_\bt \left\{[\rho(a * j) - j(a * \rho)](x, t) - \vbar^2 \partial_x \rho(x, t) \right\}\dd x \dd t + \sigma \dd B(t) \int_\bt j(x, t) - \vbar \rho(x, t) \dd x.
    $$
    The integral of the convolution terms vanishes using Fubini's Theorem and the symmetry of $a$, whilst $\int_\bt \partial_x \rho$ vanishes using integration by parts. Moreover, conservation of mass for $\rho$ means the expression reduces to the SDE
    $$
    \dd m_j(t) = \sigma (m_j(t) - \vbar m_\rho(0))\dd B(t),
    $$
    with initial condition $m_j(0)$. Define the shifted process $y(t) \coloneqq m_j(t) - \vbar m_\rho(0)$. Then $y(t)$ satisfies
    $$
    \dd y(t) = \sigma y(t) \dd B(t), \quad y(0) = m_j(0) - \vbar m_\rho(0).
    $$
    This has solution
    $$
    y(t) = y(0)\exp \brack{\sigma B(t) - \frac{1}{2}\sigma^2 t}, \quad t \geq 0
    $$
    so that
    $$
    m_j(t) = (m_j(0) - \vbar m_\rho(0))\exp \brack{\sigma B(t) - \frac{1}{2}\sigma^2 t} + \vbar m_\rho(0), \quad t \geq 0.
    $$
    Thus, choosing $m_j(0) = \vbar m_\rho(0)$ gives conservation of mass for $j$.
\end{proof}

\begin{remark}
The previous result extends to the reduced SPDE system
\eqref{eq:spde-vanishing-common-nd} on $\mathbb{T}^d$, $d>1$.
The proof is obtained by adapting the arguments to the vector-valued setting.
\end{remark}

In order to treat random initial conditions in our analysis, we require the following independence assumption.

\begin{assumption}
\label{ass:IC}
We assume the following:
\begin{enumerate}[label=(A\arabic*)]
    \item \label{ass:idp-ic}
     The initial condition $(\rho_0, j_0)\in L^2(\Omega; H^1\times H^1)$ is $\mathcal{F}_0$-measurable, and is independent of the Brownian motion $B(t)$ for all $t \geq 0$.
    \item \label{ass:p}
     For a given $p$, the initial condition $(\rho_0, j_0)\in L^p(\Omega; H^1\times H^1)$ is $\mathcal{F}_0$-measurable, and is independent of the Brownian motion $B(t)$ for all $t \geq 0$. 
\end{enumerate}
\end{assumption}

\subsection{SPDE flocking results in 1D}\label{sec: main-results}

We first make the notion of SPDE flocking introduced in \eqref{informal_SPDE_flocking_quantity} more precise.

\begin{definition}[SPDE flocking almost surely]
    For \eqref{eq:spde-vanishing-common-1d}, we say SPDE flocking occurs almost surely if
    $$
    \lim_{t \rightarrow \infty} \norm{j(t) - \vbar \rho(t)} = 0,
    $$
    holds $\mathbb{P}$-a.s. 
\end{definition}

\begin{definition}[SPDE flocking in expectation]
    For \eqref{eq:spde-vanishing-common-1d}, we say SPDE flocking occurs in expectation if
    $$
    \lim_{t \rightarrow \infty} \E{\norm{j(t) - \vbar \rho(t)}^2} = 0.
    $$
\end{definition}

These definitions characterise SPDE flocking as the momentum density $j$ aligning with a multiple of the density $\rho$ with the multiple being the particle flocking velocity $\vbar$. Observe that these are both natural considering the physical meanings of $\rho$ and $\bj$, but are also the natural analogues of almost sure flocking and $\mathbb{L}^{2,2}$ flocking (see Definitions \ref{def: flock-as}, \ref{def: flock-expect}) used in \cite{cattiaux2018stochastic}. Moreover, as noted in \cite[Section 2(d)]{zimper2025reduced}, flocking at the particle level implies flocking at the SPDE level in suitable weak metrics.

It will be convenient to decompose the interaction kernel \eqref{eq: kernel-definition} as the sum of a constant part and spatially varying part, that is, 
\begin{equation}\label{eq: decomposed-kernel}
a(x) = C_a + \theta g(x),
\end{equation}
where $C_a > 0$, $g \in L^2$, and $\theta$ is a small parameter. Without loss of generality, we assume $\theta \geq 0$. We stress that the representation \eqref{eq: decomposed-kernel} is not unique, and that the choice of $C_a$, $\theta$, and $g$ will be dictated by suitable smallness assumptions.

Before we proceed with proofs related to flocking of the reduced SPDE system, we give a well-posedness result for \eqref{eq:spde-vanishing-common-1d}. The proof is provided in Appendix \ref{app: well-posedness}.

\begin{restatable}{theorem}{wellposedness}\label{thm: well-posedness}
    Consider \eqref{eq:spde-vanishing-common-1d} with $\mathcal{F}_0$-measurable initial condition $(\rho_0, j_0) \in L^2(\Omega; H^1 \times H^1)$. Then we have the following well-posedness results:
    \begin{enumerate}[label=(\alph*)]
        \item If the interaction kernel is constant (i.e. $a \equiv C_a$), \eqref{eq:spde-vanishing-common-1d} admits a unique, global-in-time $H^1 \times H^1$-valued strong solution. 
        \item Consider the full interaction kernel \eqref{eq: kernel-definition}. Then \eqref{eq:spde-vanishing-common-1d} admits a unique local-in-time $H^1 \times H^1$-valued strong solution. Moreover, with probability one, the $H^1 \times H^1$ norm is finite where the $L^2 \times L^2$ norm is uniformly bounded.
    \end{enumerate}
    In both cases, the solution has a continuous modification where it is defined.  
\end{restatable}

We first state our flocking results when the interaction kernel is constant. This is a physically unrealistic scenario since it says a given particle communicates equally with all other particles, irrespective of how far away they are. However, it is a simpler setting in which we can illustrate some key ideas that carry over to the general setting.

\begin{theorem}\label{thm: constant-kernel-flocking}
    Let Assumption~\ref{ass:IC}\ref{ass:idp-ic}  hold and consider the case of a constant kernel $a \equiv C_a$. Then \eqref{eq:spde-vanishing-common-1d} admits the following global flocking results: 
    \begin{enumerate}[label=(\alph*)]
        \item \textit{almost sure} flocking for any $\sigma > 0$, that is $\mathbb{P}$-a.s.,
        $$
        \lim_{t \rightarrow \infty} \norm{j(t) - \vbar \rho(t)} = 0.
        $$
        \item flocking \textit{in expectation} provided $\sigma^2 < 2 C_a$, that is,
        $$
        \E{\norm{j(t) - \vbar \rho(t)}^2} \leq \exp \sbrack{(\sigma^2 - 2C_a)t}\E{\norm{j_0 - \vbar \rho_0}^2}, \quad t \geq 0.
        $$
    \end{enumerate}
\end{theorem}

This result directly mirrors the one obtained for the particle system in \cite{cattiaux2018stochastic} which we recalled in Section \ref{subsec: stochastic-cs-particle}.  
That is, our reduced SPDE preserves flocking at the continuous level under the same conditions on the noise as for the particle system derived in \cite{cattiaux2018stochastic}.

We now move onto the general kernel where we only have local well-posedness, and first define a stopping time $\tau$ that is used throughout our analysis.
\begin{definition}
    Let $(\rho(t), j(t))$ be a local strong $H^1 \times H^1$ solution of \eqref{eq:spde-vanishing-common-1d}. Define the stopping time $\tau$ by
    \begin{align}
    \tau \coloneqq 
    \begin{cases}\label{defn_stop_time}
        \inf \cbrack{t \geq 0 : \norm{(\rho(t), j(t))} \geq \frac{C_a \vbar}{4\theta \norm{g}}}, \quad &\text{if infimum taken over non-empty set}, \\
        \infty, \quad &\text{otherwise}.
    \end{cases}
    \end{align}
\end{definition}

\begin{remark}
    Statement (b) of Theorem \ref{thm: well-posedness} justifies our use of an $L^2 \times L^2$ stopping time since it ensures our $H^1 \times H^1$-valued solution cannot blow-up on $[0, \tau]$.
\end{remark}

To justify this definition and motivate our main theorems, we now make the key observation that a flocking-type estimate holds up to $\tau$.

\begin{lemma}\label{lem: p=2-flocking-bound}
For the flocking quantity $X(t) = \|j - \overline{v}\rho\|^2$, we have the bound 
    $$
    X(t \wedge \tau) \leq X_0 \exp \brack{(\sigma^2 - C_a)(t \wedge \tau) + 2 \sigma B(t \wedge \tau) - \frac{1}{2}(2\sigma)^2(t \wedge \tau)}.
    $$
\end{lemma}

Notice that this is a stopped solution to a geometric Brownian motion (GBM) and on the event $\{\tau = \infty\}$, standard properties of GBMs mean this estimate implies almost sure flocking for any $\sigma > 0$, and flocking in expectation provided $\sigma^2 < C_a$. These are analogous to the results for the constant case given in Theorem \ref{thm: constant-kernel-flocking}. Driven by this observation, our next results estimate the probability of the event $\{\tau = \infty\}$, and thus quantify how likely we are to obtain global flocking results. The first such theorem says this probability is closer to one, the closer the kernel is to being constant. 

\begin{proposition}\label{thm:p=2-stopping-time-bound}
    Let Assumption~\ref{ass:IC}\ref{ass:idp-ic} hold, let $\sigma^2 < C_a$ and assume that
    $$
    \max \cbrack{1, 2 \frac{\E{\norm{j_0}^2} + \vbar^2 \E{\norm{\rho_0}^2}}{\vbar^2} + 18 \frac{\E{\norm{j_0 - \vbar \rho_0}^2}}{\vbar^2}\brack{\frac{C_a}{C_a - \sigma^2}}^2} < \frac{C_a}{4 \theta \norm{g}}.
    $$
    Then we have the following flocking results:
    \begin{enumerate}[label=(\alph*)]
        \item almost sure flocking. That is $\mathbb{P}$-a.s.,
        $$
        \mathbf{1}_{\{\tau = \infty\}}X(t) \rightarrow 0 \quad \text{as} \quad t \rightarrow \infty.
        $$
        \item flocking in expectation. That is,
        $$\E{\mathbf{1}_{\{\tau = \infty\}}X(t)} \leq \exp \sbrack{(\sigma^2 - C_a)t}\E{\norm{j_0 - \vbar \rho_0}^2}, \quad t \geq 0.
        $$
    \end{enumerate}
    with the stopping time estimate
    $$
    \mathbb{P}\brack{\cbrack{\tau = \infty}} \geq 1 - \frac{4 \theta \norm{g}}{C_a}.
    $$
\end{proposition}

\begin{remark}
    The requirement in Theorem \ref{thm:p=2-stopping-time-bound} is nothing but a smallness assumption detailing how far the interaction kernel can deviate from a constant value. In fact, this is the natural stochastic analogue of the equivalent restriction given in \cite[Proposition 4.2]{zimper2025reduced} to ensure flocking of their deterministic reduced PDE. One can also take the alternative viewpoint and fix a particular interaction kernel. The assumption then tells us how far the initial profile can deviate from a flocked state to ensure flocking occurs. This is  unsurprising given that our reduced model was closed using an equilibrium assumption.
\end{remark}

The result of Proposition \ref{thm:p=2-stopping-time-bound} deviates from associated result for the particle system in two ways: firstly, the noise restriction is different ($\sigma^2 < C_a$ versus $\sigma^2 < 2a_{\mathrm{min}}$) and secondly our result only holds with high probability instead of unconditionally. We believe both discrepancies are technical in nature with numerical experiments in Section \ref{sec: numerics} suggesting comparable noise restrictions for both systems.

Additionally, notice that Proposition \ref{thm:p=2-stopping-time-bound} doesn't see the noise intensity $\sigma$. That is, as we send $\sigma \rightarrow 0$, we don't recover the deterministic result of \cite[Proposition 4.2]{zimper2025reduced} which says the noiseless system flocks globally. This suggests our analysis can be improved and, accordingly, the following theorem is a small noise analysis that does exactly this.

\begin{theorem}\label{thm: stopping-time-exp-decay} 

Let Assumption~\ref{ass:IC}\ref{ass:p} hold and assume there exists a deterministic $R_0 > 0$ such that $\norm{j_0}^2 + \vbar^2 \norm{\rho_0}^2 \leq R_0^2$ holds $\mathbb{P}$-a.s. Define the constants $A \coloneqq R_0$, $B \coloneqq \sqrt{2}R_0$, $D \coloneqq B^{-1}(-A+\vbar C_a/(4\theta \norm{g}))$, and assume that $A+3B < C_a \vbar / (4 \theta \norm{g})$. Choose $\alpha \in (0, (D-3)/(D+\sqrt{2}))$ and $2\sigma^2 < \alpha C_a$. Finally, define a function $g(\alpha) < 1$ given by
$$
g(\alpha) \coloneqq \frac{4\theta \norm{g}}{C_a\vbar} \brack{A+B \frac{3+\sqrt{2}\alpha}{1-\alpha}}.
$$
Then, we have the same almost sure flocking and flocking in expectation results as in Proposition \ref{thm:p=2-stopping-time-bound} but with the following exponentially decaying stopping time estimate
$$
\mathbb{P}(\{\tau = \infty\}) \geq 1 - \exp \brack{- \frac{\alpha C_a \absv{\log g(\alpha)}}{\sigma^2}}.
$$
\end{theorem}

\begin{remark}
    While the statement of Theorem \ref{thm: stopping-time-exp-decay} appears complicated, the key point is that we require a smallness assumption controlling how much the interaction kernel can deviate from a constant value. This is analogous to the smallness assumptions in Theorem \ref{thm:p=2-stopping-time-bound} and \cite[Proposition 4.2]{zimper2025reduced}.
\end{remark}

\subsection{Error analysis in 1D}

In \cite{zimper2025reduced}, they are able to prove a uniform-in-time error result between the reduced PDE and Cucker--Smale particle system measured in a weighted $\boldsymbol{H}^{-2}$ norm given by $\norm{(f, g)}_{\boldsymbol{H}^{-2}} \coloneqq \vbar^2 \norm{f}^2_{H^{-2}} + \norm{g}^2_{H^{-2}}$. Theorem \ref{thm: error} extends this result to the common noise setting. It compares $X_{\mathrm{PDE}} \coloneqq (\rho, j)$ solving \eqref{eq:spde-vanishing-common-1d} with $X_{\mathrm{CS}} \coloneqq (\rho_\ve, j_\ve)$ where the subscript $\ve$ denotes mollification of the particle empirical measure with a von Mises kernel (see Section \ref{sec: numerics} for the rigorous definition).

\begin{remark}
    The error is measured in a negative Sobolev norm to avoid using diverging-in-time propagation of chaos results. See \cite[Remark 3.3]{zimper2025reduced} for further details. In Section \ref{sec: numerics} we present experiments measuring the error in the $L^2$ norm.
\end{remark}

\begin{theorem}\label{thm: error}
    Let $d=1$ and fix $\ve > 0$. Define $\xpde \coloneqq (\rho, j)$ and $\xcs \coloneqq (\rho_\ve, j_\ve)$ where $v_i(0) \in L^2(\Omega; \mathbb{R})$ for all $i=1, \ldots, N$ and $(\rho_0, j_0) \in L^2(\Omega; H^{-2} \times H^{-2})$. Assume both are $\mathcal{F}_0$-measurable and independent of the Brownian motion $B(t)$. Assume the interaction potential $a$ has bounded first derivative and that the Fourier modes $\hat{a}$ satisfy
    $$
    \absv{\hat{a}^{(r)}(\xi)} \lesssim (1 + \absv{\xi}^2)^{-1}, \quad r=0,1,2~ \forall \xi \in \mathbb{R}.
    $$
    Further assume that $\sigma^2 < \min\{C_a, 2a_{\mathrm{min}}\}$ and define $R_v^2 \coloneqq \max_{1 \leq i \leq N} \E{(v_i(0) - \vbar)^2}$. Finally, let the assumptions of Theorem \ref{thm: stopping-time-exp-decay} hold so that $\mathbb{P}(\tau < \infty)$ is exponentially small in the size of the noise. Then we have the following error bound between the particles and SPDE:
    \begin{multline*}
    \E{\sup_{t \geq 0}\norm{\xpde(t \wedge \tau) - \xcs(t \wedge \tau)}_{\hmtwo}} \lesssim \E{\norm{\xpde(0) - \xcs(0)}_{\hmtwo}} + \frac{2\E{\norm{j_0 - \vbar \rho_0}^2}^{1/2}}{C_a -  \sigma^2}\\
    + \frac{2R_v}{2a_{\mathrm{min}}-\sigma^2} + C(\sigma) \brack{\frac{\E{\norm{j_0-\vbar \rho_0}^2}}{C_a - \sigma^2} + \frac{R_v^2}{2a_{\mathrm{min}}-\sigma^2}}^{1/2} + \frac{2R_v(1+C(a)(\ve + \sqrt{\ve}))}{2a_{\mathrm{min}}-\sigma^2} \\
    + C(\vbar)^2 \brack{1+\frac{C_a}{4 \theta \norm{g}}}^2 \frac{\E{\norm{j_0 - \vbar \rho_0}^2}}{(C_a-\sigma^2)^2}.
    \end{multline*}
\end{theorem}

\begin{proof}
    See Appendix \ref{app: error}.
\end{proof}

\begin{remark}
    While Theorem \ref{thm: error} has a complicated appearance, its interpretation is natural: it provides a uniform-in-time error result between the particles and SPDE when both systems are flocking in expectation. The error depends on the discrepancy in initial conditions, the initial flocking discrepancy in the SPDE, the initial alignment discrepancy of the velocities, how far the kernel deviates from being constant, how small the kernel gets (i.e. a measure of localisation given we are on a bounded domain), and the width of the von Mises kernel. Moreover, suppose we are initially flocking, that is, $v_i(0) \equiv \vbar$ for all $1 \leq i \leq N$. Then $R_v = 0$ and if we initialise the SPDE such that $\xpde(0) \equiv \xcs(0)$ we get $j_0 = \vbar \rho_0$. In particular, all terms in the error estimate vanish which reflects that our SPDE is exact in the flocking regime.
\end{remark}

\section{The reduced SPDE model for i.i.d. noise}

\subsection{Derivation}\label{subsec_derivation_SPDE_iid}

For clarity, we provide full details for the derivation of the one-dimensional reduced SPDE model and state later the reduced model in arbitrary dimension. Set  
$$
d\mathcal{N}_i = \sigma(v_i - \overline{v})dB_i,
$$ 
in \eqref{general_part_sys}, where $\{B_i\}_{i=1}^{N}$ are independent Brownian motions and sigma is some constant. This is exactly as in \cite[Equation (1.9)]{cattiaux2018stochastic}. Deriving an SPDE in this case takes more effort, and requires some approximations. We now detail this derivation.

\emph{$\bullet$ Dynamics of $\overline{v}(t)$.} 

The mean velocity $\overline{v}(t) = N^{-1}\sum_{i=1}^{N}{v_i(t)}$ is no longer preserved in time, as 
\begin{align}\label{mean_vel}
    d\overline{v}(t) = \sigma N^{-1}\sum_{i=1}^{N}{(v_i-\overline{v}(t))dB_i}.
\end{align}

\emph{$\bullet$ Deriving equations for $\rho$ and $j$.}

 A formal application of the chain rule on \eqref{rho}--\eqref{j} gives us
\begin{align*}
    d \rho(x,t) & = -N^{-1}\sum_{i=1}^{N}{\delta'(x-x_i(t))v_i(t)d t} = - \nabla \cdot j dt
\end{align*}
 where the symbol $'$ indicates differentiation in space, and 
 \begin{align}\label{Evolution_J}
     d j(x,t) & = N^{-1}\sum_{i=1}^{N}{dv_i(t)\delta(x-x_i(t))} + N^{-1}\sum_{i=1}^{N}{v_i(t)d\delta(x-x_i(t))}  \nonumber\\
     & = N^{-1}\sum_{i=1}^{N}{dv_i(t)\delta(x-x_i(t))} - N^{-1}\sum_{i=1}^{N}{v_i^2(t)\delta'(x-x_i(t))}dt \nonumber\\
     & = N^{-1}\sum_{i=1}^{N} \left(N^{-1}\sum_{j=1}^{N}{a(x_i-x_j)(v_j-v_i)}\right) \delta(x-x_i(t)) dt \nonumber\\
     & \quad\quad - N^{-1}\sum_{i=1}^{N}{v_i^2(t)\delta'(x-x_i(t))}dt + \sigma N^{-1}\sum_{i=1}^{N}{(v_i-\overline{v}(t))\delta(x-x_i(t))dB_i}
     \nonumber\\
     & =: T_1dt + T_2dt + T_3.
 \end{align}

 \emph{$\bullet$ Closing $T_1$ and $T_2$ from \eqref{Evolution_J}.}
 
$T_1$ and $T_2$ are dealt with similarly to \cite{zimper2025reduced}, but now using $\overline{v}(t)$, so $T_1 = \rho (a\ast j)  - j (a \ast \rho)$, $T_2 \approx - \overline{v}(t)^2 \nabla \rho$.

\emph{$\bullet$ Closing $T_3$ from \eqref{Evolution_J}.}
As for the noise $T_3$, we compute its cross-variation structure (indicated by the symbol $\llangle \cdot, \cdot \rrangle$) by taking two separate points $x_1,x_2$ and computing
\begin{align}\label{cross_variation_T3}
\langle T_3,T_3\rangle & = \left\llangle \sigma N^{-1}\sum_{i=1}^{N}{(v_i-\overline{v}(t))\delta(x_1-x_i(t))dB_i}, \sigma N^{-1}\sum_{i=1}^{N}{(v_i-\overline{v}(t))\delta(x_2-x_i(t))dB_i}\right\rrangle \nonumber\\
& \stackrel{iid. \{B_i\}_{i=1}^{N}}{=} \sigma^2 N^{-1} \left[N^{-1}\sum_{i=1}^{N}{(v_i-\overline{v}(t))^2\delta(x_1-x_i(t))}\right]\delta(x_1-x_2).
\end{align}
If we define 
\begin{align}
    k(x,t) := N^{-1}\sum_{i=1}^{N}{(v_i-\overline{v}(t))^2\delta(x-x_i(t))},
\end{align}
then the process
\begin{align}\label{noise_SPDE_ind_B}
    \sigma N^{-1/2}\sqrt{k}\,d\xi,
\end{align}
where $\xi$ is a space-time white noise, has formally the same cross-variation as \eqref{cross_variation_T3}. The computations from \eqref{cross_variation_T3}--\eqref{noise_SPDE_ind_B} have taken inspiration from \cite[Section 3.2]{cornalba2019regularized}.

\emph{$\bullet$ Hierarchy on $j$, differential of $k$.}

As $k$ is not closable in $\rho$ and $j$, we need to approximate it.
It\^o calculus gives
\begin{align}\label{differential_k}
    & dN^{-1}\sum_{i=1}^{N}{(v_i-\overline{v}(t))^2\delta(x-x_i(t))} \nonumber\\
    & \quad  = N^{-1}\sum_{i=1}^{N}{d(v_i-\overline{v}(t))^2\delta(x-x_i(t))} + N^{-1}\sum_{i=1}^{N}{(v_i-\overline{v}(t))^2d\delta(x-x_i(t))} \nonumber\\
    &  \quad = N^{-1}\sum_{i=1}^{N}{2(v_i-\overline{v}(t))\left[N^{-1}\sum_{j=1}^{N}{a(x_i-x_j)(v_j-v_i)} dt + \sigma(v_i-\overline{v}(t))dB_i -d\overline{v}(t)\right]\delta(x-x_i(t))} \nonumber\\
    &  \quad \qquad + \frac{1}{2}N^{-1}\sum_{i=1}^{N}{2\left[ d(v_i-\overline{v}(t))\right]^2}\delta(x-x_i(t))  \nonumber \\
    &  \quad \qquad - N^{-1}\sum_{i=1}^{N}{(v_i-\overline{v}(t))^2v_i\delta'(x-x_i(t))}dt.
    \end{align}
    Furthermore, using the independence of the $(B_i)_{i=1}^{N}$, we have
\begin{align}\label{differential_diff_velocities}
        \llangle d(v_i-\overline{v}(t))\rrangle & = \left\llangle \sigma (v_i -\overline{v}(t))dB_i - \sigma N^{-1}\sum_{j=1}^{N}{(v_j-\overline{v}(t))dB_j} \right\rrangle \nonumber \\
        & = \left\llangle \sigma \left (1 - \frac{1}{N}\right) (v_i -\overline{v}(t))dB_i - \sigma N^{-1}\sum_{j\neq i}{(v_j-\overline{v}(t))dB_j} \right\rrangle \nonumber\\
        & =  \sigma^2 \left (1-\frac{1}{N} \right )^2 (v_i - \overline{v}(t))^2 \dd t + \frac{\sigma^2}{N^2}\sum_{j \neq i} (v_j - \overline{v}(t))^2 \dd t \nonumber\\
        & = \sigma^2 \left ( 1-\frac{2}{N}\right)(v_i - \overline{v}(t))^2 \dd t + \frac{\sigma^2}{N^2} \sum_{j=1}^N (v_j - \overline{v}(t))^2\dd t.
\end{align}
We therefore get
\begin{align*}
    & N^{-1}\sum_{i=1}^N \left[ d(v_i-\overline{v}(t))\right]^2 \delta(x - x_i(t)) \\
    & \quad \quad = \sigma^2 \left (1-\frac{2}{N} \right) k \dd t  + N^{-1}\sum_{i=1}^N \delta(x-x_i(t)) \frac{\sigma^2}{N^2}\sum_{j=1}^N (v_j - \overline{v}(t))^2 \dd t.
\end{align*}
In the second term, the sum over $i$ is only with respect to space so we recover $\rho$. The final sum above is precisely the integral in space of $k$. 
Therefore,
\begin{align}\label{Y_3_correct}
N^{-1}\sum_{i=1}^N \left[ d(v_i-\overline{v}(t))\right]^2 \delta(x - x_i(t))  = \sigma^2 \left (1-\frac{2}{N} \right) k \dd t + \frac{\sigma^2}{N} \rho \left(\int_x k(x)\right ) \dd t. 
\end{align}

Next we combine \eqref{Y_3_correct} and \eqref{differential_diff_velocities}, dropping only the non-linear, lower-order contribution $\frac{\sigma^2}{N} \rho \left(\int_x k(x)\right )$ in \eqref{Y_3_correct}, as well as suitably adding and subtracting $\overline{v}(t)$, we obtain
\begin{align}\label{differential_k_gen}
    & dN^{-1}\sum_{i=1}^{N}{(v_i-\overline{v})^2\delta(x-x_i(t))} \nonumber\\
    & \quad \approx N^{-1}\sum_{i=1}^{N}{2(v_i-\overline{v})\left[N^{-1}\sum_{j=1}^{N}{a(x_i-x_j)((v_j-\overline{v}(t))-(v_i-\overline{v}(t)))}dt \right]\delta(x-x_i(t))} \nonumber\\
    & \quad \qquad + N^{-1}\sum_{i=1}^{N}{2(v_i-\overline{v})\left[\sigma (v_i-\overline{v}(t))dB_i - d\overline{v}(t)\right]\delta(x-x_i(t))} \nonumber\\
    &  \quad \qquad + \sigma^2 \left(1-\frac{2}{N}\right) N^{-1}\sum_{i=1}^{N}{(v_i-\overline{v})^2\delta(x-x_i(t))} dt \nonumber \\
    &  \quad \qquad - N^{-1}\sum_{i=1}^{N}{(v_i-\overline{v}(t))^2v_i\delta'(x-x_i(t))}dt \nonumber\\
    & \quad =: Y_1 + Y_2 + Y_3 + Y_4.
\end{align}

\emph{$\bullet$ Treating the terms $Y_1,\dots,Y_4$ in \eqref{differential_k}}.

The term $Y_3$ is precisely $\sigma^2 (1-2N^{-1}) kdt$. The term $Y_2$ can be neglected, as it's purely noise, so will enjoy a pre-factor $N^{-1/2}$ with respect to everything else in the evolution equation for $k$. As for the term $Y_4$, we perform the flocking approximation $v_i \approx \overline{v}(t)$, thus getting $Y_4 \approx -\overline{v}(t)\nabla k$. Finally, $Y_1$ can be rewritten (upon replacing $a(x_i(t)-x_j(t))$ with $a(x-x_j(t))$ thanks to multiplier $\delta(x-x_i(t))$, and using the convolution representations) as
\begin{align*}
    Y_1 
    & = N^{-1}\sum_{i=1}^{N}{2(v_i-\overline{v})\left[N^{-1}\sum_{j=1}^{N}{a(x-x_j)((v_j-\overline{v}(t))-(v_i-\overline{v}(t)))} dt\right]\delta(x-x_i(t))}\\
    & = N^{-1}\sum_{i=1}^{N}{2(v_i-\overline{v})\delta(x-x_i(t))\left[N^{-1}\sum_{j=1}^{N}{a(x-x_j)(v_j-\overline{v})} dt\right]} \\
    & \qquad - N^{-1}\sum_{i=1}^{N}{2(v_i-\overline{v}(t))^2\delta(x-x_i(t))\left[N^{-1}\sum_{j=1}^{N}{a(x-x_j)} dt\right]}\\
    & = 2(j-\overline{v}\rho)a\ast (j-\overline{v}\rho)dt - 2k a\ast \rho dt.
\end{align*}
Combining all the approximations, we end with the SPDE in $(\rho,j,k)$ being
\begin{align}
    d \rho & = - \nabla \cdot j dt \label{reduced_CS_rho_iid_B_gen} ,\\
    d j & = \left\{\rho (a\ast j)  - j (a \ast \rho) - \overline{v}^2(t) \nabla \rho\right\}dt + \sigma N^{-1/2}\sqrt{k}\,d\xi
\label{reduced_CS_j_iid_B_gen},\\
    d k & =
    2(j-\overline{v}(t)\rho)a\ast (j-\overline{v}(t)\rho)dt + \left(\sigma^2(1-2N^{-1})- 2 a\ast \rho\right)k dt - \overline{v}(t)\nabla kdt. \label{reduced_CS_k_iid_B_gen}
\end{align}

\emph{$\bullet$ Equation for mass of $j$}

   Using the definition of $j$ at the particle level \eqref{j} and of the average velocity \eqref{informal_flocking}, it is immediate to notice that $m_j(t) := \int_{x}{j} = \overline{v}(t)$. Furthermore, using this equivalence, we choose to translate this information at the SPDE level by setting 
   \begin{align}\label{defn_v_av}
   \overline{v}(t) := \int_{x}{j}.
   \end{align}
   In the derivation done so far, \eqref{defn_v_av} entails that, integrating the $j$ equation \eqref{reduced_CS_j_iid_B_gen}, we have the representation
    \begin{align}\label{evol_average_j}
        d\overline{v}(t) = \sigma N^{-1/2}\int_{x}\sqrt{k}d\xi.
    \end{align}
    \begin{remark}\label{rem_coloured_noise_v_av}
    Throughout, we will always use the definition \eqref{defn_v_av} at the continuum level. As we will choose to regularise the noise $\xi$ in later parts of the paper, this will by reflected in using a regularised noise in spite of a space-time white noise in \eqref{evol_average_j}.
    \end{remark}

\begin{remark}
In the special case of constant kernel $a\equiv C_a >0$,  the SPDE \eqref{reduced_CS_rho_iid_B_gen}--\eqref{reduced_CS_j_iid_B_gen}--\eqref{reduced_CS_k_iid_B_gen} reduces to
\begin{align}
    d \rho & = - \nabla \cdot j dt \label{const_reduced_CS_rho_iid_B} ,\\
    d j & = \left\{C_a[\overline{v}(t)\rho  - j] - \overline{v}^2(t) \nabla \rho\right\}dt + \sigma N^{-1/2}\sqrt{k}\,d\xi
\label{const_reduced_CS_j_iid_B},\\
    d k & = 
     \underbrace{\left(\sigma^2(1-2N^{-1})- 2 C_a\right)}_{:= -\gamma}k dt - \overline{v}(t)\nabla kdt. \label{const_reduced_CS_k_iid_B}
\end{align}
\end{remark}
\begin{remark}
The structure of \eqref{reduced_CS_rho_iid_B_gen}--\eqref{reduced_CS_k_iid_B_gen} appears to make sense with respect to proving flocking. In particular:
\begin{itemize}
\item in the case of constant kernel $a$, we have that $a\ast \rho = C_a$, thus granting, subject to small enough noise intensity $\sigma$, a dissipative term $(-2C_a + \sigma^2(1-2N^{-1})) k$ in the $k$ equation;
\item The term $2(j-\overline{v}(t)\rho)a\ast (j-\overline{v}(t)\rho)$ incorporates the flocking definition, and takes a small $N^{-1/2}$ prefactor when injected in the $L^2$-norm for $\|j\|^2$.
\end{itemize}
\end{remark}

\begin{remark}
    The procedure adopted to get \eqref{reduced_CS_rho_iid_B_gen}--\eqref{reduced_CS_k_iid_B_gen} 
is a hierarchy in the number power $p$ for the quantities 
$$
N^{-1}\sum_{i=1}^{N}{(v_i-\overline{v}(t))^p \delta(x-x_i(t))},\qquad p \in \mathbb{N},
$$
(we truncate the hierarchy by dropping the noise in the $k$-equation). The hope with this hierarchy is that it will be relatively stable, due to the small noise prefactor $N^{-1/2}$ gained in each step of the hierarchy.  
\end{remark}

\subsubsection{Derivation in higher dimensions}

The $\rho$ equation is unchanged. For $\bj$, we use It\^o's lemma on each component $j_m$. The drift contributions stay the same but the noise cross-variation calculation changes. In the It\^o differential of $j_m$, we can define a noise term
$$
T_m \coloneqq \sigma N^{-1} \sum_{i=1}^N (v_{i, m}(t) - \vbar_m(t)) \delta(\bx - \bx_i(t)) \dd B_{i, m}(t).
$$
We then compute the cross-variation $\left \llangle T_m(\bx), T_n(\by) \right \rrangle$:
$$
\left \llangle T_m(\bx), T_n(\by) \right \rrangle = \sigma^2 N^{-1} \brack{N^{-1}\sum_{i=1}^N(v_{i, m}(t) - \vbar_m(t))^2 \delta(\bx - \bx_i(t))}\delta(\bx -\by) \delta_{mn}.
$$
Motivated by this we define $\bk(\bx, t) \in \btd$ via
$$
k_m(\bx, t) \coloneqq N^{-1}\sum_{i=1}^N(v_{i, m}(t) - \vbar_m(t))^2 \delta(\bx - \bx_i(t)), \quad m = 1, \ldots, d.
$$
This means that we can close the noise term by choosing it to be $\sigma N^{-1/2}\sqrt{\bk} \odot \dd \bxi$ where the square root is taken elementwise, $\bxi$ is a $d$-dimensional space-time white noise, and $\odot$ denotes elementwise (Hadamard) product. We then compute the It\^o differential of each component of $\bk$. This is similar to the one-dimensional calculation and gives the following system of SPDEs:
\begin{align}
    \dd \rho &= - \nabla \cdot \bj \dd t, \label{reduced_CS_rho_iid_B_gen_multidim}\\
    \dd \bj &= [\rho(a * \bj) - \bj(a*\rho) - \bvbar(t) \nabla \cdot \bj] \dd t + \sigma N^{-1/2} \sqrt{\bk} \odot \dd \bxi, \label{reduced_CS_j_iid_B_gen-multidim}\\
    \dd \bk &= [2(\bj - \bvbar(t) \rho)]\odot[a * (\bj - \bvbar(t) \rho)] \dd t + \sbrack{\sigma^2(1-2N^{-1}) - 2a*\rho}\bk \dd t - \bvbar(t) \nabla \cdot \bk \dd t. \label{reduced_CS_k_iid_B_gen-multidim}
\end{align}
Similarly to the common noise setting, we note that the equations are closed using one possible choice of closure. Moreover, $\bvbar(t)$ evolves via
$$
\dd \bvbar(t) = \sigma N^{-1/2} \int_{\bx} \sqrt{\bk} \odot \dd \bxi.
$$

\subsection{SPDE flocking results in 1D}\label{main_results_iid_case}

We now state the SPDE-flocking results for the 1D independent noise case. 
In order to study the equation in an $L^2$ setting, we need to regularize it, as is also the case for singular SPDEs such as the Dean–Kawasaki equation. To this end, we replace the space-time white noise $\xi$ with a coloured noise, which, for simplicity, we continue to denote by $\xi$. Specifically, since the noise $\xi$ is coloured, we may expand it as
\begin{align}\label{coloured_noise}
\dd \xi = \sum_{z}{\sqrt{\lambda_z} e_z(x)d\beta_z(t)}.
\end{align}
where $\{e_z\}$ is orthonormal  in $L^2_x$, uniformly bounded in $L^\infty$. We denote $\Tr(\xi):= \sum_{z}{\lambda_z}$. This may reflect either coloured noise, or mode truncation resulting from a numerical discretisation. This noise regularisation setting is standard for SPDEs of fluctuating hydrodynamics, see for instance \cite{cornalba2023dean} or \cite{djurdjevac2024weak}.

We first  prove the well-posedness of \eqref{const_reduced_CS_rho_iid_B}--\eqref{const_reduced_CS_k_iid_B}, where we exchange the space-time white noise with the regularized noise given by \eqref{coloured_noise}. Furthermore, we use the notation 
\[
\|f\|^2_{H^1} = \|f\|^2_{L^2} + \|\partial_x f\|^2_{L^2}.
\]
    
\begin{theorem}[Global $H^1$ well-posedness]
\label{thm:global_H1}
Fix $T>0$ and assume that
$\rho_0,j_0\in H^1(\mathbb T)$,
$k_0\geq 0$ a.e., and
$q_0:=\sqrt{k_0}\in H^1(\mathbb T)$. Assume furthermore that the noise satisfies
\begin{equation}\label{summability_W_1_infty}
    \sum_{z\geq 1}
    \lambda_z\|e_z\|_{W^{1,\infty}}^2<\infty.
\end{equation}
Assume that the random velocity  $\overline v$
satisfies the non-degeneracy and boundedness
condition
\[
    0<v_*\leq |\overline v(t)|
    \leq v^*<\infty,
    \qquad t\in[0,T],
    \quad\text{a.s.},
\]
for some deterministic constants $v_*,v^*>0$.

Then the system \eqref{const_reduced_CS_rho_iid_B}--\eqref{const_reduced_CS_k_iid_B}
admits a unique adapted solution
$(\rho,j,k,\overline v)$ such that
$
    (\rho,j,k)\in
    C([0,T];H^1(\mathbb T)^3)
    \text{ a.s.},
$ 
and
$ 
    \overline v\in C([0,T];\mathbb R)
\text{ a.s.}
$
Moreover,
\[
    \mathbb E\left[
    \sup_{0\leq t\leq T}
    \left(
        \|\rho(t)\|_{H^1}^2
        +\|j(t)\|_{H^1}^2
    \right)
    \right]<\infty,
\]
and
\[
    k(t,x)\geq 0
    \quad\text{for all }t\in[0,T]
    \text{ and a.e. }x\in\mathbb T,
    \quad\text{a.s.}
\]
\end{theorem}

The proof of the theorem is provided in Appendix \ref{app: well-posedness-iid}.
 
\begin{remark}
The deterministic velocity bounds are a restrictive conditional hypothesis, not a consequence of the velocity SDE \eqref{velocitySDE}. Without imposing them globally, one may instead work up to the exit time
$\tau_v=\inf\{t\geq0:|\overline v(t)|\leq v_*\text{ or }|\overline v(t)|\geq v^*\}$,
provided $v_*<|\overline v_0|<v^*$. On every stopped interval the same estimates hold with constants depending on $v_*,v^*$. This yields the corresponding local-in-time solution up to the exit time. We do not use this localization in the proof below.
\end{remark}

Next we state our flocking result (Theorem \ref{thm: constant-kernel-flocking-iid}, and related result in Remark \ref{thm: general-kernel-flocking-iid-exp}) and the associated assumptions (Assumption \ref{ass_smallness_iid}).
The proofs are provided in Appendix \ref{sec:proofs_iid}.

\begin{assumption}[Smallness assumption for i.i.d. noise case]\label{ass_smallness_iid}
In order to avoid repetitions in Theorems \ref{thm: constant-kernel-flocking-iid} and Remark \ref{thm: general-kernel-flocking-iid-exp}, we use a slight abuse of notation in this assumption only and drop subscripts in $h$ when referring to objects from Remark \ref{thm: general-kernel-flocking-iid-exp}. Assume the following:
\begin{enumerate}[label=(A\arabic*)]
    \item The initial condition $(\rho_0, j_0, k_0)\in L^2(\Omega; H^1\times H^1)$ is $\mathcal{F}_0$-measurable, and is independent of the noise $\xi(t)$ for all $t \geq 0$.
    \item \emph{(for both Theorem \ref{thm: constant-kernel-flocking-iid} and Remark \ref{thm: general-kernel-flocking-iid-exp})} The size of the $L^2$-norm of the SPDE noise (this being $\sigma N^{-1/2}\Tr^{1/2}(\xi)$) is bounded as $(\sigma N^{-1/2}\Tr^{1/2}(\xi))^2 \leq \delta \cdot C_a$ for small enough $\delta>0$. Furthermore, we assume $\Tr(\xi)\geq 1$. 
    \item \emph{(for both Theorem \ref{thm: constant-kernel-flocking-iid} and Remark \ref{thm: general-kernel-flocking-iid-exp})}  The parameter $\sigma$ is small enough so that $\sigma^2(1-2N^{-1})^2 \leq  \delta C_a$ for small enough $\delta > 0$. \label{ass:l2noise_size} 
    \item \emph{(for Theorem \ref{thm: constant-kernel-flocking-iid} only)} The initial data $Y_0:=\|j(0)-\overline{v}(0)\rho(0)\|^2 + \|k(0)\|_{H^1}$ and $\eta_0 := \|j(0)\|^2 + \overline{v}^2(0)\|\rho(0)\|^2 + \|k(0)\|_{H^1}$ are small enough relative to $\overline{v}(0)\sqrt{\Tr(\xi)}$.
    \item \emph{(for Remark \ref{thm: general-kernel-flocking-iid-exp} only)} The initial data $Y_0:=\|j(0)-\overline{v}(0)\rho(0)\|^2 + \|k(0)\|_{L^2}$ and $\eta_0 := \|j(0)\|^2 + \overline{v}^2(0)\|\rho(0)\|^2 + \|k(0)\|_{L^2}$ are small enough relative to $\min\{\overline{v}(0)\sqrt{\Tr(\xi)}; \overline{v}(0)\frac{C_a}{\theta\|g\|}\}$.
    \item \emph{(for both Theorem \ref{thm: constant-kernel-flocking-iid} and Remark \ref{thm: general-kernel-flocking-iid-exp})} We have the bound $\|j(0)\|^2 + \overline{v}^2(0)\|\rho(0)\|^2 + \|k(0)\|_{L^2} \leq R^2_0$ for some deterministic constant $R^2_0$.
\end{enumerate}
\end{assumption}

\begin{theorem}[i.i.d. noise, constant kernel case]\label{thm: constant-kernel-flocking-iid}
    Consider the case of a constant kernel $a \equiv C_a$. In \eqref{const_reduced_CS_rho_iid_B}--\eqref{const_reduced_CS_j_iid_B}--\eqref{const_reduced_CS_k_iid_B}, replace the space-time white noise with the trace-class noise \eqref{coloured_noise}. Assume the smallness conditions in Assumption \ref{ass_smallness_iid}.
     Then, for some $\delta > 0$, we have the following flocking estimate 
       \begin{align}\label{flocking_est_iid}
        \E{\mathbf{1}_{\tau=\infty}\norm{j(t) - \vbar(t) \rho(t)}^2} \leq e^{- \delta C_a t}\E{\norm{j_0 - \vbar(0) \rho_0}^2 +\|k(0)\|_{H^1}}, \quad t \geq 0.
        \end{align}
        where $\tau$ is a stopping time satisfying
        \begin{align}\label{stop_time_bound_iid}
    \mathbb{P}(\tau < \infty) \propto \exp\left\{ - C (\sigma^2 N^{-1} \gamma^{-1} \overline{v}^2(0))^{-\alpha} \right\} + \exp\{ - (\sigma N^{-1/2}\Tr^{1/2}(\xi))^{-2}\},
        \end{align}
        for some $C>0$ and $\alpha>0$. 
\end{theorem}
For the general kernel case, we work with a discretised version of the SPDE, for simplicity via finite differences, see \eqref{reduced_CS_rho_iid_B_gen_discr}--\eqref{reduced_CS_j_iid_B_gen_discr}--\eqref{reduced_CS_k_iid_B_gen_discr}. That is, we work on periodic mesh $G_h$ of $\mathbb{T}$ with mesh size $h>0$, and write $\|f_h\|_{p,h} = (\sum_{x\in G_h} hf_h^p(x) )^{1/p}$, and abbreviate $\|\cdot\|_{2,h} = \|\cdot\|_h$. The operators $\nabla_h, \nabla_h\cdot $ are standard differential operators. Furthermore, we denote $f_h\ast_h g_h(x) = \sum_{y\in G_h}{h f_h(x-y)g_h(y)}$.

As hinted in Remark \ref{rem:non-constant-kernel}, the SPDE well-posedness in this case is more complicated, and we have chosen not to pursue proving it for the discretised system \eqref{reduced_CS_rho_iid_B_gen_discr}--\eqref{reduced_CS_j_iid_B_gen_discr}--\eqref{reduced_CS_k_iid_B_gen_discr}. What we have focused on is instead verifying the SPDE flocking \emph{conditional} on well-posedness: for this reason, we frame our analysis in a remark (see Remark \ref{thm: general-kernel-flocking-iid-exp} below) rather than a theorem. However, our numerical simulations support our findings, and confirm our modelling choices.

\begin{remark}[i.i.d. noise, full kernel case]\label{thm: general-kernel-flocking-iid-exp}
    Consider the case of a general kernel $a = C_a + \theta g$. Consider a finite difference discretised version of \eqref{reduced_CS_rho_iid_B_gen}--\eqref{reduced_CS_j_iid_B_gen}--\eqref{reduced_CS_k_iid_B_gen} with mesh size $h$
    \begin{align}
    d \rho_h & = - \nabla_h \cdot j_h dt \label{reduced_CS_rho_iid_B_gen_discr} ,\\
    d j_h & = \left\{\rho_h (a\ast_h j_h)  - j_h (a \ast_h \rho_h) - \overline{v}^2(t) \nabla_h \rho_h\right\}dt + \sigma N^{-1/2}\sqrt{k_h}\,d\xi_h
\label{reduced_CS_j_iid_B_gen_discr},\\
    d k_h & =
    2(j_h-\overline{v}(t)\rho_h)a\ast (j_h-\overline{v}(t)\rho_h)dt + \left(\sigma^2(1-2N^{-1})- 2 a\ast_h \rho_h\right)k_h dt - \overline{v}(t)\nabla k_h dt, \label{reduced_CS_k_iid_B_gen_discr}
\end{align}
where $\xi_h$ is a discretised version of $\xi$ in \eqref{coloured_noise}, such that $\Tr(\xi_h) \gtrsim h^{-1/2}$. Assume we have a global solution to \eqref{reduced_CS_rho_iid_B_gen_discr}--\eqref{reduced_CS_j_iid_B_gen_discr}--\eqref{reduced_CS_k_iid_B_gen_discr}.
Assume the smallness conditions in Assumption \ref{ass_smallness_iid}.
     Then, for some $\delta > 0$, we have the following flocking estimate 
       \begin{align}\label{flocking_est_iid_general}
        \E{\mathbf{1}_{\tau=\infty}\norm{j_h(t) - \vbar(t) \rho_h(t)}_h^2} \leq e^{- \delta C_a t}\E{\norm{j_{h,0} - \vbar(0) \rho_{h,0}}_h^2 +\|k_{h,0}\|_h}, \quad t \geq 0,
        \end{align}
        where $\tau$ is a stopping time satisfying
        \begin{align}
            \mathbb{P}(\tau<\infty) & \lesssim \exp\{-(\sigma^{-2} N^{}\delta C_a \overline{v}^2(0))^{\alpha} \} + \exp\{ - (\sigma N^{-1/2}\Tr^{1/2}(\xi_h))^{-2}\}.
        \end{align}
\end{remark}

\subsection{Improved kinetic closure}\label{sec: improved-kinetic-closure}

We now outline a method for improving the behaviour of our reduced SPDE model in the pre-flocking regime. We consider $d=1$ for simplicity although we expect such a procedure could be adapted to the multi-dimensional setting. One of the main sources of error in the derivation given in Section \ref{subsec_derivation_SPDE_iid} is the closing (flocking) approximation
$$
T_2 = -N^{-1}\sum_{i=1}^N v_i^2(t) \delta'(x-x_i(t)) \approx -N^{-1}\sum_{i=1}^N \vbar^2(t) \delta'(x-x_i(t)).
$$
To circumvent such an approximation, we can expand 
$$
v_i^2(t) = (v_i(t) - \vbar(t) + \vbar(t))^2 = (v_i(t) - \vbar(t))^2 + \vbar^2(t) + 2 \vbar(t) (v_i(t) - \vbar(t)).
$$
Intuitively, we have removed the flocking approximation at the expense of picking up additional terms in the reduced SPDE. We now show 
that these terms can be written in terms of our existing densities ($\rho, j, k)$ so that the system remains closed. Specifically, we write
\begin{align*}
    T_2 &= -N^{-1}\sum_{i=1}^N v_i^2(t) \delta'(x-x_i(t)) \\
    &= -N^{-1}\sum_{i=1}^N \brack{(v_i(t) - \vbar(t))^2 + \vbar^2(t) + 2 \vbar(t) (v_i(t) - \vbar(t))}\delta'(x-x_i(t)) \\
    &= - \partial_x k - \vbar^2(t) \partial_x\rho - 2\vbar(t) \partial_x(j-\vbar(t) \rho).
\end{align*}
The remainder of the derivation is unchanged, giving the following system 
\begin{align}
    \dd \rho &= - \partial_x j dt, \label{eq: improved-rho}\\
    \dd j &= \left\{\rho (a\ast j)  - j (a \ast \rho) - \overline{v}^2(t) \partial_x \rho - \partial_x k - 2\vbar(t) \partial_x(j-\vbar(t) \rho)\right\}dt + \sigma N^{-1/2}\sqrt{k}\,d\xi, \label{eq: improved-j}\\
    \dd k &=
    2(j-\overline{v}(t)\rho)a\ast (j-\overline{v}(t)\rho)dt + \left(\sigma^2(1-2N^{-1}) - 2 a\ast \rho\right)k dt - \overline{v}(t)\partial_x kdt. \label{eq: improved-k}
\end{align}

\begin{remark}\label{remark_improved_model}
    In the improved system \eqref{eq: improved-rho}--\eqref{eq: improved-j}--\eqref{eq: improved-k}, the transport coefficient for $j - \overline{v}\rho$ switches sign ($+\overline{v}$ to $-\overline{v}$) with respect to the standard closed system \eqref{reduced_CS_rho_iid_B_gen}--\eqref{reduced_CS_j_iid_B_gen}--\eqref{reduced_CS_k_iid_B_gen}. In other words, the transport in the improved model follows the average flock velocity, rather than its opposite. 
\end{remark}
The $\rho$ and $j$ equations are now exact with the closing approximation seen entirely at the level of the $k$-SPDE. In particular, referring back to the derivation in Section \ref{subsec_derivation_SPDE_iid} (in particular, \eqref{differential_k_gen}), the It\^o differential for $k$ reads
$$
\dd k = Y_1 + Y_2 + Y_3 -N^{-1}\sum_{i=1}^N (v_i(t) - \vbar(t))^2 v_i(t) \delta'(x-x_i(t)) \dd t,
$$
where the final term is closed using the flocking approximation $v_i(t) \approx \vbar(t)$. Equivalently, one can write $v_i(t) = v_i(t) - \vbar(t) + \vbar(t)$ to get
$$
-N^{-1}\sum_{i=1}^N (v_i(t) - \vbar(t))^2 v_i(t) \delta'(x-x_i(t)) = -\vbar(t) \partial_x k -N^{-1}\sum_{i=1}^N (v_i(t) - \vbar(t))^3 \delta'(x-x_i(t)),
$$
with the closing approximation now being that we drop the final term. In fact, we could iterate this procedure to produce a hierarchy of coupled SPDEs in the quantities
$$
-N^{-1}\sum_{i=1}^N (v_i(t) - \vbar(t))^p \delta(x-x_i(t)),
$$
with the closing approximation manifesting purely via the residual $-N^{-1}\sum_{i=1}^N (v_i(t) - \vbar(t))^p \delta'(x-x_i(t))$. For sufficiently small velocity spreads $v_i(t) - \vbar(t)$, we expect the error due to the residual to be negligible. 
To substantiate our claim that this technique leads to an improved model \eqref{eq: improved-rho}--\eqref{eq: improved-j}--\eqref{eq: improved-k}, we present supporting numerical evidence in 
 Section \ref{sec: numerics-improved-closure}. 
 In addition, computations similar to those of Section \ref{subsec_SPDE_constant_kernel_env_noise} allow us to show that \eqref{eq: improved-rho}--\eqref{eq: improved-j}--\eqref{eq: improved-k} encodes flocking behaviour in the constant kernel case $a = C_a$, although controlling the energy is subject to analogous considerations as in the high dimension case. We defer a full theoretical error analysis of this to future work. 
It would also be interesting to investigate if it is beneficial to go further into the hierarchy.

\section{Numerical experiments}\label{sec: numerics}

In this section, we present numerical experiments using the standard interaction kernel \eqref{eq: kernel-definition}. Unless stated otherwise, we choose $r=1/2$ and $\lambda = 50$. The code used to generate all figures is available at \url{https://github.com/kamran-arora/reduced-SPDE-stochastic-cucker-smale}. The repository contains Python implementations of the interacting particle systems and reduced SPDEs in arbitrary spatial dimension $d$.

\subsection{Implementation details}

We discretise the reduced SPDEs using centered finite differences in space with additional artificial diffusion of the form $\nu_h \Delta_h$ and fully-explicit Euler-Maruyama in time. The symbol $\Delta_h$ denotes the standard discretisation of the Laplacian and $\nu_h = \absv{\vbar(t)}h/2$. In $d=1$ this is equivalent to first-order upwinding of the transport terms \cite[Chapter 10]{leveque2007finite} whilst in higher dimensions we treat it as a form of numerical regularisation. The discrete convolution is implemented using the Fast Fourier transform. For the SPDE with common noise, the driving Brownian motion is the same as the one used to simulate the associated particle system. For the SPDE with i.i.d.\ noise, the discrete space-time white noise $\boldsymbol{\xi}_h$ is defined as follows: let $G_h$ denote our spatial grid with spacing $h>0$ and let $\boldsymbol{e}^h_{\bx}(\by) \coloneqq h^{-d/2}\delta_{\bx \by}$ be a basis. Then,
$$
\boldsymbol{\xi}_h(t) \coloneqq \sum_{\bx \in G_h} \boldsymbol{B}_{\bx}(t)\boldsymbol{e}^h_{\bx},
$$
where $(\boldsymbol{B}_{\bx})_{\bx \in G_h}$ is a collection of independent, standard $d$-dimensional Brownian motions. Initial densities are computed by placing a von Mises (VM) kernel with variance $\ve^2$ at each particle position and then evaluating the $\ve$-smoothed empirical measures over the numerical grid. On $[-L/2, L/2]$, the VM kernel has distribution
$$
f_{\mathrm{VM}}(x|\ve) = \frac{\exp\brack{\ve^{-2} \cos\brack{\frac{2\pi x}{L}}}}{I_0(\ve^{-2})L},
$$
where $I_0(\cdot)$ is the modified Bessel function of the first kind of order $0$. In higher dimensions the VM kernel is defined via products of one-dimensional VM distributions. That is,
$$
f_{\mathrm{VM}}^{(d)}(x_1, \ldots, x_d | \ve) \coloneqq \prod_{i=1}^d f_{\mathrm{VM}}(x_i | \ve).
$$

In order to make comparisons between the particle systems and their associated reduced SPDEs it will be convenient to represent the former using smoothed empirical measures. In particular, we define
\begin{align*}
    \rho_{\ve}(\bx, t) &\coloneqq N^{-1} \sum_{i=1}^N f_{\mathrm{VM}}^{(d)}(\bx - \bx_i(t) | \ve), \\
    \bj_{\ve}(\bx, t) &\coloneqq N^{-1} \sum_{i=1}^N \bv_i(t)f_{\mathrm{VM}}^{(d)}(\bx - \bx_i(t) | \ve), \\
    \boldsymbol{k}_{\ve}(\bx, t) &\coloneqq N^{-1} \sum_{i=1}^N (\bv_i(t) - \bvbar(t))^2f_{\mathrm{VM}}^{(d)}(\bx - \bx_i(t) | \ve).
\end{align*}

\subsection{Approximation quality}\label{subsec: approx-quality}

We first consider the approximation quality of the reduced SPDE with common noise in $d=1$ \eqref{eq:spde-vanishing-common-1d}. We plot the mean and variance of $\rho$ and $j$ at times $t \in \{0.5, 2.0\}$. For each realisation, we generate a random initial condition of $N=250$ particles. The initial positions are split equally over three clusters, and within each cluster the velocities are sampled uniformly over a chosen, cluster-specific interval. This produces three clusters, each moving with a different average initial velocity. We simulate the particle system and reduced SPDE using timestep $\Delta t = 0.0005$ until $t=2.0$ for $100$ realisations on $[-50, 50]$ with $\sigma=1.0$. The VM parameter is $\ve = 0.125$, and the grid is discretised uniformly into $256$ cells. The results are shown in Figure \ref{fig: approximation-vc} where we see good agreement between the particles and SPDE, with clear alignment between $\rho$ and $j$ in the bottom panel.

\begin{figure}[h!]
    \centering
    \includegraphics[width=0.9\linewidth]{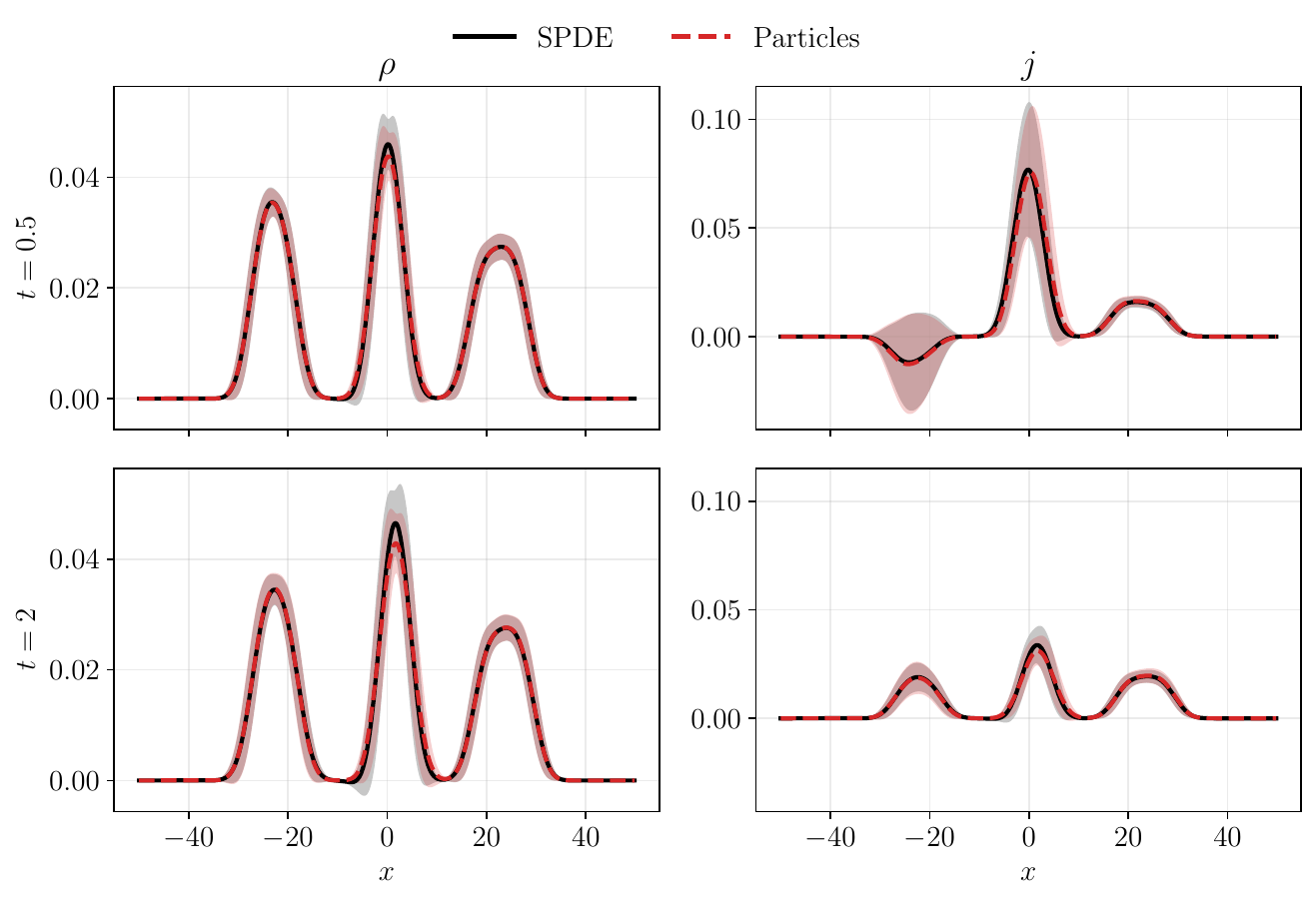}
    \caption{Comparison between the expected values of $\rho, j$ and $\rho_\ve, j_\ve$ for the common noise model averaged over $100$ realisations at times $t \in \{0.5, 2.0\}$. The shading corresponds to one standard deviation.}
    \label{fig: approximation-vc}
\end{figure}

We now consider the reduced SPDE with i.i.d.\ noise in $d=1$ \eqref{reduced_CS_rho_iid_B_gen}-\eqref{reduced_CS_k_iid_B_gen}. The setup is otherwise identical. The results are shown in Figure \ref{fig: approximation-viid} where we see good agreement between the particles and SPDE with clear alignment between $\rho$ and $j$ in the bottom panel. Moreover, by the final time $k$ has decayed close to zero.

\begin{figure}[h!]
    \centering
    \includegraphics[width=0.9\linewidth]{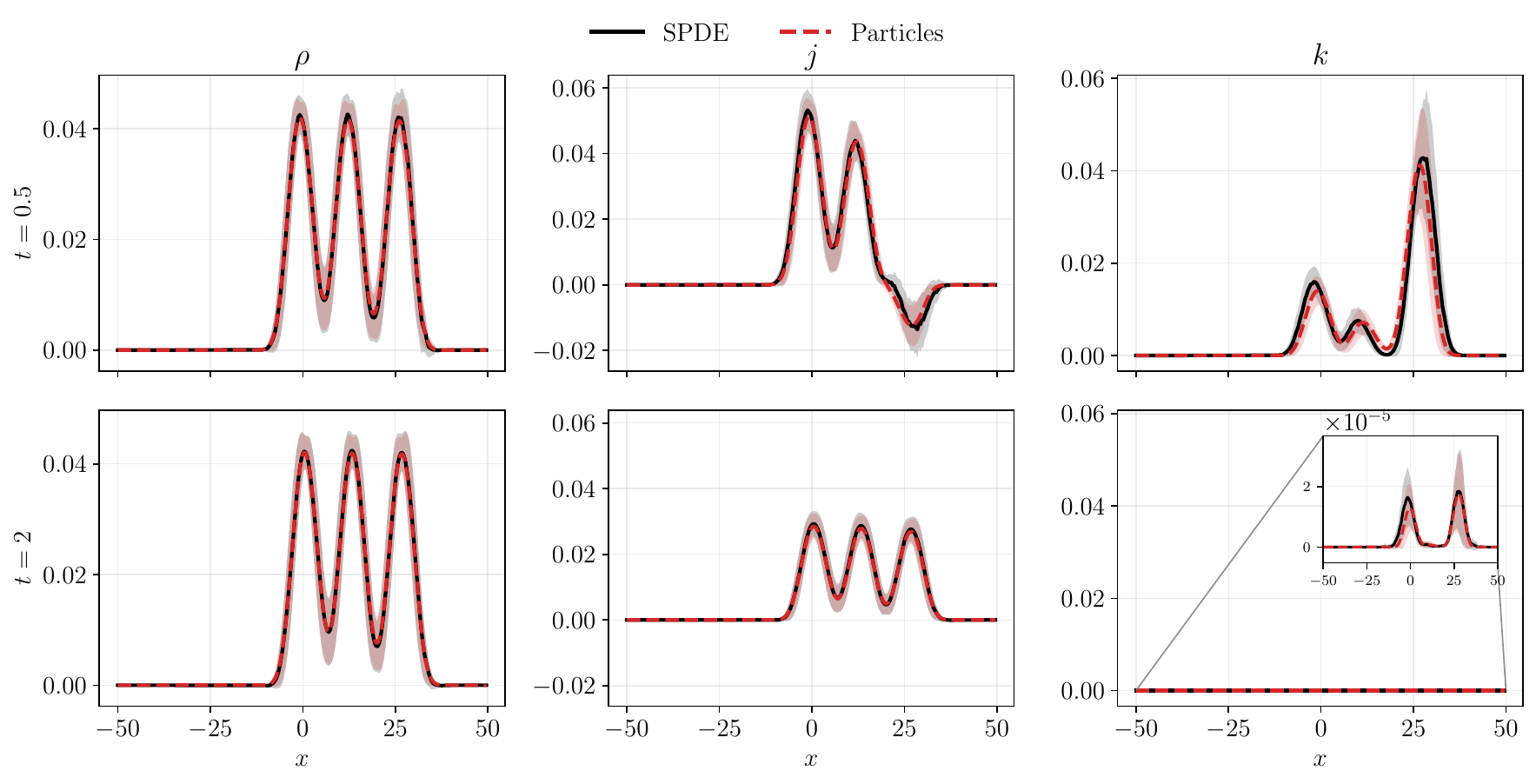}
    \caption{Comparison between the expected values of $\rho, j, k$ and $\rho_\ve, j_\ve, k_\ve$ for the i.i.d.\ noise model averaged over $100$ realisations at times $t \in \{0.5, 2.0\}$. The shading corresponds to one standard deviation.}
    \label{fig: approximation-viid}
\end{figure}

\subsection{Flocking}

We now consider flocking for the reduced SPDE with common noise in $d=1$ \eqref{eq:spde-vanishing-common-1d} and $d=2$ \eqref{eq:spde-vanishing-common-nd} by comparing the quantities $\mathbb{E}[\norm{\bj(t) - \vbar \rho(t)}^2]$ and $\mathbb{E}[\norm{\bj_{\ve}(t) - \vbar \rho_{\ve}(t)}^2]$. For each realisation, we generate a random initial configuration of $N=250$ particles using the same procedure as in the previous section. This ensures that the initial condition is firmly in the non-flocking regime. We simulate both the particles and SPDE until $t=0.5$ using timestep $\Delta t = 0.005$ for $1,000$ realisations when $d=1$ and $100$ realisations when $d=2$. The VM parameter is $\ve=0.125$ in $d=1$ and $\ve=0.2$ in $d=2$ whilst the grid is discretised uniformly into $256$ cells for $d=1$ and $64 \times 64$ cells for $d=2$. The results are shown in Figure \ref{fig: flocking-vc}\subref{fig: flocking-vc-1d} and Figure \ref{fig: flocking-vc}\subref{fig: flocking-vc-2d} for $d=1$ and $d=2$ respectively. As expected from our theory, we observe flocking for small values of $\sigma$ and non-flocking for large values of $\sigma$. Moreover, we have good agreement between the particle system and reduced SPDE.

\begin{figure}[h!]
    \centering
    \begin{subfigure}{0.48\textwidth}
        \centering
        \includegraphics[width=\linewidth]{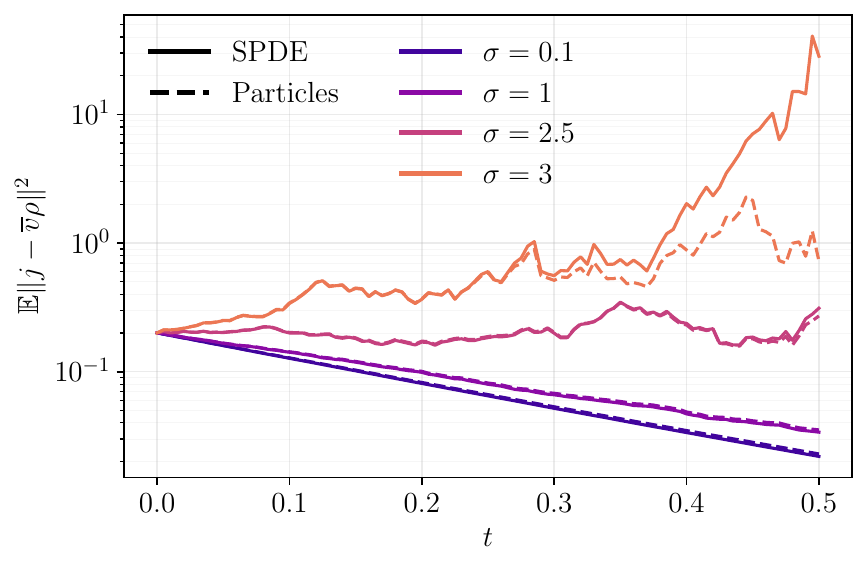}
        \caption{$d=1$}
        \label{fig: flocking-vc-1d}
    \end{subfigure}\hfill
    \begin{subfigure}{0.48\textwidth}
        \centering
        \includegraphics[width=\linewidth]{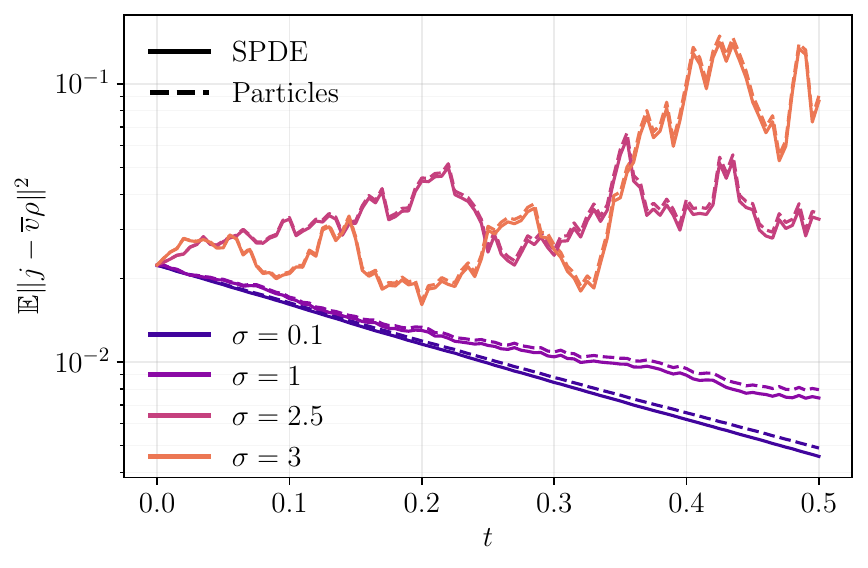}
        \caption{$d=2$}
        \label{fig: flocking-vc-2d}
    \end{subfigure}\hfill
    \caption{Comparison of the flocking quantities $\mathbb{E}\norm{j(t)-\vbar \rho(t)}^2$ and $\mathbb{E}\norm{j_\ve(t) - \vbar \rho_\ve(t)}^2$ between the particle system and reduced SPDE in the common noise setting for $d \in \{1, 2\}$ and varying values of $\sigma$.}
    \label{fig: flocking-vc}
\end{figure}

We now consider flocking for the reduced SPDE with i.i.d noise in $d=1$ \eqref{reduced_CS_rho_iid_B_gen}-\eqref{reduced_CS_k_iid_B_gen} and $d=2$ \eqref{reduced_CS_rho_iid_B_gen_multidim}-\eqref{reduced_CS_k_iid_B_gen-multidim}. The experimental setup is identical to the previous setting but we now compare the flocking quantities $\mathbb{E}[\norm{\bj(t) - \vbar \rho(t)}^2] + \mathbb{E}[\norm{k}]$ and $\mathbb{E}[\norm{\bj_{\ve}(t) - \vbar \rho_{\ve}(t)}^2] + \mathbb{E}[\norm{k_\ve}]$ where we recall that $\vbar$ is time-dependent in this setting. The results are shown in Figure \ref{fig: approximation-viid}\subref{fig: flocking-iid-1d} and Figure \ref{fig: approximation-viid}\subref{fig: flocking-iid-2d} for $d=1$ and $d=2$ respectively. Once again, our numerics support the theory, that is, we observe exponential decay for small $\sigma$ and no decay for large $\sigma$. We also see good agreement between the particles and SPDE.

\begin{figure}[h!]
    \centering
    \begin{subfigure}{0.48\textwidth}
        \centering
        \includegraphics[width=\linewidth]{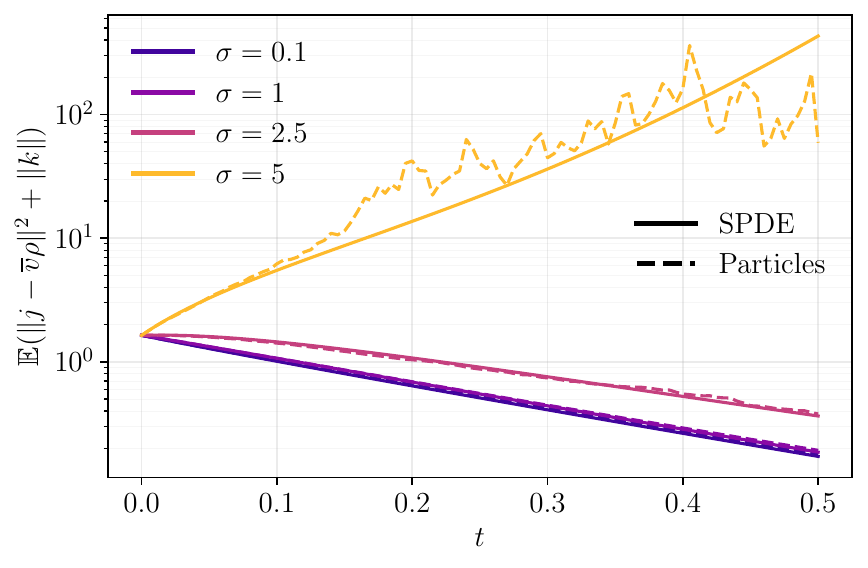}
        \caption{$d=1$}
        \label{fig: flocking-iid-1d}
    \end{subfigure}\hfill
    \begin{subfigure}{0.48\textwidth}
        \centering
        \includegraphics[width=\linewidth]{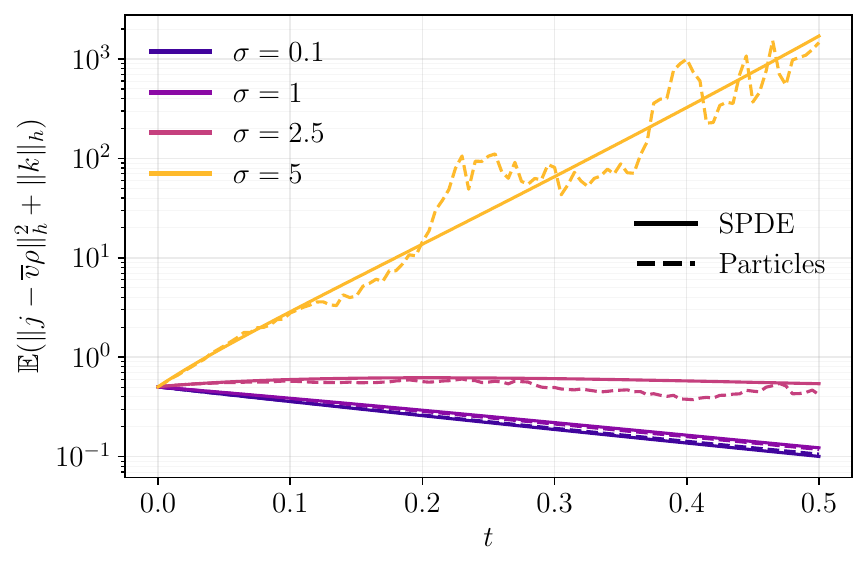}
        \caption{$d=2$}
        \label{fig: flocking-iid-2d}
    \end{subfigure}\hfill
    \caption{Comparison of the flocking quantities $\mathbb{E}\norm{j(t)-\vbar(t) \rho(t)}^2 + \mathbb{E}\norm{k}$ and $\mathbb{E}\norm{j_\ve(t) - \vbar(t) \rho_\ve(t)}^2 + \mathbb{E}\norm{k_\ve}$ between the particle system and reduced SPDE in the i.i.d setting for $d \in \{1, 2\}$ and varying values of $\sigma$.}
    \label{fig: flocking-iid}
\end{figure}

\subsection{Computational cost}

One of the core motivations for the reduced SPDE is its improved computational cost over the particle system. In particular, whereas the cost of the particle system will increase superlinearly with the total number of particles $N$, the SPDE cost will be independent of $N$, depending only on the chosen grid resolution. To demonstrate this, we vary $N$ for both models and measure the time taken to simulate one timestep $\Delta t$. We take $N \in \cbrack{2^3, \ldots, 2^{12}}$ and solve both systems for $100$ timesteps, taking the average time. We then average this result over $100$ runs, each run with a different initial condition drawn from a common distribution and a different realisation of the driving noise. For $d=1$, we discretise the domain uniformly into $256$ cells and in $d=2$ we discretise uniformly into $64 \times 64$ cells. The results are shown in Figure \ref{fig: performance-comparison} for $d \in \{1, 2\}$. As expected, the particle model exhibits superlinear growth in the number of particles $N$ whereas the cost of the reduced SPDE is independent of $N$. In 1D, we see that the common noise SPDE is more computationally efficient for $N \approx 30$ and the i.i.d noise SPDE for $N \approx 60$. Moreover, in 2D we see that the common noise SPDE is more computationally efficient for $N \approx 200$ and the i.i.d noise SPDE for $N \approx 300$.

\begin{figure}[h!]
    \centering
    \begin{subfigure}{0.48\textwidth}
        \centering
        \includegraphics[width=\linewidth]{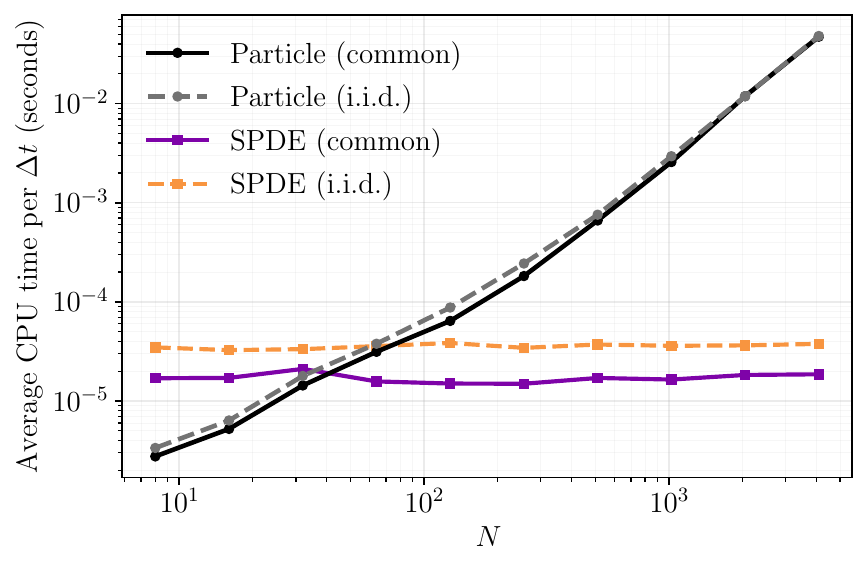}
        \caption{$d=1$}
    \end{subfigure}\hfill
    \begin{subfigure}{0.48\textwidth}
        \centering
        \includegraphics[width=\linewidth]{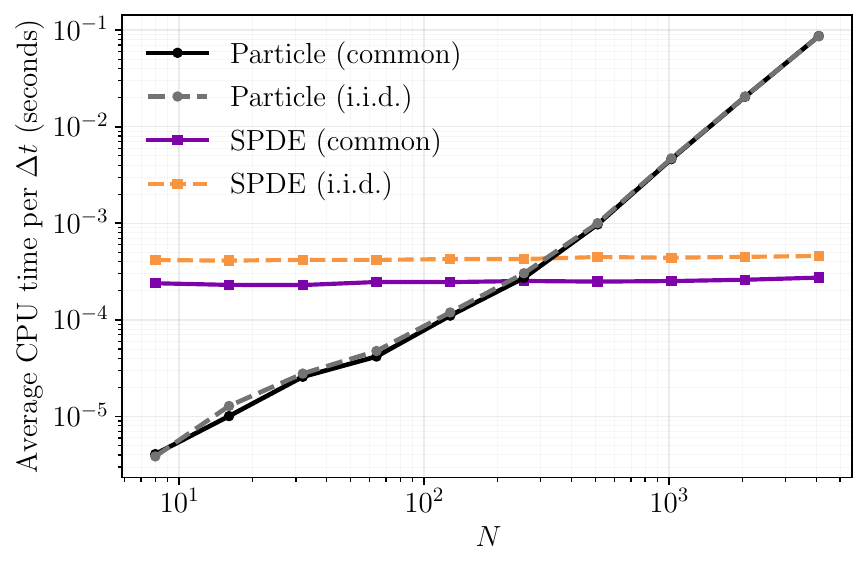}
        \caption{$d=2$}
    \end{subfigure}\hfill
    \caption{Average CPU time in seconds to compute one timestep of the particle model and associated reduced SPDE in dimensions $d \in \{1, 2\}$ for varying $N$ averaged over 100 runs.}
    \label{fig: performance-comparison}
\end{figure}

\subsection{Accuracy in the $L^2$-norm in the common noise setting}

To complement Theorem \ref{thm: error} which gives the error between the particle system and reduced SPDE in the $\hmtwo$-norm, we present numerical results for the error in the squared $L^2$-norm. Recalling the definitions $\xpde(t) \coloneqq (\rho(t), \bj(t))$ and $\xcs(t) \coloneqq (\rho_\ve(t), \bj_\ve(t))$, this amounts to considering
$$
\E{\norm{\xpde(T) - \xcs(T)}^2},
$$
for some terminal time $T>0$. We work with $d \in \{1, 2\}$ and consider two types of initial condition: the first consists of initial positions drawn uniformly at random on $[-50, 50]^d$ and initial velocities drawn uniformly at random on $[-20, 20]^d$, whilst the second corresponds to a three cluster configuration. We perform $250$ realisations with a random initial configuration for each. In $d=1$ we discretise the domain uniformly into $256$ cells and when $d=2$ we discretise the domain uniformly into $64 \times 64$ cells. We fix the terminal time $T=2.0$ and set $\sigma=0.5$. We choose $\ve \in \{0.12, 0.2\}$ for $d=1$ and $\ve \in \{0.2, 0.33\}$ for $d=2$. We plot the terminal error against the number of particles for $N \in \{2^4, \ldots, 2^{11}\}$. The results for $d=1$ are shown in Figure \ref{fig: L2-error-vc}\subref{fig: L2-error-vc-1d} where for the spatially uniform initial condition we see an expected dependence of the form $N^{-1}\ve^{-\beta}$. This is in line with \cite[Remark 3.4]{zimper2025reduced} where a similar observation is noted in the deterministic setting as a consequence of several considerations in \cite{cornalba2021well}. For the clustered initial condition, as $N$ grows we see that the error begins to level off. A possible reason is that the error in $N^{-1}$ has saturated and the plateauing is the residual error from the flocking approximation used to close the model. The results for $d=2$ are shown in Figure \ref{fig: L2-error-vc}\subref{fig: L2-error-vc-2d} where we see similar behaviour.  Theorem \ref{thm: error} also suggests that the error will depend on a number of other factors including the initial velocity spread, the flocking rate, and the variance of the von Mises kernel. We defer investigation of these dependencies to future work.

\begin{figure}[h!]
    \centering
    \begin{subfigure}{0.48\textwidth}
        \centering
        \includegraphics[width=\linewidth]{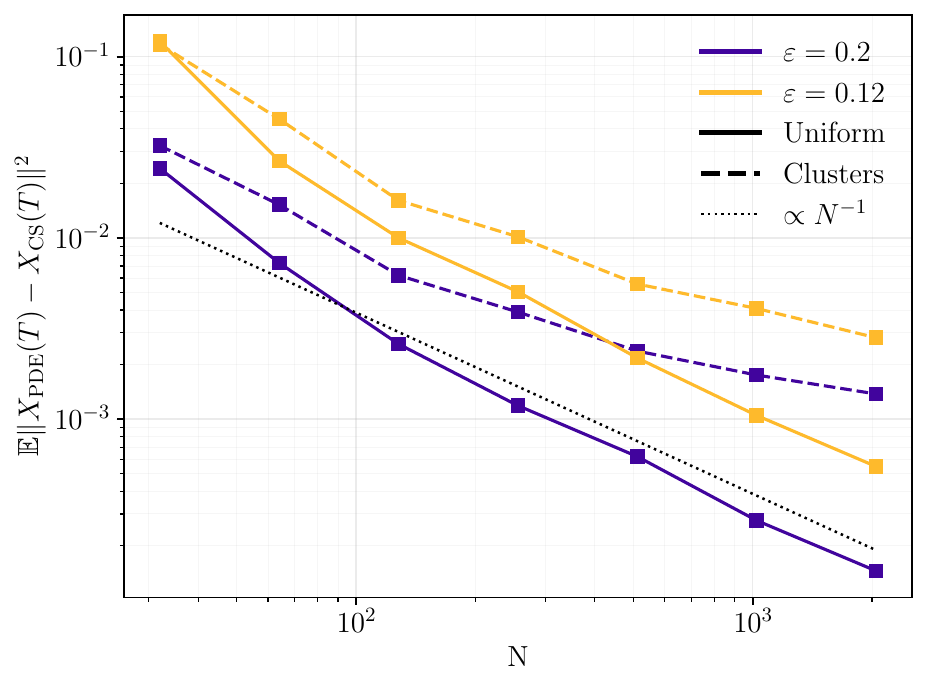}
        \caption{$d=1$}
        \label{fig: L2-error-vc-1d}
    \end{subfigure}\hfill
    \begin{subfigure}{0.48\textwidth}
        \centering
        \includegraphics[width=\linewidth]{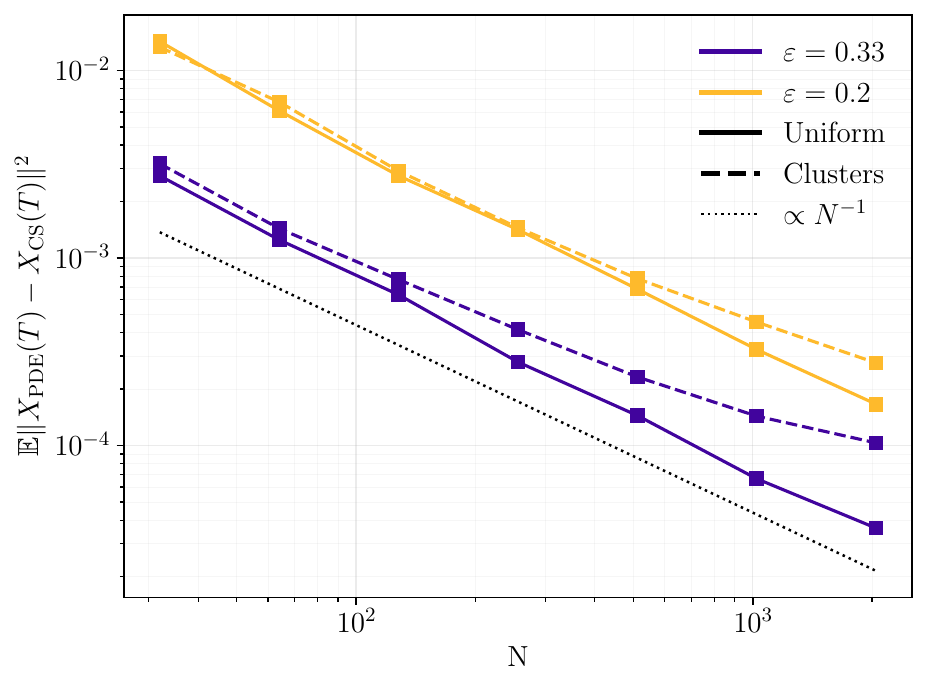}
        \caption{$d=2$}
        \label{fig: L2-error-vc-2d}
    \end{subfigure}\hfill
    \caption{Strong error between the particle system and reduced SPDE in the common noise setting for $d \in \{1, 2\}$. The error is computed for varying values of $N$ with two different choices of initial condition and two different choices of $\varepsilon$.}
    \label{fig: L2-error-vc}
\end{figure}

\subsection{Improved kinetic closure}\label{sec: numerics-improved-closure}

In Section \ref{sec: improved-kinetic-closure}, we derived an alternative formulation of the reduced SPDE with i.i.d.\ noise and claimed that it improved the model in the pre-flocking regime. We now substantiate this claim numerically. Since the proposed improvement lies in the drift of the $j$-SPDE, it is sufficient to perform experiments in the deterministic setting. Define the improvement factor for a quantity $f $ by
$$
I(f, t) \coloneqq \frac{\E{\norm{f(t) - f_\ve(t)}^2}}{\E{\lVert f^{(\mathrm{new})}(t) - f_\ve(t)\rVert^2}},
$$
where the superscript denotes the solution of the reduced SPDE with the improved kinetic closure \eqref{eq: improved-rho}-\eqref{eq: improved-k}. In particular, $I(f, t) > 1$ means the improved kinetic closure leads to a reduction in the error of quantity $f$ between the reduced SPDE and the particle system at time $t$. The expectation is due to the randomness in the initial condition since we set $\sigma = 0$ for our experiments. We fix $N=1,000$ and use $100$ realisations. The initial positions and velocities are sampled using the three-cluster strategy described in Section \ref{subsec: approx-quality}. The initial velocity spread $\E{((v_i(0) - \vbar(0))^2}$ is varied to shift how far into the pre-flocking regime we are. We set $\ve=0.125$, $T=2.0$ and discretise the domain uniformly into $256$ cells.

\begin{figure}[h!]
    \centering

    \begin{subfigure}{0.31\textwidth}
        \centering
        \includegraphics[width=\linewidth]{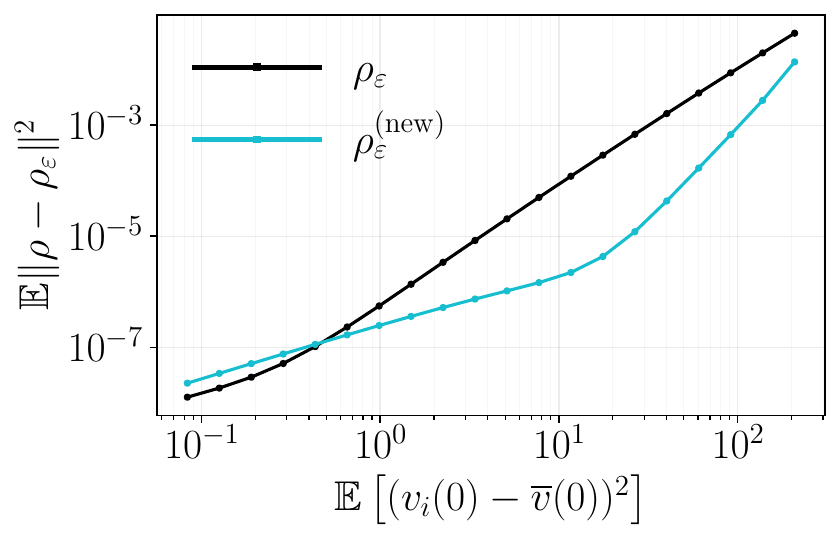}
        \caption{$\rho$: error}
        \label{fig: det-rho}
    \end{subfigure}
    \hfill
    \begin{subfigure}{0.31\textwidth}
        \centering
        \includegraphics[width=\linewidth]{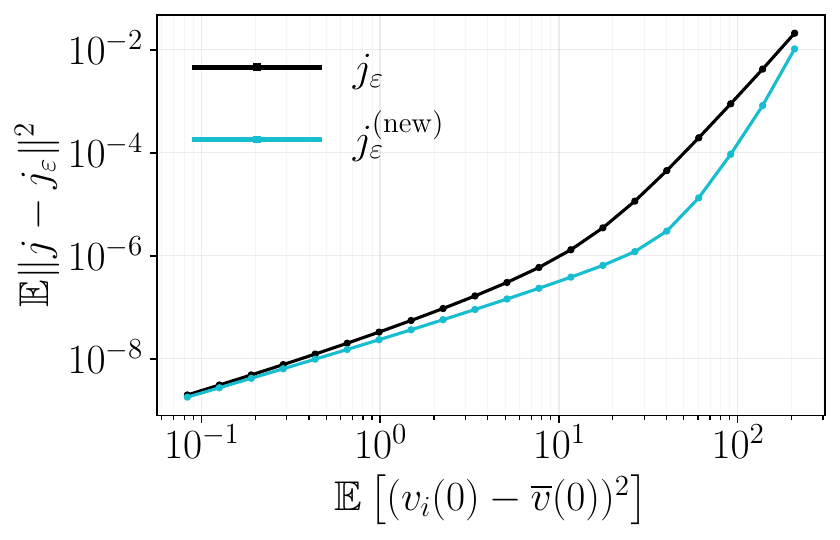}
        \caption{$j$: error}
        \label{fig: det-j}
    \end{subfigure}
    \hfill
    \begin{subfigure}{0.31\textwidth}
        \centering
        \includegraphics[width=\linewidth]{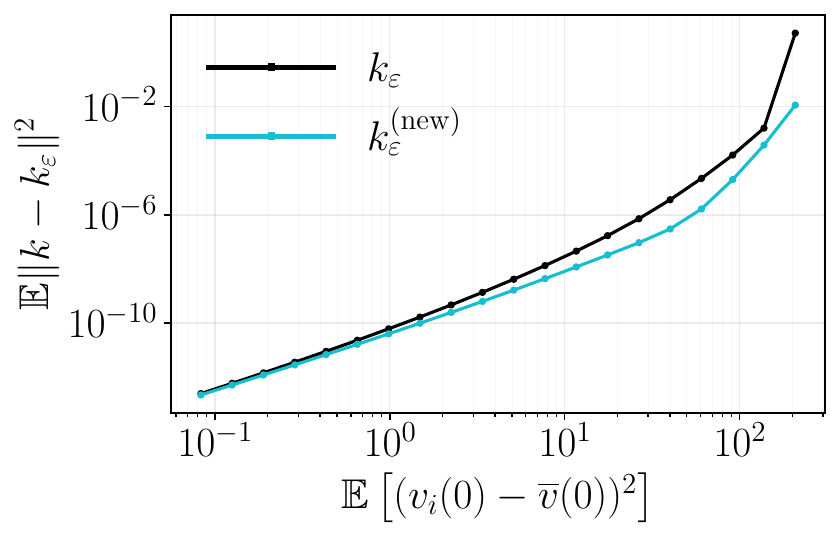}
        \caption{$k$: error}
        \label{fig: det-k}
    \end{subfigure}

    \vspace{0.75em}

    \begin{subfigure}{0.31\textwidth}
        \centering
        \includegraphics[width=\linewidth]{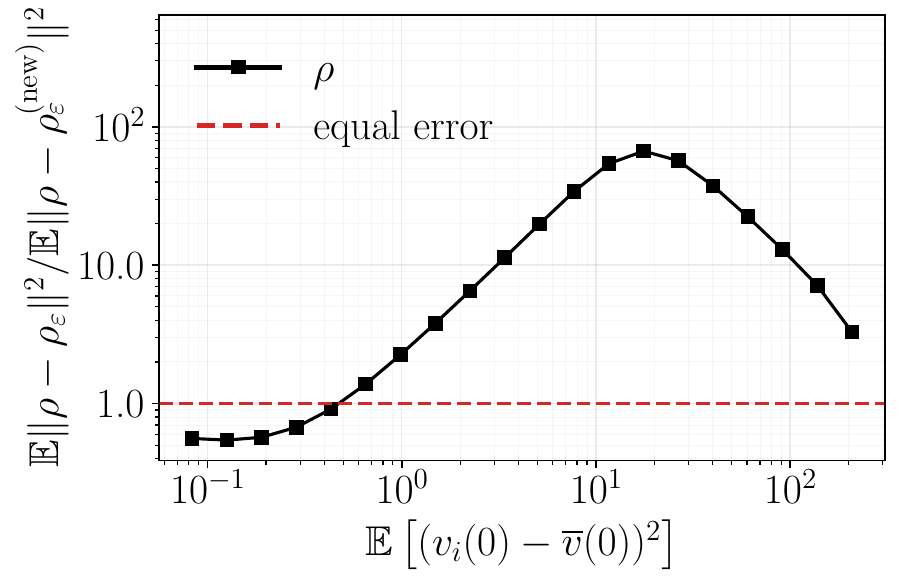}
        \caption{$\rho$: improvement factor}
        \label{fig: det-rho-improv}
    \end{subfigure}
    \hfill
    \begin{subfigure}{0.31\textwidth}
        \centering
        \includegraphics[width=\linewidth]{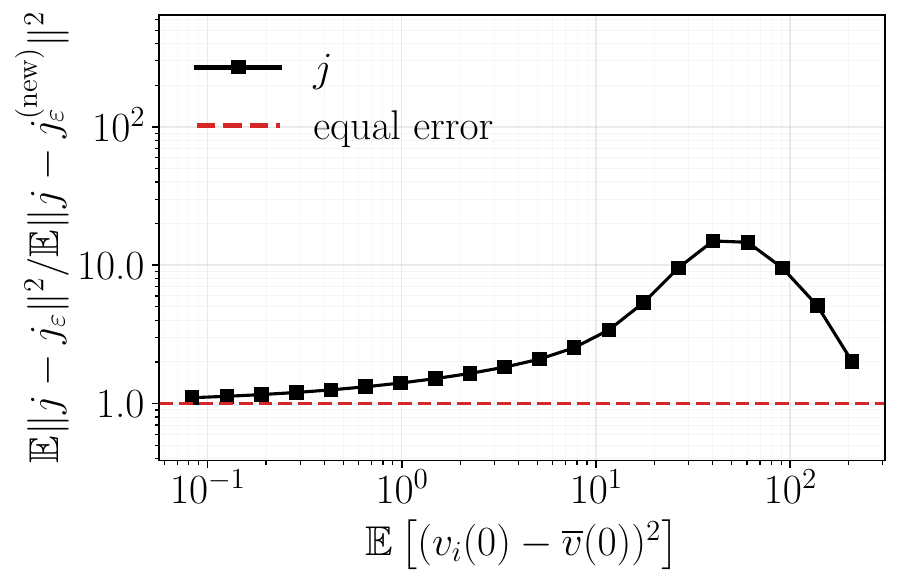}
        \caption{$j$: improvement factor}
        \label{fig: det-j-improv}
    \end{subfigure}
    \hfill
    \begin{subfigure}{0.31\textwidth}
        \centering
        \includegraphics[width=\linewidth]{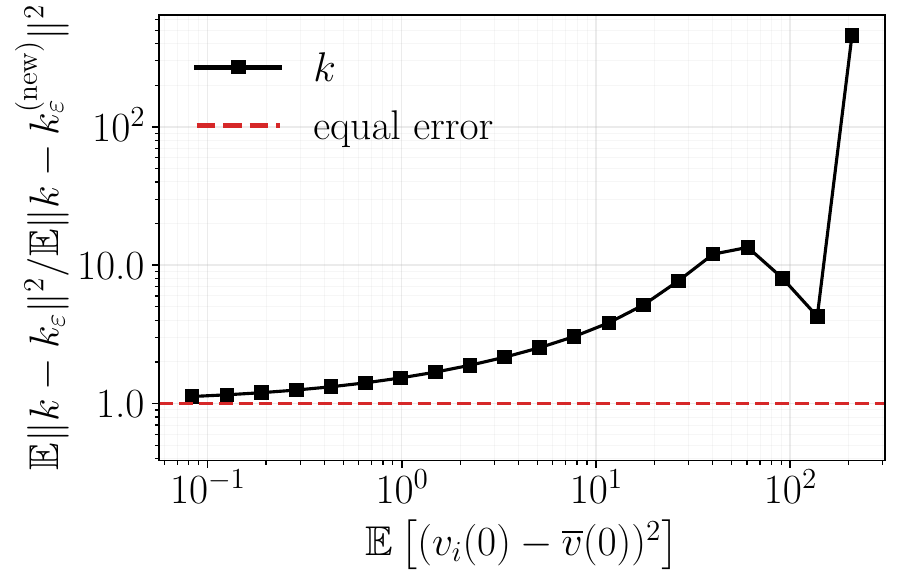}
        \caption{$k$: improvement factor}
        \label{fig: det-k-improv}
    \end{subfigure}

    \caption{Comparison of $L^2$ errors and associated improvement factors for the standard and improved SPDE models with i.i.d noise.}
    \label{fig: det-improved-closure}
\end{figure}

In Figure \ref{fig: det-improved-closure} we plot the results of our experiments. The top panels show the strong errors $\E{\norm{f(T) - f_\ve(T)}^2}$ and $\E{\norm{f^{(\mathrm{new})}(T) - f_\ve(T)}^2}$ for each of the densities $(\rho, j, k)$. The bottom panels show the associated improvement factor. For small velocity spreads (i.e.\ close to the flocking regime) we see comparable behaviour between the models. This is expected since the extra terms picked up in \eqref{eq: improved-j} are close to zero when $v_i(t) \approx \vbar(t)$. As the velocity spread increases, we see substantial reductions in error for the improved model. In particular, we see up to a $70$ times improvement in $\rho$, up to a $15$ times improvement in $j$, and up to a $10$ times improvement in $k$. For the largest velocity spreads, the behaviour of improved model degrades. We expect this is due to the error incurred in closing the $k$-SPDE (see Section \ref{sec: improved-kinetic-closure} for a more detailed discussion) becoming the dominant effect.

\subsection{Influence of noise}

On $\mathbb{R}^d$ when $r>0.5$ the deterministic CS model \eqref{eq: standard-cs} can fail to flock for certain initial conditions. Roughly speaking, this is because one can choose initial positions in say two clusters and then the interaction kernel can be too localised for particles in these two groups to interact. Thus, the system fails to flock globally. Interestingly, it is observed in \cite{cattiaux2018stochastic} that the addition of an environmental noise, for example in the sense of \eqref{eq: stochastic-cs-vc}, can induce global flocking in such situations. We now show that our reduced SPDE in the common noise setting also exhibits this behaviour.

\begin{remark}
    Note that this observation requires the domain to be $\mathbb{R}^d$ since this allows for groups of particles to diverge arbitrarily far apart whereas on $\btd$ flocking is unconditional for any $r \geq 0$ and initial configuration since particles will always meet eventually.
\end{remark}

\begin{figure}[h!]
    \centering
    \begin{subfigure}{0.48\textwidth}
        \centering
        \includegraphics[width=\linewidth]{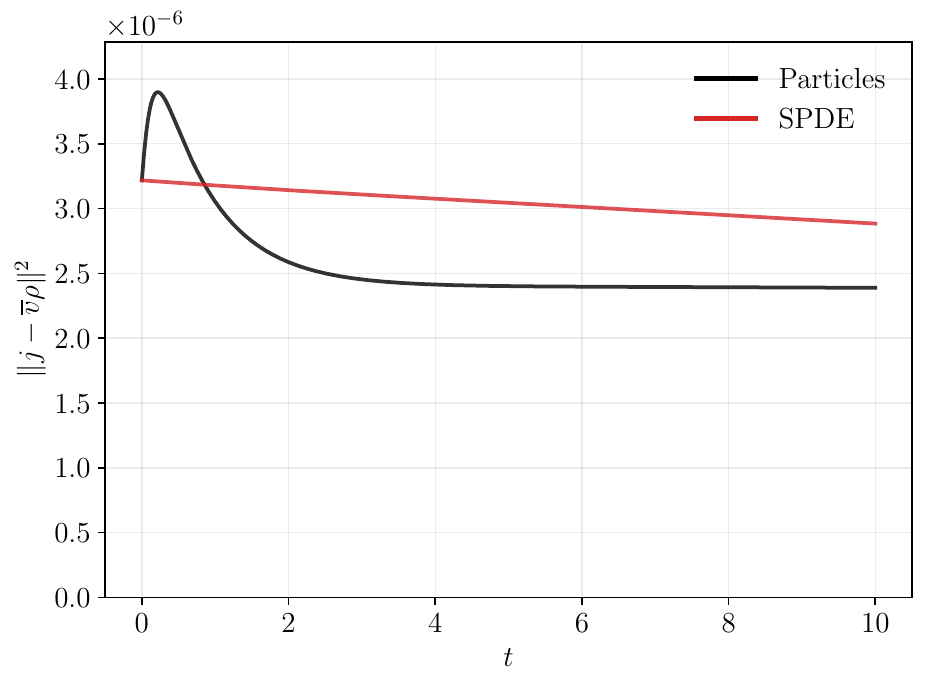}
        \caption{Non-flocking in the deterministic CS and reduced PDE dynamics}
        \label{fig:non-flock-deterministic}
    \end{subfigure}\hfill
    \begin{subfigure}{0.48\textwidth}
        \centering
        \includegraphics[width=\linewidth]{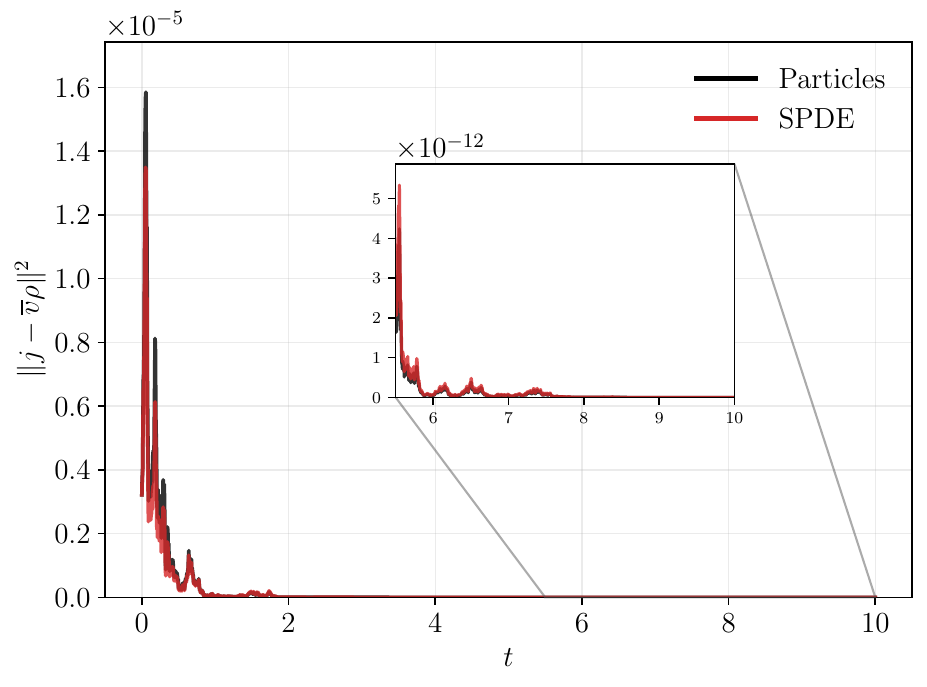}
        \caption{Recovery of flocking using sCS and reduced SPDE dynamics}
        \label{fig:flock-with-noise}
    \end{subfigure}
    \caption{Effect of environmental noise on driving flocking dynamics. (A): non-flocking in the deterministic setting. (B): recovery of flocking in the presence of environmental noise}
    \label{fig:non-flock-and-flock}
\end{figure}

We choose the interaction kernel \eqref{eq: kernel-definition} parameters to be $\lambda = 10$ and $r=4.0$. We consider $N=100$ particles. For the first $99$ particles, we sample their positions uniformly on $[-12, -10]$ and velocities uniformly on $[-0.1, 1.2]$. The final particle's position is sampled uniformly on $[10, 12]$ and velocity sampled uniformly on $[0.05, 0.15]$. We run the particle system on an unbounded domain up to $t=10.0$ with timestep $\Delta t = 0.005$, first without noise and then with $\sigma = 1.5$. The mollified densities $\rho_\ve, j_\ve$ and the SPDE initial configuration are generated using $\ve=0.125$, and the grid is discretised into $256$ uniform cells. In order to mimic an unbounded domain, we have chosen the parameters such that the dynamics occur on a spatial scale sufficiently far away from the endpoints of the domain. This ensures that the periodic wraparound in the implementation of the finite difference gradient has no effect. We also adjust the implementation of the convolution accordingly.

In Figure \ref{fig:non-flock-and-flock}\subref{fig:non-flock-deterministic}, we plot the evolution of $\norm{j(t) - \vbar \rho(t)}$ and $\norm{j_\ve(t) - \vbar \rho_\ve(t)}$ over time in the deterministic ($\sigma = 0$) setting. As expected, we see that the system doesn't flock. The discrepancy between the particle system and SPDE is expected since we are in the pre-flocking regime for the entire simulation. In Figure \ref{fig:non-flock-and-flock}\subref{fig:flock-with-noise}, we plot the same quantities for $\sigma = 1.5$ and see that both the particle system and reduced SPDE are now driven to flocking by the inclusion of an environmental noise. This is in line with the observations made in \cite{cattiaux2018stochastic}.   

\section{Conclusion and future work}\label{sec: conclusion}

In this work we have derived and analysed two general reduced SPDE models for stochastic Cucker--Smale particle dynamics that encompass environmental and i.i.d. type noises. Our models are computationally cheaper alternatives that provably preserve the flocking behaviour of the underlying particle system. To derive our model, we consider suitable empirical particle densities which depend on position and time, but feature velocity dependency in a multiplicative fashion. The models are then closed using the approximation that the particle velocities are near to alignment.

Our main theorems prove well-posedness and flocking of the reduced SPDE models in dimension $d=1$. Moreover, for an environmental noise we are able to quantify the discrepancy between the reduced SPDE and mollified particle system in a negative Sobolev norm for $d=1$. 
We then provide extensive numerical experiments to verify our theoretical results in $d=1$ and to explore the behaviour of our models in $d=2$. Furthermore, we introduced a method for improving the kinetic closure in the i.i.d.\ setting that utilises the hierarchy structure of the empirical particle densities.

One future avenue to explore is whether going further along the hierarchy entails modelling advantages in spite of the increased computational cost. We also plan to study the flocking behaviour for a reduced SPDE with Dean--Kawasaki type noise (corresponding to pure i.i.d.\ Brownian motions at the particle level). This is not covered by the present analysis due to the considerations in Remark \ref{remark: dk-noise}. Finally, it would be interesting to explore if our reduced models can act as computationally efficient tools for the optimal control of Cucker--Smale particle systems.

\medskip

\textbf{AI statement.} 
ChatGPT 6 (OpenAI) was used to assist with the proof of Theorem \ref{thm:global_H1}, specifically in reviewing the proof, checking mathematical arguments and estimates, clarifying technical justifications, and improving the presentation and LaTeX formatting. In particular, the assistance included clarifying the stochastic differential equation well-posedness argument, the Hilbert–Schmidt estimates, and the use of the Burkholder–Davis–Gundy inequality in the energy estimates. ChatGPT 6 (OpenAI) was also used for general proofreading and assistance with the numerical experiments. All suggested modifications were reviewed by the authors, who retain full responsibility for the correctness of the results.
\medskip

\textbf{Acknowledgments.} The authors wish to thank Tony Shardlow for several insightful discussions.
This work has been partially funded by the Deutsche Forschungsgemeinschaft (DFG) under
Germany’s Excellence Strategy through grant no. EXC-2046 The Berlin Mathematics Research Center MATH+
(project no. 390685689). KA is supported by a scholarship from the EPSRC Centre for Doctoral Training in Statistical Applied Mathematics at Bath (SAMBa), under the project EP/S022945/1.  ADj gratefully acknowledges funding by Daimler and Benz Foundation as part of the
scholarship program for junior professors and postdoctoral researchers.

\bibliographystyle{alpha}
\bibliography{ref.bib}

\appendix
\section{Well-posedness for the common noise case}\label{app: well-posedness}

In this Appendix we prove Theorem \ref{thm: well-posedness}.
Recall that we work on a fixed probability space $(\Omega, \mathcal{F}, \mathbb{P})$ with a complete and right continuous filtration $(\mathcal{F}_t)_{t \geq 0}$ that the Brownian motion $B(t)$ is adapted to. For $k \in \mathbb{N}_0$, define the spaces $\mathcal{W}^k \coloneqq H^k \times H^k$ where $H^0 = L^2$ with weighted norms $\norm{(f, g)}_{\mathcal{W}^k}^2 = \vbar^2\norm{f}_{H^k}^2 + \norm{g}_{H^k}^2$. The weight $\vbar \coloneqq \vbar(0)$ is the alignment velocity defined in \eqref{eq: vbar-definition}.

\subsection{Constant kernel}

When the kernel is constant, \eqref{eq:spde-vanishing-common-1d} reduces to
$$
    \begin{cases}
        \dd \rho = -\partial_x j \dd t, \\
        \dd j = [-C_a(j-\vbar \rho) - \vbar^2 \partial_x \rho] \dd t + \sigma(j-\vbar \rho) \dd B_t.
    \end{cases}
$$
Write this as a linear evolution equation on $\mathcal{W}^1$, that is, set $Z(t) = (\rho(t), j(t))$ and consider
\begin{equation}\label{eq:spde-vanishing-common-1d-pure-form-constant}
\begin{cases}
    \dd Z(t) = \sbrack{A Z(t) + F(Z(t))} \dd t + B(Z(t)) \dd W(t), \\
    Z(0) = Z_0,
\end{cases}
\end{equation}
 where $AZ(t) = (-\partial_x j(t), -\vbar^2 \partial_x \rho(t))$, $F(Z(t)) = (0, -C_a(j(t) - \vbar \rho(t))$, $B(Z(t)) = (0, \sigma(j - \vbar \rho))$, $W(t)$ is a noise process whose definition we defer, and $Z_0 = (\rho_0, j_0)$ is a $\mathcal{F}_0$-measurable random variable. Moreover, we define $\mathcal{D(A)} \coloneqq \mathcal{W}^2$.

\begin{lemma}[Drift term]\label{lem: semigroup-analysis}
    The operator $A: \mathcal{D(A)} \subset \mathcal{W}^1 \rightarrow \mathcal{W}^1$ generates a $C^0$ semigroup of contractions $\{S(t)\}_{t \geq 0}$ such that $\norm{S(t)}_{\mathrm{op}} \leq 1$ for all $t$. Furthermore, the non-linearity $F\colon \mathcal{W}^1\rightarrow \mathcal{W}^1$ is globally Lipschitz.
\end{lemma}

\begin{proof}
    This is proved by verifying the assumptions of the Hille-Yosida theorem as stated in \cite[Theorem 3.1]{pazy2012semigroups}. See \cite[Lemma 4.2]{cornalba2019regularized} for details. The Lipschitz property of $F$ is immediate. 
\end{proof}

We now deal with the noise term. Recall from \eqref{eq:spde-vanishing-common-1d}, that our noise is a single Brownian motion $\tilde B(t)$ (written this way to avoid confusion with the stochastic integrand) that acts only on the $j$-component. For compatibility with the general form given in \eqref{eq:spde-vanishing-common-1d-pure-form-constant}, we define $W(t)$ via
$W(t) = (0, 1) \tilde B(t)$. 
That is, $W(t)$ is a $Q$-Wiener process where $Q$ is the projection onto $e_2=(0,1)$. Then the stochastic integrand is a bounded linear operator $B: \mathcal{W}^1 \rightarrow L_2^0$ defined via $B((\rho, j))(a, b) = (0, \sigma(j-\vbar \rho)b)$. Here, $L_2^0$ is the space of Hilbert-Schmidt operators ${Q}^{1/2}(\mathbb{R}^2) \rightarrow \mathcal{W}^1$.

\begin{lemma}\label{lem: HS-analysis}
    The stochastic integrand $B$ is a bounded linear operator $\mathcal{W}^1 \rightarrow L_2^0$.
\end{lemma}

\begin{proof}
    Linearity of $B$ is clear. Now take $(\rho, j) \in \mathcal{W}^1$ and compute
    $$
    \norm{B((\rho, j))}_{L_2^0}^2 = \norm{B((\rho, j))(0, 1)}_{\mathcal{W}^1}^2 = \norm{\sigma(j-\vbar \rho)}_{\mathcal{W}^1}^2 \leq C(\sigma, \vbar) \norm{(\rho, j)}_{\mathcal{W}^1}^2.
    \qedhere
    $$
\end{proof}

We now have all the ingredients to apply standard existence and uniqueness results.

\begin{proposition}\label{prop: mild-solution-constant}
    Fix $T>0$ and assume $Z_0$ is $\mathcal{F}_0$-measurable. Then equation \eqref{eq:spde-vanishing-common-1d} admits a unique mild solution on $[0, T]$ with a continuous modification.
\end{proposition}

\begin{proof}
    Use Lemmas \ref{lem: semigroup-analysis}, and \ref{lem: HS-analysis} then apply \cite[Theorem 7.2 (i), (ii)]{da2014stochastic}.
\end{proof}

Lemmas \ref{lem: semigroup-analysis}, \ref{lem: HS-analysis}, together with \cite[Theorem 6.5]{da2014stochastic} imply that the mild solution to \eqref{eq:spde-vanishing-common-1d} is a weak solution. 
Furthermore, the assumptions of \cite[Propositions F.0.4]{prevot2007concise} can be verified in light of Lemma \ref{lem: HS-analysis}, and so the weak solution is upgraded to a strong solution on $[0, T]$. Finally, a bootstrapping argument extends the solution to $[0, \infty)$.

\subsection{Full kernel}

We now consider the full interaction kernel and first show how one obtains a local $\mathcal{W}^1$-valued mild solution. Write \eqref{eq:spde-vanishing-common-1d} as a semilinear evolution equation on $\mathcal{W}^1$, that is, we let $Z(t) = (\rho(t), j(t))$ and consider
\begin{equation}\label{eq:spde-vanishing-common-1d-pure-form}
\begin{cases}
    \dd Z(t) = \sbrack{A Z(t) + \tilde F(Z(t))} \dd t + B(Z(t)) \dd W(t), \\
    Z(0) = Z_0.
\end{cases}
\end{equation}
Here $A, B, W, Z_0$ are defined as in the previous section but the non-linearity is now given by $\tilde F(Z(t)) = (0, \rho(a * j) - j(a * \rho))$. Once again, we define $\mathcal{D(A)} \coloneqq \mathcal{W}^2$. The semigroup analysis and noise analysis are unchanged so we only need to deal with the non-linearity. 

\begin{lemma}\label{lem: non-linearity-analysis}
    The non-linearity $\tilde F: \mathcal{W}^1 \rightarrow \mathcal{W}^1$ is locally bounded and locally Lipschitz.
\end{lemma}

\begin{proof}
    To prove local boundedness, choose $(\rho, j)$ in a $\mathcal{W}^1$-ball of radius $n$. It is then straightforward to estimate $\norm{\tilde F((\rho, j))}_{\mathcal{W}^1}^2 \leq C(a, n)$ using Young's convolution inequality.

    To prove $\tilde F$ is locally Lipschitz, choose $(\rho_1, j_1)$ and $(\rho_2, j_2)$ in a $\mathcal{W}^1$-ball of radius $n$. We then estimate
    $$
    \norm{\tilde F((\rho_1, j_1)) - \tilde F((\rho_2, j_2))}_{\mathcal{W}^1}^2 \leq \norm{\rho_1(a * j_1) - \rho_2(a*j_2)}_{H^1}^2 + \norm{j_1(a*\rho_1) - j_2(a*\rho_2)}_{H^1}^2.
    $$
    By definition of the $H^1$ norm, we obtain
    \begin{multline*}
        \norm{\rho_1(a*j_1) - \rho_2(a*j_2)}_{H^1}^2 \leq \norm{\rho_1(a*j_1) - \rho_2(a*j_2)}_{L^2}^2 \\
        + \norm{\partial_x\rho_1(a*j_1) - \partial_x\rho_2(a*j_2)}_{L^2}^2 + \norm{\rho_1(a*\partial_xj_1) - \rho_2(a*\partial_xj_2)}_{L^2}^2.
    \end{multline*}
    Adding and subtracting $\rho_1(a*j_2)$ in the first term, and using Young's convolution inequality gives
    \begin{align*}
        \norm{\rho_1(a*j_1) - \rho_2(a*j_2)}_{L^2}^2 &\leq \norm{\rho_1(a*(j_1-j_2))}_{L^2}^2 + \norm{(\rho_1-\rho_2)(a*j_2)}_{L^2}^2 \\
        &\leq C(n, a) \norm{(\rho_1, j_1) - (\rho_2, j_2)}_{\mathcal{W}^1}^2.
    \end{align*}
    Similar estimates bound the latter two terms in the same fashion. By symmetry, we have also dealt with $\norm{j_1(a*\rho_1) - j_2(a*\rho_2)}_{H^1}^2$. Putting everything together gives
    $$
    \norm{\tilde F((\rho_1, j_1)) - \tilde F((\rho_2, j_2))}_{\mathcal{W}^1}^2 \leq C(n, a)\norm{(\rho_1, j_1) - (\rho_2, j_2)}_{\mathcal{W}^1}^2.   \qedhere
    $$
\end{proof}

We now have all the ingredients to apply standard existence and uniqueness results.

\begin{proposition}
    Equation \eqref{eq:spde-vanishing-common-1d-pure-form} admits a unique local-in-time $\mathcal{W}^1$-valued mild solution. Moreover, the solution is c\'adl\'ag.
\end{proposition}

\begin{proof}
    Existence and uniqueness follows from Lemmas \ref{lem: semigroup-analysis}, \ref{lem: non-linearity-analysis}, \ref{lem: HS-analysis}, and \cite[Theorem 4.5]{tappe2012some}. The solution is c\'adl\'ag due to \cite[Remark 4.6]{tappe2012some}. 
\end{proof}

Note that the previous Lemmas can easily be adapted to the $\mathcal{W}^0$ (i.e. $L^2 \times L^2$) case to give a local well-posedness result in this setting too.

\begin{proposition}
    Equation \eqref{eq:spde-vanishing-common-1d-pure-form} admits a unique local-in-time $\mathcal{W}^0$-valued mild solution. Moreover, the solution is c\'adl\'ag.
\end{proposition}

Since $\mathcal{W}^1 \subset \mathcal{W}^0$, (local) uniqueness implies that the solutions coincide where they both exist. We now define stopping times associated to these solutions.

\begin{definition}
    For $K, R >0$, define stopping times
    $$
    \tau_{0, K} \coloneqq 
    \begin{cases}
        \inf \cbrack{t \geq 0 : \norm{(\rho(t), j(t))}_{\mathcal{W}^0} \geq K}, \quad &\text{if infimum taken over non-empty set}, \\
        \infty, \quad &\text{otherwise},
    \end{cases}
    $$
    and
    $$
    \tau_{1, R} \coloneqq 
    \begin{cases}
        \inf \cbrack{t \geq 0 : \norm{(\rho(t), j(t))}_{\mathcal{W}^1} \geq R}, \quad &\text{if infimum taken over non-empty set}, \\
        \infty, \quad &\text{otherwise},
    \end{cases}
    $$
    We also define $\tau_{1, \infty} \coloneqq \lim_{R \rightarrow \infty}\tau_{1, R}$.
\end{definition}

In the main text our estimates rely on an $L^2 \times L^2$ stopping time. However, it is not clear that the $H^1 \times H^1$-valued solution is finite up to this stopping time. In the next lemma we prove that this is, in fact, the case.

\begin{lemma}\label{lem: H1-control-via-L2}
    For any $K>0$,
    $$
    \mathbb{P}(\{\tau_{1, \infty} \leq \tau_{0, K}\})=0.
    $$
    In words, with probability one the $H^1 \times H^1$ norm is finite where the $L^2 \times L^2$ norm is uniformly bounded.
\end{lemma}

\begin{proof}
    Fix $K, R >0$, a deterministic time $T>0$, and set $\theta \coloneqq \tau_{0, K} \wedge \tau_{1, R}$. For $t \in [0, T]$, estimate
    \begin{multline*}
    \E{\norm{Z(t \wedge \theta)}_{\mathcal{W}^1}^2} \leq 3 \E{\norm{S(t \wedge \theta)Z_0}_{\mathcal{W}^1}^2}
    + 3\E{\norm{\int_0^{t \wedge \theta}S(t \wedge \theta-s)F(Z(s)) \dd s}_{\mathcal{W}^1}^2} \\
    + 3\E{\norm{\int_0^{t \wedge \theta}S(t \wedge \theta-s)B(Z(s)) \dd W(s)}_{\mathcal{W}^1}^2} =: I_1 + I_2 + I_3.
    \end{multline*}
    For $I_1$, we simply use contractivity of the semigroup to get
    $$
    I_1 \leq 3 \E{\norm{Z_0}_{\mathcal{W}^1}^2}.
    $$
    For \(I_2\), by definition we have
    \[
    I_2
    =
    3\mathbb{E}\left[
    \norm{
    \int_0^{t\wedge\theta}
    S(t\wedge\theta-s)F(Z(s))\dd s
    }_{\mathcal W^1}^{2}
    \right] = 3\mathbb{E}\left[
    \norm{
    \int_0^t
    \mathbf{1}_{[0,\theta]}(s)
    S(t\wedge\theta-s)F(Z(s))\dd s
    }_{\mathcal W^1}^{2}
    \right].
    \]
    By the definition of \(\tau_{0,K}\), for \(s<\theta\) we have $\|Z(s)\|_{\mathcal W^0}<K$ therefore using the local growth estimate for \(F\) gives
    \[
    \mathbf{1}_{[0,\theta]}(s)
    \|F(Z(s))\|_{\mathcal W^1}^{2}
    \leq
    C(K)\,
    \mathbf{1}_{[0,\theta]}(s)
    \|Z(s)\|_{\mathcal W^1}^{2}.
    \]
    Using the triangle inequality for the Bochner integral and the
    contractivity of the semigroup \(S\) on \(\mathcal W^1\), we obtain
    \[
    \begin{aligned}
    I_2
    &\leq
    3\mathbb{E}\left[
    \left(
    \int_0^t
    \mathbf{1}_{[0,\theta]}(s)
    \|F(Z(s))\|_{\mathcal W^1}\dd s
    \right)^2
    \right].
    \end{aligned}
    \]
    Applying the Cauchy--Schwarz inequality in time yields
    $$
    I_2
    \leq
    3t
    \int_0^t
    \mathbb{E}\left[
    \mathbf{1}_{[0,\theta]}(s)
    \|F(Z(s))\|_{\mathcal W^1}^{2}
    \right]\dd s \leq
    3TC(K)
    \int_0^t
    \mathbb{E}\left[
    \mathbf{1}_{[0,\theta]}(s)
    \|Z(s)\|_{\mathcal W^1}^{2}
    \right]\dd s.
    $$
    We conclude that
    \[
    I_2
    \leq
    C(K, T)
    \int_0^t
    \mathbb{E}\left[
    \mathbf{1}_{[0,\tau_{0, K} \wedge \tau_{1, R}]}(s)
    \|Z(s)\|_{\mathcal W^1}^{2}
    \right]\dd s
    \leq 
    C(K,T)
    \int_0^t
    \mathbb{E}\left[
    \|
    Z(s\wedge\tau_{0,K}\wedge\tau_{1,R})
    \|_{\mathcal W^1}^{2}
    \right]\dd s.
    \]  
For $I_3$, we first define the stopped stochastic convolution
$$
X(r) \coloneqq 
\int_0^r
S(r-s)\mathbf{1}_{\{s<\theta\}}B(Z(s))\,\dd W(s),
\qquad 0\leq r\leq t.
$$
It follows that
$$
I_3
=
3\mathbb{E}\left[
\|X(t\wedge\theta)\|_{\mathcal W^1}^{2}
\right]
\leq
3\mathbb{E}\left[
\sup_{0\leq r\leq t}
\|X(r)\|_{\mathcal W^1}^{2}
\right].
$$
Since \(\mathcal W^1\) is a Hilbert space and 
\((S(t))_{t\geq0}\) is a strongly continuous contraction semigroup on \(\mathcal W^1\), then the maximal inequality for stochastic convolutions \cite{van2011maximal} gives
\[
\mathbb{E}\left[
\sup_{0\leq r\leq t}
\norm{
\int_0^r S(r-s)\Phi(s)\,\dd W(s)
}_{\mathcal W^1}^{2}
\right]
\leq
C
\mathbb{E}\left[
\int_0^t
\|\Phi(s)\|_{\mathcal L_2(U,\mathcal W^1)}^{2}\,\dd s
\right]
\]
for every predictable process $\Phi\in L^2\bigl(\Omega\times[0,t]; \mathcal L_2(U,\mathcal W^1)\bigr)$. Applying this inequality with
$\Phi(s) = \mathbf{1}_{\{s<\theta\}}B(Z(s))$ gives
\[
I_3
\leq
C\mathbb{E}\left[
\int_0^t
\mathbf{1}_{\{s<\theta\}}
\|B(Z(s))|_{\mathcal L_2(U,\mathcal W^1)}^{2}
\,\dd s
\right]=
C\int_0^t
\mathbb{E}\left[
\mathbf{1}_{\{s<\theta\}}
\|B(Z(s))|_{\mathcal L_2(U,\mathcal W^1)}^{2}
\right]\dd s.
\]
By the definition of \(\tau_{0,K}\), on the set
\(\{s<\theta\}\) we have $\norm{Z(s)}_{\mathcal W^0}<K$ from which it follows that
\[
\mathbf{1}_{\{s<\theta\}}
\norm{B(Z(s))}_{\mathcal L_2(U,\mathcal W^1)}^{2}
\leq
C(K)\mathbf{1}_{\{s<\theta\}}
\norm{Z(s)}_{\mathcal W^1}^{2},
\]
using the local growth of $B$. Consequently, 
\[
I_3
\leq
C(K)
\int_0^t
\mathbb{E}\left[
\mathbf{1}_{\{s<\tau_{0, K} \wedge \tau_{1, R}\}}
\|Z(s)\|_{\mathcal W^1}^{2}
\right]\dd s \leq C(K)
\int_0^t
\mathbb{E}\left[
\|
Z(s\wedge\tau_{0,K}\wedge\tau_{1,R})
\|_{\mathcal W^1}^{2}
\right]\dd s.
\]
Altogether, we have
    $$
    \E{\norm{Z(t \wedge \tau_{0, K} \wedge \tau_{1, R})}_{\mathcal{W}^1}^2} \leq 3 \E{\norm{Z_0}_{\mathcal{W}^1}^2} + C(K, T) \int_0^t \E{\norm{Z(s \wedge \tau_{0, K} \wedge \tau_{1, R})}_{\mathcal{W}^1}^2} \dd s.
    $$
    Recalling that $\E{\norm{Z_0}_{\mathcal{W}^1}^2} < \infty$ by assumption and applying Gr\"onwall's inequality gives
    $$
    \E{\norm{Z(t \wedge \tau_{0, K} \wedge \tau_{1, R})}_{\mathcal{W}^1}^2} \leq 3 \E{\norm{Z_0}_{\mathcal{W}^1}^2} e^{C(K, T)}.
    $$
    By the stopping time definitions we have
    $$
    \mathbf{1}_{\{\tau_{1, R} \leq t \wedge \tau_{0, K}\}} \leq \frac{\norm{Z(t \wedge \tau_{0, K} \wedge \tau_{1, R}}_{\mathcal{W}^1}^2}{R^2},
    $$
    for any $R, K > 0$ and $t \in [0, T]$. Taking expectations gives
    $$
    \mathbb{P}(\{\tau_{1, R} \leq t \wedge \tau_{0, K}\}) \leq \frac{\E{\norm{Z(t \wedge \tau_{0, K} \wedge \tau_{1, R}}_{\mathcal{W}^1}^2}}{R^2} \leq 3 \E{\norm{Z_0}_{\mathcal{W}^1}^2}\frac{e^{C(K, T)}}{R^2}.
    $$
    By monotonicity of the stopping times $(\tau_{1, R})_{R > 0}$, we can pass to the limit to obtain
    $$
    \mathbb{P}(\{\tau_{1, \infty} \leq t \wedge \tau_{0, K}\}) = 0, \quad t \in [0, T].
    $$
    Take a sequence $(t_n)_{n \in \mathbb{N}} \uparrow T$ to extend the estimate to
    $$
    \mathbb{P}(\{\tau_{1, \infty} \leq T \wedge \tau_{0, K}\}) = 0, \quad T>0.
    $$
    An analogous argument then removes the dependence on $T$ to give the desired result.
\end{proof}

We now show existence of a continuous modification to our local solution using a truncation argument.

\begin{lemma}\label{lem: W0-cts-modification}
    For each $K>0$, the $\mathcal{W}^0$-valued solution has a continuous modification up to $\tau_{0, K}$.
\end{lemma}

\begin{proof}
    For any $\varepsilon>0$, define a smooth cut-off function $\varphi_{\varepsilon}: [0, \infty) \rightarrow [0, 1]$ by
    $$
    \varphi_\varepsilon(x) = 
    \begin{cases}
        1, \quad &x \leq K^2, \\
        \text{smooth decay}, \quad &K^2 < x \leq K^2+\varepsilon, \\
        0, \quad &x > K^2 + \varepsilon.
    \end{cases}
    $$
    We also define $\hat \varphi: \mathcal{W}^0 \rightarrow \mathbb{R}_+$ by $\hat \varphi((\rho, j)) = \norm{(\rho, j)}_{\mathcal{W}^0}^2$, and set $\phi = \varphi_{\varepsilon} \circ \hat \varphi$. Then the truncated non-linearity $\phi(Z(t))F(Z(t))$ vanishes for $\norm{Z(t)}_{\mathcal{W}^0} > K + \varepsilon$ but coincides with $F(Z(t))$ up to $\tau_{0, K}$. For fixed $K>0$, one can argue similarly to the proof of Lemma \ref{lem: non-linearity-analysis} to show that $\phi(Z(t))F(Z(t))$ is globally Lipschitz. Combined with Lemmas \ref{lem: semigroup-analysis} and \ref{lem: HS-analysis}, we can apply \cite[Theorem 7.2]{da2014stochastic} to obtain a unique global $\mathcal{W}^0$-valued solution with a continuous modification. Finally, since the true and truncated SPDEs coincide up to $\tau_{0, K}$, so must their solutions by (local) uniqueness.
\end{proof}

\begin{corollary}
    For each $K>0$, the $\mathcal{W}^1$-valued solution has a continuous modification up to $\tau_{0, K}$.
\end{corollary}

\begin{proof}
    This follows from the inclusion $\mathcal{W}^1 \subset \mathcal{W}^0$, (local) uniqueness, Lemma \ref{lem: H1-control-via-L2} and Lemma \ref{lem: W0-cts-modification}.
\end{proof}

Finally, we need to upgrade our local $\mathcal{W}^1$-valued mild solution to a local $\mathcal{W}^1$-valued strong solution on $[0, \tau_{0, K}]$ for any $K>0$. To do so, we fix an arbitrary final time $T>0$ and $K>0$. We then consider the stopped solution $Z(t \wedge \tau_{0, K})$ for $t \in [0, T]$. Since we now have a deterministic time horizon, we can check the criteria of \cite[Propositions F.0.4, F.0.5]{prevot2007concise}. These follow from continuity and adaptedness (hence predictability) of the solution, estimates similar to our previous analysis of $\tilde F, B$, and that the $\mathcal{W}^1$-norm doesn't blow up on $[0, \tau_{0, K}]$.

This concludes the proof of Theorem \ref{thm: well-posedness}.

 \section{SPDE flocking in the case of common noise}\label{sec: proofs}

This section is devoted to proving SPDE flocking behaviour in the case when particles have a common noise. We first prove Theorem \ref{thm: constant-kernel-flocking} and then, via a sequence of auxiliary lemmas, prove Proposition \ref{thm:p=2-stopping-time-bound} and Theorem \ref{thm: stopping-time-exp-decay}.

\subsection{SPDE flocking with a constant kernel}\label{subsec_SPDE_constant_kernel_env_noise}

\begin{proof}[Proof of Theorem \ref{thm: constant-kernel-flocking}]
    By Theorem \ref{thm: well-posedness}, the SPDE admits a continuous, global-in-time $H^1 \times H^1$-valued mild solution. We first write
    $$
    \dd (j - \vbar \rho) = \sbrack{\rho(a * j) - j(a*\rho) - \vbar^2 \frac{\partial \rho}{\partial x} + \vbar \frac{\partial j}{\partial x}} \dd t + \sigma(j - \vbar \rho) \dd B(t),
    $$
    and then apply It\^o's formula to get 
    $$
    \dd \norm{j-\vbar \rho}^2 = 2\abrack{j - \vbar \rho, d(j-\vbar \rho)} + \sigma^2 \abrack{j-\vbar \rho, j-\vbar \rho} \dd t.
    $$
    For the first term, we write
    \begin{align*}
        &2\abrack{j - \vbar \rho, d(j-\vbar \rho)} \\
        &= 2\abrack{j - \vbar \rho, \sbrack{\rho(a * j) - j(a*\rho) - \vbar^2 \frac{\partial \rho}{\partial x} + \vbar \frac{\partial j}{\partial x}} \dd t + \sigma(j - \vbar \rho) \dd B(t)} \\
        &= 2\abrack{j - \vbar \rho, \sbrack{\rho(a * j) - j(a*\rho) - \vbar^2 \frac{\partial \rho}{\partial x} + \vbar \frac{\partial j}{\partial x}} \dd t} + 2\sigma \norm{j - \vbar \rho}^2 \dd B(t).
    \end{align*}
    Since $a \equiv C_a$ is a constant, we have
    $$
    -j(a * \rho) = -jC_a \int_\bt \rho(x, t) \dd x = -jC_a m_\rho(0) = -jC_a,
    $$
    by conservation of mass and our choice that $m_\rho = 1$. Similarly,
    $$
    \rho(a * j) = \rho C_a \int_\bt j(x, t) \dd x = \rho C_a \int_\bt m_j(0) = \rho C_a \vbar,
    $$
    using the relationship $m_j = \vbar m_\rho = \vbar$ (see Proposition \ref{prop: conservation-of-mass}). We then apply integration by parts to see that 
    $$
    \abrack{j - \vbar \rho, -\vbar^2 \frac{\partial \rho}{\partial x} + \vbar \frac{\partial j}{\partial x}} = 0.
    $$
    Together, these computations give
    $$
    2\abrack{j - \vbar \rho, d(j-\vbar \rho)} = 2C_a \abrack{j - \vbar \rho, \vbar \rho - j} \dd t + 2 \sigma \norm{j-\vbar \rho}^2 \dd B(t),
    $$
    and therefore,
    $$
    \dd \norm{j - \vbar \rho}^2 = (\sigma^2 - 2C_a) \norm{j-\vbar \rho}^2 \dd t + 2 \sigma \norm{j-\vbar \rho}^2 \dd B(t).
    $$
    Define $X(t) \coloneqq \norm{j(t) - \vbar \rho(t)}^2$ so that the flocking quantity solves the following SDE:
    $$
    \dd X(t) = (\sigma^2 - 2C_a)X(t) \dd t + 2 \sigma X(t) \dd B(t).
    $$
    This is a Geometric Brownian motion so we can write down the unique solution
    $$
    X(t) = X(0) \exp \sbrack{(\sigma^2 - 2C_a)t + 2 \sigma B(t) - \frac{1}{2}(2 \sigma)^2 t}.
    $$
    To get almost sure flocking, we take logs and divide by $t$ to get
    $$
    \frac{\log X(t)}{t} = \frac{\log X(0)}{t} - (\sigma^2 + 2C_a) + 2\sigma \frac{B(t)}{t}.
    $$
    Taking $t \rightarrow \infty$ gives
    $$
    \lim_{t \rightarrow \infty} \frac{\log X(t)}{t} = - (\sigma^2 + 2C_a) < 0,
    $$
    since $B(t)/t \rightarrow 0$ as $t \rightarrow \infty$. We therefore conclude that $X(t)$ and thus $\norm{j(t) - \vbar \rho(t)} \rightarrow 0$ almost surely as $t \rightarrow \infty$ for any choice of $\sigma > 0$. Additionally, we can take the expected value of the expression for $X(t)$ to get
    $$
    \E{X(t)} = \E{X(0)} \exp \sbrack{(\sigma^2 - 2C_a)t}.
    $$
    This follows by observing that $\exp \brack{2\sigma B(t) - \frac{1}{2}(2\sigma)^2t}$ is a martingale with mean $1$, and that it is independent of the initial data $X_0$. We can then take limits to deduce flocking in expectation provided $\sigma^2 < 2C_a$.   \qedhere 
\end{proof}

\subsection{SPDE flocking with the general kernel}

Going forwards, we use the notation $X(t) \coloneqq \norm{j(t) - \vbar \rho(t)}^2$. Conceptually, the proof for the general kernel shares aspects with the constant case but with additional steps and complexity. Noticeably, we still derive an evolution equation for $X(t)$ but it is no longer a Geometric Brownian motion. However, we can appeal to a stochastic comparison principle to bound its dynamics from above by a suitably chosen Geometric Brownian motion. This only holds locally-in-time but by deriving an energy estimate for $\vbar^2 \norm{\rho}^2 + \norm{j}^2$, 
we can use a Markov argument to estimate $\mathbb{P}(\tau < \infty)$.

We begin by deriving an evolution equation for $X(t)$.

\begin{lemma}\label{lem: Xt-SDE}
    The process $X(t)$ satisfies the following SDE:
    $$
    \dd X(t) = (\sigma^2 - 2C_a)X(t) \dd t + 2 \sigma X(t) \dd B(t) + 2 \theta \abrack{j-\vbar \rho, -j(g*\rho) + \rho(g * j)} \dd t.
    $$
\end{lemma}

\begin{proof}
    Apply It\^o's lemma to compute $\dd \norm{j(t) - \vbar \rho(t)}^2$. The computations mimic those in the proof of Theorem \ref{thm: constant-kernel-flocking} except that the convolution terms now pick up a non-local component from the spatial component of the kernel. In particular, we now have
    $$
    \rho(a * j) = \rho C_a \vbar + \theta \rho (g * j), \quad -j(a*\rho) = -jC_a - \theta j(g * \rho), 
    $$
    with the latter term in each expression being the extra contribution to the SDE for $X(t)$.
\end{proof}

We now prove Lemma \ref{lem: p=2-flocking-bound}. The essence of the proof is an application of a stochastic comparison theorem \cite[Chapter VI, Theorem 1.1]{ikeda2014stochastic} with an appropriately chosen Geometric Brownian motion.

\begin{proof}[Proof of Lemma \ref{lem: p=2-flocking-bound}]
    The proof is split into two parts: we first provide an estimate on the non-local term in the SDE for $X(t)$. Then we use a comparison theorem for SDEs to obtain the desired bound.

    \textbf{Step 1:} By adding and subtracting $2\theta \abrack{j - \vbar \rho, \vbar \rho (g*\rho)}$, we can write
    $$
        2 \theta \abrack{j-\vbar \rho, -j(g*\rho) + \rho(g * j)} = 2\theta \abrack{j-\vbar \rho, \rho(g*(j-\vbar \rho))} - 2\theta \abrack{j-\vbar \rho, (j-\vbar \rho)(g*\rho)}.
    $$
    Using Young's convolution inequality we estimate 
    $$
    \abrack{j-\vbar \rho, \rho(g*(j-\vbar \rho))} \leq \norm{g}\norm{\rho}\norm{j-\vbar \rho}^2,
    $$
    and 
    $$
    \abrack{j-\vbar \rho, (j-\vbar \rho)(g*\rho)} \leq \norm{g}\norm{\rho}\norm{j-\vbar \rho}^2.
    $$
    Together with the definition of the stopping time $\tau$ in \eqref{defn_stop_time}, these give
    $$
    2\theta \abrack{j-\vbar \rho, \rho(g*(j-\vbar \rho))} - 2\theta \abrack{j-\vbar \rho, (j-\vbar \rho)(g*\rho)} \leq 4\theta \norm{g}\norm{\rho}\norm{j-\vbar \rho}^2.
    $$
    Then, up to the stopping time $\tau$ we have 
    $$
    2 \theta \abrack{j-\vbar \rho, -j(g*\rho) + \rho(g * j)} \leq C_a \norm{j - \vbar \rho}^2 = C_aX(t).
    $$
    \textbf{Step 2:} We now look to apply \cite[Chapter VI, Theorem 1.1]{ikeda2014stochastic}. Consider $t \leq \tau$, and define two SDEs 
    \begin{align}
    dX(t) &= \beta_X(t) \dd t + 2\sigma X(t) \dd B(t), \label{sde_X}\\
    dY(t) &= \beta_Y(t) \dd t + 2\sigma Y(t) \dd B(t), \label{sde_Y_comparison}
    \end{align}
    with the same initial condition $X(0)=Y(0)$, and where
    \begin{align*}
        \beta_X(s, \omega) &\coloneqq (\sigma^2 - 2C_a)X(s, \omega) + 2 \theta \abrack{j-\vbar \rho, -j(g*\rho) + \rho(g * j)}(s, \omega), \\
        \beta_Y(s, \omega) &\coloneqq (\sigma^2 - C_a)Y(s, \omega).
    \end{align*}
    Since the diffusion is simply the map $(t,x) \mapsto 2\sigma x$, the assumption of \cite[Chapter VI, Theorem 1.1]{ikeda2014stochastic} are satisfied. 
    If we then define
    $
    b(t, x) \coloneqq b_1(t, x) = b_2(t, x) = (\sigma^2 - C_a)x,
    $
    we have $\beta_X(t) \leq b(t, X(t))$ and $\beta_Y(t) = b(t, Y(t))$ as consequences of the  stopping time restriction. 
    Moreover, the process $Y(t)$ solving \eqref{sde_Y_comparison} 
    is a Geometric Brownian motion, so we have pathwise uniqueness of solutions to this SDE. This allows us to work under the weaker assumption that $b_1(t, x) \leq b_2(t, x)$. We can therefore apply \cite[Chapter VI, Theorem 1.1]{ikeda2014stochastic} to deduce that $X(t) \leq Y(t)$ holds $\mathbb{P}$-a.s. Since $Y(t)$ is a Geometric Brownian motion with $Y(0) = X(0)$:
    $$
    Y(t) = X(0)\exp \brack{(\sigma^2 - C_a)t + 2\sigma B(t) - \frac{1}{2}(2\sigma)^2 t}.
    $$
    The above restrictions hold for $t \leq \tau$, but we can stop the processes to obtain $X(t \wedge \tau) \leq Y(t \wedge \tau)$ $\mathbb{P}$-a.s. for any $t \geq 0$, that is,
    $$
    X(t \wedge \tau) \leq X_0 \exp \brack{(\sigma^2 - C_a)(t \wedge \tau) + 2 \sigma B(t \wedge \tau) - \frac{1}{2}(2\sigma)^2(t \wedge \tau)}.
    \qedhere
    $$
    \end{proof}

We now want to prove Proposition \ref{thm:p=2-stopping-time-bound}. The key step in the proof is an application of Markov's inequality after observing that
$$
\mathbb{P}(\{\tau < t\}) \leq \mathbb{P}\brack{\cbrack{\norm{(\rho(t \wedge \tau), j(t \wedge \tau))}^2 \geq \brack{\frac{C_a \vbar}{4 \theta \norm{g}}}^2}}.
$$
We therefore require an energy estimate for the $L^2 \times L^2$ norm which we give in the following technical lemma.

\begin{lemma}\label{lem:p=2-energy-estimate}
    Let Assumption \ref{ass:IC} \ref{ass:idp-ic} 
    hold, assume that $\sigma^2 < C_a$ and define the energy $\eta(t) \coloneqq \vbar^2 \norm{\rho(t)}^2 + \norm{j(t)}^2$. Then for any $t \geq 0$ we have the moment bound
    $$
    \E{\eta(t \wedge \tau)} \leq 2 \brack{\E{\norm{j_0}^2} + \vbar^2\E{\norm{\rho_0}^2}} + 18\E{\norm{j_0 - \vbar \rho_0}^2} \brack{\frac{C_a}{C_a - \sigma^2}}^2.
    $$
\end{lemma}

\begin{proof}

Fix $t \geq 0$. It\^o's lemma gives
    \begin{multline*}
        \dd \eta(t) = 2\vbar^2 \abrack{\rho, -\partial_x j} \dd t + 2 \abrack{j, \rho(a*j) - j(a*\rho) - \vbar^2 \partial_x \rho} \dd t, \\
        + 2 \sigma \abrack{j, (j-\vbar \rho) \dd B(t)} + \sigma^2 \abrack{j-\vbar \rho, j-\vbar \rho} \dd t.
    \end{multline*}
    Using integration by parts reduces the top line to $2 \abrack{j, \rho(a*j) - j(a*\rho)}$. We then add and subtract $\vbar \rho$ in suitable $j$-terms as in \cite{zimper2025reduced} to get 
    \begin{align*}
        \frac{1}{2}\dd \eta(t) & = \abrack{j, \rho(\theta g)*(j-\vbar \rho)}\dd t - (C_a-\frac{1}{2}\sigma^2)\norm{j-\vbar \rho}^2 \dd t - \abrack{j-\vbar \rho, C_a \vbar \rho} \dd t \\
        & \quad - \abrack{(j-\vbar \rho)(\theta g)*\rho, j-\vbar \rho} \dd t - \abrack{(j-\vbar \rho)(\theta g)*\rho, \vbar \rho} \dd t + \sigma \abrack{j, (j-\vbar \rho) \dd B(t)} \\
        & =: \sum_{i=1}^5 T_i(t) \dd t + \sigma \abrack{j, (j(t)-\vbar \rho(t)) \dd B(t)}.
    \end{align*}
    We then integrate both sides of the equality up to time $t \wedge \tau$. This gives
    $$
    \eta(t \wedge \tau) = \eta(0) + 2\int_0^{t \wedge \tau}\sum_{i=1}^5 T_i(s) \dd s + 2 \sigma \int_0^{t \wedge \tau}\abrack{j(s), j(s) - \vbar \rho(s)} \dd B(s).
    $$
    Again following \cite{zimper2025reduced}, the definition of $\tau$ means we can bound
    $$
    \absv{T_1(s)}+ \absv{T_3(s)} + \absv{T_5(s)} \leq \frac{3C_a}{2}\norm{j-\vbar \rho}(s)\sqrt{\eta(s)}, \quad s \in [0, t\wedge \tau].
    $$
    We then bound 
    $$
    T_4 = -\abrack{(j-\vbar \rho)(\theta g)*\rho, j-\vbar \rho} \leq \theta \norm{g} \norm{\rho} \norm{j-\vbar \rho}^2 \leq \frac{C_a}{4}\norm{j-\vbar \rho}^2,
    $$
    using Young's convolution inequality, and the definition of $\tau$ in the final inequality. Therefore,
    $$
    T_2 + T_4 \leq - (C_a - \frac{1}{2}\sigma^2)\norm{j-\vbar \rho}^2\dd t + \frac{C_a}{4}\norm{j-\vbar \rho}^2\dd t \leq 0,
    $$
    if we choose $\sigma^2 \leq 3C_a/2$. Note that we assume $\sigma^2 < C_a$ so this is immediately satisfied. Overall, this gives
    $$
    \eta(t \wedge \tau) \leq \eta(0) + 3C_a \int_0^{t \wedge \tau} \sqrt{X(s)} \sqrt{\eta(s)} \dd s + 2\sigma \int_0^{t \wedge \tau}\abrack{j(s), j(s) - \vbar \rho(s)} \dd B(s).
    $$
    The latter term vanishes 
    upon taking expectations, giving
    $$
    \E{\eta(t \wedge \tau)} \leq \E{\eta(0)} + 3C_a \E{\int_0^{t \wedge \tau} \sqrt{X(s)} \sqrt{\eta(s)} \dd s}.
    $$
    For $s \in [0, t \wedge \tau]$, Lemma \ref{lem: p=2-flocking-bound} gives the estimate
    $$
    X_s \leq X_0 \exp\brack{(\sigma^2 - C_a)s}M_s,
    $$
    where $M_s \coloneqq \exp \brack{2 \sigma B_s - \frac{1}{2}(2\sigma)^2s}$ is an exponential martingale. Using this bound, and the non-negativity of the integrand,  we obtain
    $$
    \E{\eta(t \wedge \tau)} \leq \E{\eta(0)} + 3C_a \E{\int_0^t X_0^{\frac{1}{2}}\exp\brack{\frac{\sigma^2-C_a}{2}s}M_s^{\frac{1}{2}} \sqrt{\eta(s \wedge \tau)} \dd s}.
    $$
    Applying Cauchy-Schwarz on the expectation gives
    $$
    \E{\eta(t \wedge \tau)} \leq \E{\eta(0)} + 3C_a \brack{\E{X_0}}^{\frac{1}{2}}\int_0^t \exp\brack{\frac{\sigma^2-C_a}{2}s}\brack{\E{\eta(s \wedge \tau)}}^{\frac{1}{2}} \dd s,
    $$
    where we have used the independence of $X_0$ and $M_s$, and that $\E{M_s} = 1$. Applying a non-linear generalisation of Gr\"onwall's inequality \cite[Theorem 21]{dragomir2003applications} gives
    $$
    \E{\eta(t \wedge \tau)} \leq \sbrack{\brack{\E{\eta(0)}}^{\frac{1}{2}} + \frac{3C_a}{2}\brack{\E{X_0}}^{\frac{1}{2}}\int_0^t \exp\brack{\frac{\sigma^2-C_a}{2}s} \dd s}^2.
    $$
    Noting that the integral is bounded above by $2/(C_a - \sigma^2)$ for any $t \geq 0$, and using $(a+b)^2 \leq 2a^2 + 2b^2$
    $$
    \E{\eta(t \wedge \tau)} \leq 2 \brack{\E{\norm{j_0}^2} + \vbar^2\E{\norm{\rho_0}^2}} + 18\E{\norm{j_0 - \vbar \rho_0}^2} \brack{\frac{C_a}{C_a - \sigma^2}}^2.
    \qedhere
    $$
\end{proof}

We are now in a position to prove Proposition \ref{thm:p=2-stopping-time-bound}. 

\begin{proof}[Proof of Proposition \ref{thm:p=2-stopping-time-bound}]
    
    Fix $t \geq 0$ and consider the event $\cbrack{\tau < t}$. Then by continuity of $t \mapsto \norm{(\rho(t), j(t))}$ we have
    $$
    \mathbb{P}(\cbrack{\tau < t}) \leq \mathbb{P}(\cbrack{\tau \leq t}) = \mathbb{P}\brack{\cbrack{\norm{(\rho(t \wedge \tau), j(t \wedge \tau))} = \frac{C_a \vbar}{4 \theta \norm{g}}}}.
    $$
    By definition of the stopping time, $\norm{(\rho(t \wedge \tau), j(t \wedge \tau))}$ is bounded above by $C_a\vbar/(4 \theta \norm{g})$, so we can write
    $$
    \mathbb{P}\brack{\cbrack{\norm{(\rho(t \wedge \tau), j(t \wedge \tau))} = \frac{C_a \vbar}{4 \theta \norm{g}}}} = \mathbb{P}\brack{\cbrack{\norm{(\rho(t \wedge \tau), j(t \wedge \tau))}^2 \geq \brack{\frac{C_a\vbar}{4 \theta \norm{g}}}^2}}.
    $$
    Therefore, using Markov's inequality gives
    $$
    \mathbb{P}(\{\tau < t\}) \leq \brack{\frac{4 \theta \norm{g}}{C_a \vbar}}^2\E{\norm{(\rho(t \wedge \tau), j(t \wedge \tau))}^2} = \brack{\frac{4 \theta \norm{g}}{C_a}}^2 \frac{\E{\eta(t \wedge \tau)}}{\vbar^2}.
    $$
    The moment bound from Lemma \ref{lem:p=2-energy-estimate} coupled with the smallness assumption gives
    $$
    \mathbb{P}(\{\tau < t\}) \leq \frac{4 \theta \norm{g}}{C_a}.
    $$
    We now consider an increasing sequence of times $(t_n)_{n \in \mathbb{N}}$ such that $t_n \rightarrow \infty$ as $n \rightarrow \infty$. Defining $E_n \coloneqq \{\tau < t_n\}$, we have $E_n \subseteq E_{n+1}$ and $\bigcup_n E_n = \{\tau < \infty\}$. Thus, using continuity of measure, we deduce that
    $$
    \mathbb{P}(\{\tau < \infty\}) = \mathbb{P}(\bigcup_n E_n) = \lim_{n \rightarrow \infty} \mathbb{P}(E_n) \leq \frac{4 \theta \norm{g}}{C_a},
    $$
    so that
    $\mathbb{P}(\{\tau = \infty\}) = 1 - \mathbb{P}(\{\tau < \infty\}) \geq 1 - \frac{4 \theta \norm{g}}{C_a}.$
\end{proof}

\subsection{Small noise analysis}

We now conduct a small noise analysis and show that we can send $\sigma \rightarrow 0$ and recover flocking with a general kernel. This is exactly the deterministic result in \cite{zimper2025reduced}. The proof strategy is largely the same as in the previous section except now we consider arbitrary moments of the energy $\eta$. We state the required estimate in the following technical lemma.

\begin{lemma}\label{lem:general-p-energy-estimate} 
    Let Assumption \ref{ass:IC}  \ref{ass:p} hold and for fixed $\sigma>0$, assume that $C_a > \sigma^2$, and that $p \in [2, 1 + C_a/\sigma^2)$. 
    Then for any $t \geq 0$, we have the moment bound
    $$
    \E{\eta^{p/2}(t \wedge \tau)} \leq \sbrack{A(p) + B(p)\brack{\frac{3C_a + \sqrt{2}\sigma^2(p-2)}{C_a - \sigma^2(p-1)}}}^p,
    $$
    where
    $$
    A(p) \coloneqq \brack{\E{\brack{\norm{j_0}^2 + \vbar^2 \norm{\rho_0}^2}^{p/2}}}^{1/p}, \quad B(p) \coloneqq \brack{\E{\norm{j_0 - \vbar \rho_0}^{p}}}^{1/p}.
    $$
\end{lemma}

\begin{proof}
    Recall from Lemma \ref{lem:p=2-energy-estimate} that $\eta(t)$ satisfies
    $$
    \dd \eta(t) = 2 \sum_{i=1}^5 T_i(t) \dd t + 2 \sigma \abrack{j, j-\vbar \rho} \dd B(t).
    $$
    Applying It\^o's lemma with $\varphi(x) = x^{p/2}$ gives
    $$
    \dd \eta_t^{p/2} = \frac{p}{2}\eta_t^{p/2-1}\dd \eta_t + \frac{1}{2}\brack{\frac{p}{2}}\brack{\frac{p}{2}-1}\eta_t^{p/2-2} \dd \abrack{\eta}_t.
    $$
    Inserting the definitions of $\dd \eta_t$ and $\dd \abrack{\eta}_t$, and then integrating up to $t \wedge \tau$ gives
    \begin{multline*}
    \eta^{p/2}(t \wedge \tau) =  \eta^{p/2}(0) + p \int_0^{t \wedge \tau}\eta_s^{p/2-1} \sum_{i=1}^5 T_i(s) \dd s \\ + p \sigma^2 \brack{\frac{p}{2}-1}\int_0^{t \wedge \tau}\eta_s^{p/2-2}\abrack{j(s), j(s)-\vbar \rho(s)}^2 \dd s \\ + p\sigma \int_0^{t \wedge \tau} \eta_s^{p/2-1}\abrack{j(s), j(s) - \vbar \rho(s)} \dd B(s). 
    \end{multline*}
    As in the proof of Lemma \ref{lem:p=2-energy-estimate}, we estimate
    $$
    \absv{T_1} + \absv{T_3} + \absv{T_5} \leq \frac{3C_a}{2}\sqrt{X_t}\eta_t^{1/2}, \quad T_2 + T_4 \leq 0.
    $$
    The latter bound requires $\sigma^2 < 3C_a/2$ and this is automatically satisfied due to the assumption $p \in [2, 1+C_a/\sigma^2)$. For the middle term, we estimate
    $$
    \eta_s^{p/2-2}\abrack{j(s), j(s)-\vbar \rho(s)}^2 \leq \eta_s^{p/2-2}\norm{j(s)}^2X_s \leq \eta_s^{p/2-1}X_s,
    $$
    using Cauchy-Schwarz and that $\norm{j(s)}^2 \leq \eta(s)$ by definition. Plugging in these estimates gives
    \begin{multline*}
        \eta^{p/2}(t \wedge \tau) \leq \eta^{p/2}(0) + \frac{3C_a p }{2}\int_0^{t \wedge \tau}\sqrt{X_s}\eta_s^{(p-1)/2} \dd s + p\sigma^2 \brack{\frac{p}{2} - 1}\int_0^{t \wedge \tau}X_s \eta_s^{p/2-1} \dd s \\ + p\sigma \int_0^{t \wedge \tau} \eta_s^{p/2-1}\abrack{j(s), j(s) - \vbar \rho(s)} \dd B(s).
    \end{multline*}
    Now observe that
    $$
    X_s = \norm{j(s)-\vbar \rho(s)}^2 \leq 2\brack{\norm{j(s)}^2 + \vbar^2 \norm{\rho(s)}^2} = 2 \eta_s,
    $$
    so that
    $$
    X_s \eta_s^{p/2-1} = \sqrt{X_s}\sqrt{X_s} \eta_s^{p/2-1} \leq \sqrt{2}\sqrt{X_s}\eta^{(p-1)/2}.
    $$
    We can therefore merge the two deterministic integrals giving
    \begin{align}\label{pth_eta_before_expectation}
    \eta^{p/2}(t \wedge \tau) & \leq \eta^{p/2}(0) + \brack{\frac{3C_a p }{2} + \sqrt{2}p\sigma^2 \brack{\frac{p}{2} - 1}}\int_0^{t \wedge \tau}\sqrt{X_s}\eta_s^{(p-1)/2} \dd s \nonumber \\
    & \quad\quad + p\sigma \int_0^{t \wedge \tau} \eta_s^{p/2-1}\abrack{j(s), j(s) - \vbar \rho(s)} \dd B(s).
    \end{align}
    Taking expectations then gives
    \begin{align}\label{sublinear_gronwall_env_noise}
    \E{\eta^{p/2}(t \wedge \tau)} \leq \E{\eta^{p/2}_0} + \gamma(p) \E{\int_0^{t \wedge \tau}\sqrt{X_s}\eta_s^{(p-1)/2} \dd s},
    \end{align}
    where we have defined 
    $$
    \gamma(p) \coloneqq \frac{3C_a p }{2} + \sqrt{2}p\sigma^2 \brack{\frac{p}{2} - 1}.
    $$
    For $s \in [0, \tau]$, Lemma \ref{lem: p=2-flocking-bound} gives the estimate
    $$
    X_s \leq X_0 \exp \brack{(\sigma^2 - C_a)s}M_s,
    $$
    where $M_s \coloneqq \exp \brack{2\sigma B_s - \frac{1}{2}(2\sigma)^2s}$ is a martingale with mean 1. Using this estimate, and the non-negativity of $\eta$ gives
    $$
    \E{\eta^{p/2}(t \wedge \tau)} \leq \E{\eta^{p/2}_0} + \gamma(p) \E{\int_0^t X_0^{1/2} \exp \brack{\frac{\sigma^2-C_a}{2}s}M_s^{1/2}\brack{\eta^{p/2}(s \wedge \tau)}^{\frac{p-1}{p}} \dd s}.
    $$
    We then apply H\"older on the expectation with exponents $q=\frac{p}{p-1}$ and $r=p$ to get
    \begin{multline*}
        \E{\eta^{p/2}(t \wedge \tau)} \leq \E{\eta^{p/2}_0} \\ + \gamma(p) \int_0^t \brack{\E{X_0^{p/2}\exp \brack{ \frac{p(\sigma^2-C_a)}{2}s}M_s^{p/2}}}^{1/p}\brack{\E{\eta(s \wedge \tau)^{p/2}}}^{\frac{p-1}{p}} \dd s.
    \end{multline*}
    By adding and subtracting $\frac{1}{2}(p \sigma)^2 s$, we can write
    $$
    X_0^{p/2}\exp \brack{ \frac{p(\sigma^2-C_a)}{2}s}M_s^{p/2} = X_0^{p/2}\exp \brack{\frac{p}{2}(\sigma^2(p-1)-C_a)s}\tilde M_s,
    $$
    where $\tilde M_s \coloneqq \exp (p \sigma B_s - \frac{1}{2}(p \sigma)^2s)$ is a martingale with mean 1. Therefore, the estimate reduces to
    \begin{multline*}
    \E{\eta^{p/2}(t \wedge \tau)} \leq \E{\eta^{p/2}_0} \\ + \gamma(p)\brack{\E{X_0^{p/2}}}^{1/p}\int_0^t\exp \brack{\frac{1}{2}(\sigma^2(p-1) - C_a)s} \E{\eta(s \wedge \tau)^{p/2}}^{\frac{p-1}{p}} \dd s,
    \end{multline*}
    where we have used the independence of the initial data and the exponential martingale. Applying a non-linear generalisation of Gr\"onwall's inequality \cite[Theorem 21]{dragomir2003applications} gives
    $$
    \E{\eta^{p/2}(t \wedge \tau)} \leq \sbrack{\brack{\E{\eta^{p/2}_0}}^{1/p} + \frac{\gamma(p)}{p}\brack{\E{X_0^{p/2}}}^{1/p} \int_0^t \exp \brack{\frac{1}{2}(\sigma^2(p-1) - C_a)s} \dd s}^p.
    $$
    The integral of the exponential is bounded above by $2/(C_a - \sigma^2(p-1))$ for $p < 1 + C_a/\sigma^2$ so that
    \begin{align}\label{pth_eta_final}
    \E{\eta^{p/2}(t \wedge \tau)} \leq \sbrack{\brack{\E{\eta^{p/2}_0}}^{1/p} + \brack{\E{X_0^{p/2}}}^{1/p} \brack{\frac{3C_a + \sqrt{2}\sigma^2(p-2)}{C_a - \sigma^2(p-1)}}}^p.
    \end{align}
\end{proof}

We can now prove Theorem \ref{thm: stopping-time-exp-decay}. The proof proceeds in two steps: we first extend the argument in the proof of Proposition \ref{thm:p=2-stopping-time-bound} to estimate $\mathbb{P}(\{\tau < t\})$ except this time we consider arbitrary powers of the energy $\eta$. Secondly, for small enough $\sigma$, the bound is of the form $\mathbb{P}(\{\tau < t\}) \leq f(p(\sigma))^{p(\sigma)}$ which when optimised in $p(\sigma)$ (for some fixed $\sigma$) leads to an estimate that decays exponentially as $\sigma \rightarrow 0$.

\begin{proof}[Proof of Theorem \ref{thm: stopping-time-exp-decay}]
    Analogous to the proof of Proposition \ref{thm:p=2-stopping-time-bound}, we fix $t \geq 0$ and consider the event $\{\tau < t\}$. By continuity of $t \mapsto \norm{(\rho(t), j(t))}$, the definition of $\tau$, and Markov's inequality, we have
    \begin{align*}
        \mathbb{P}(\{\tau < t\}) &\leq \mathbb{P}\brack{\cbrack{\norm{(\rho(t \wedge \tau), j(t \wedge \tau))}^p \geq \brack{\frac{C_a \vbar}{4 \theta \norm{g}}}^p}} \\
        &\leq \frac{\E{\norm{(\rho(t \wedge \tau), j(t \wedge \tau))}^p}}{(C_a\vbar/4 \theta \norm{g})^p} \\
        &\leq \brack{\frac{4\theta \norm{g}}{C_a}}^p \frac{\E{\eta^{p/2}(t \wedge \tau)}}{\vbar^p},
    \end{align*}
    for any $p \geq 2$. 
    Defining $K \coloneqq 4 \theta \norm{g}/C_a$ and using the moment bound from Lemma \ref{lem:general-p-energy-estimate} gives
    $$
    \mathbb{P}(\{\tau < t\}) \leq \sbrack{\frac{K}{\vbar}\brack{A(p) + B(p)\brack{\frac{3C_a + \sqrt{2}\sigma^2(p-2)}{C_a - \sigma^2(p-1)}}}}^p =: f(p),
    $$
    where we now have the restriction $p \in [2, 1+C_a/\sigma^2)$. We now want to optimise $f(p)$ over $p$ in order to obtain an exponentially decaying estimate. First, we use the uniform bound on the initial data to bound
    $$
    A(p) \leq R_0 = A, \quad \text{and} \quad B(p) \leq \sqrt{2}R_0 = B,
    $$
    for any $p \geq 2$. Instead of pursuing a global minimiser, it is sufficient (and much simpler) to consider $p^* = \alpha C_a/\sigma^2$ for
    $$
    \alpha \in \brack{0, \frac{D-3}{D+\sqrt{2}}} \subset (0, 1), \quad D \coloneqq \frac{-A + \vbar/K}{B}.
    $$
    Then $p^* \geq 2$ if and only if $\sigma^2 \leq \alpha C_a/2$ which is true by assumption. Next, we estimate 
    $$
    f(p^*) \leq \sbrack{\frac{K}{ \vbar} \brack{A + B\brack{\frac{C_a(3+\sqrt{2}\alpha) -2\sqrt{2}\sigma^2}{C_a(1-\alpha) + \sigma^2}}}}^{\frac{\alpha C_a}{\sigma^2}} \leq \brack{g(\alpha)}^{\frac{\alpha C_a}{\sigma^2}},
    $$
    where we have defined
    $$
    g(\alpha) \coloneqq \frac{K}{ \vbar} \brack{A + B\brack{\frac{3+\sqrt{2}\alpha}{1-\alpha}}}.
    $$
    Taking logs in our estimate gives
    $$
    \log f(p^*) \leq \frac{\alpha C_a}{\sigma^2} \log g(\alpha).
    $$
    To get exponential decay we require $\log g(\alpha) < 0$ which occurs if and only if $g(\alpha) < 1$. This is true if and only if $\alpha \in \brack{0, \frac{D-3}{D+\sqrt{2}}}$, and $K(A + 3B) <  \vbar$. The latter requirement is necessary to ensure $D>3$ so that $\alpha \in (0, 1)$. In fact, it is nothing but a smallness assumption on the interaction kernel. Written explicitly it reads
    $$
    \frac{\brack{\E{\brack{\norm{j_0}^2 + \vbar^2 \norm{\rho_0}^2}^{p/2}}}^{1/p} + 3\brack{\E{\norm{j_0 - \vbar \rho_0}^{p/2}}}^{1/p}}{\vbar} \leq \frac{A+3B}{\vbar} < \frac{C_a}{4 \theta \norm{g}},
    $$
    and we notice that this is similar in flavour to the smallness assumption imposed in Proposition \ref{thm:p=2-stopping-time-bound}. To conclude, observe that
    $$
    \mathbb{P}(\{\tau < t\}) \leq \inf_{p} f(p) \leq f(p^*) \leq \exp \brack{- \frac{\alpha C_a \absv{\log g(\alpha)}}{\sigma^2}}.
    $$
    One can then mimic the argument in the proof of Proposition \ref{thm:p=2-stopping-time-bound} to send $t \rightarrow \infty$ and obtain
    $$
    \mathbb{P}(\{\tau = \infty\}) \geq 1 - \exp \brack{- \frac{\alpha C_a \absv{\log g(\alpha)}}{\sigma^2}}.
    \qedhere
    $$
\end{proof}

\section{Well-posedness for the independent noise case}\label{app: well-posedness-iid}

In this section, we prove Theorem \ref{thm:global_H1}, establishing the well-posedness of the system \eqref{const_reduced_CS_rho_iid_B}--\eqref{const_reduced_CS_k_iid_B}. We begin with an auxiliary result on the well-posedness of the velocity equation.

\begin{lemma}[Well-posedness of the velocity equation]
Let $\overline v(0)=\overline v_0$  and define the accumulated displacement
$V(t) := \int_0^t \overline{v}(s) ds$. Then the coupled system 
\begin{align}
    dV_t &= \overline{v}(t) dt \\
    d\overline{v} &= \sum_{z \geq 1} B_z(t,V(t)) d \beta_z(t), \label{velocitySDE}
\end{align}
where $B_z(t,V) := \sigma N^{-1/2}\sqrt{\lambda_z} e^{- \gamma t/2} \int_{\mathbb{T}}q_0(x-V) e_z(x) dx, \, z \geq 1$,
admits a unique strong solution with continuous paths.
\end{lemma}
\begin{proof}
For a given continuous velocity path $\overline{v}$, the transport equation
\eqref{const_reduced_CS_k_iid_B} has the explicit solution
\[
    k(t,x)=e^{-\gamma t}k_0(x-V(t)),
\qquad
    \sqrt{k(t,x)}
    =e^{-\gamma t/2}q_0(x-V(t)).
\]
Substitution into \eqref{evol_average_j} adapted for coloured noise (see Remark \ref{rem_coloured_noise_v_av}) gives the stated equation for
$\overline v$.

For any $q_0\in H^1(\mathbb T)$ and $V,W\in\mathbb R$,
the fundamental theorem of calculus for Sobolev
functions and Minkowski's inequality yield
\begin{align}
\|q_0(\cdot-V)-q_0(\cdot-W)\|_{L^2}
&\leq
\int_{\min(V,W)}^{\max(V,W)}
\|\partial_xq_0(\cdot-s)\|_{L^2}\,ds
\nonumber\\
&=|V-W|\|\partial_xq_0\|_{L^2},
\label{eq:translation-estimate}
\end{align}
where we used the translation invariance of the
$L^2$ norm.

By the Cauchy--Schwarz inequality and \eqref{eq:translation-estimate},
\begin{align*}
\sum_{z\geq1}|B_z(t,V)-B_z(t,W)|^2
&\leq
\frac{\sigma^2}{N}e^{|\gamma|T}
|V-W|^2\|\partial_xq_0\|_{L^2}^2
\sum_{z\geq1}\lambda_z\|e_z\|_{L^2}^2.
\end{align*}
Similarly,
\begin{align*}
\sup_{\substack{t\in[0,T]\\V\in\mathbb R}}
\sum_{z\geq1}|B_z(t,V)|^2
&\leq
\frac{\sigma^2}{N}e^{|\gamma|T}
\|q_0\|_{L^2}^2
\sum_{z\geq1}\lambda_z\|e_z\|_{L^2}^2
<\infty.
\end{align*}
Since $\sum_{z\geq1}\lambda_z<\infty$,
these estimates show that
$B(t,V)=(B_z(t,V))_{z\geq1}$ is bounded
and globally Lipschitz in $V$ as an
$\ell^2$-valued map, uniformly for
$t\in[0,T]$.

Together with the global Lipschitz continuity
and linear growth of the drift
$(V,\overline v)\mapsto(\overline v,0)$,
the standard existence and pathwise uniqueness
theorem for finite-dimensional SDEs driven
by cylindrical Brownian motion \cite[Chapter 7]{da2014stochastic}, with
Hilbert--Schmidt diffusion coefficients,
implies that the coupled system admits a
unique global strong solution
$(V,\overline v)$ with continuous paths.
\end{proof}

\begin{proof}[Proof of Theorem \ref{thm:global_H1}]
By the previous lemma,  $(V,\overline v)$ is uniquely determined, with 
$  V(t)=\int_0^t\overline v(s)\,ds.$  
The function $k$ is then determined by the transport equation. Indeed,
along the characteristics $\mathcal{X}(t)=x+V(t)$, 
\[
    k(t,\mathcal{X}(t))=e^{-\gamma t}k_0(x).
\]
Hence,
\begin{equation}\label{eq:k-explicit}
 k(t,x)=e^{-\gamma t}k_0(x-V(t)),\qquad
 q(t,x):=\sqrt{k(t,x)}=e^{-\gamma t/2}q_0(x-V(t)).
\end{equation}
Translations act strongly continuously on $H^1(\mathbb T)$, so $q\in C([0,T];H^1)$ a.s. Since $H^1(\mathbb T)$ is an algebra in one dimension, $k=q^2\in C([0,T];H^1)$ a.s.

Set
\[
    q_z(t,x):=\sigma N^{-1/2}\sqrt{\lambda_z}\,q(t,x)e_z(x) ,\qquad
 \overline q_z(t)=\int_{\mathbb T}q_z(t,x)\,dx.
\]
Using $\|fg\|_{H^1}\leq C\|f\|_{H^1}\|g\|_{W^{1,\infty}}$ and translation invariance of the $H^1$ norm gives
\begin{equation}\label{bound_q_zs}
 \sup_{t\leq T}\sum_z\|q_z(t)\|_{H^1}^2
 \leq C\frac{\sigma^2}{N}e^{|\gamma|T}\|q_0\|_{H^1}^2
 \sum_z\lambda_z\|e_z\|_{W^{1,\infty}}^2=:A_*<\infty.
\end{equation}
Here $A_*$ is deterministic and uniform in $t$. The operator $G(t):\ell^2\to H^1$ defined by $G(t)\mathbf e_z=q_z(t)$ satisfies
$\|G(t)\|_{HS(\ell^2,H^1)}^2=\sum_z\|q_z(t)\|_{H^1}^2$; thus $G\in L^2(\Omega\times(0,T);HS(\ell^2,H^1))$. Moreover,
\begin{equation}\label{eq:qbar-bound}
 \sup_{t\leq T}\sum_z|\overline q_z(t)|^2
 \leq |\mathbb T|\sup_{t\leq T}\sum_z\|q_z(t)\|_{L^2}^2
 \leq |\mathbb T|A_*=:b_*<\infty,
\end{equation}
and
\begin{equation}\label{eq:q-H1}
 \mathbb E\int_0^T\sum_z\|q_z(t)\|_{H^1}^2\,dt\leq TA_*<\infty.
\end{equation}

With $q = \sqrt{k}$ already constructed, the remaining  linear stochastic system is
\begin{align}
    d\rho &= -\partial_x j dt\\
    dj &= \{ C_a (\overline{v}(t) \rho - j) - \overline{v}^2(t) \partial_x \rho\} dt + \sum_z \sigma N^{-1/2} \sqrt{\lambda_z} q(t) e_z d\beta_z(t).
\end{align}
We also have $d\overline{v}(t) = \sum_z \overline{q}_z(t) d\beta_z(t)$.

\textbf{Galerkin approximation.}
Let $P_m$ be the  Fourier approximation  onto modes $|n|\leq m$, and let  $V_m := P_m L^2(\mathbb{T})$. Consider 
\begin{align}
    d\rho_m &= -\partial_x j_m\,dt,\label{eq:rho-m}\\
    dj_m &= \left\{C_a(\overline{v}(t)\rho_m-j_m)-\overline{v}(t)^2\partial_x\rho_m\right\}\,dt
    +\sum_{z\ge 1}P_mq_z(t)\,d\beta_z(t).\label{eq:j-m}
\end{align}
with initial data $(P_m \rho_0, P_m j_0)$. Since $\overline v$ is independent of $x$ and $P_m$ commutes with $\partial_x$, the drift lies in $V_m$. The coefficients of this finite-dimensional linear SDE are predictable, its drift matrices are uniformly bounded by the assumed deterministic velocity bounds, and its additive diffusion has uniformly bounded Hilbert--Schmidt norm by \eqref{bound_q_zs}. It has a unique adapted continuous solution. In addition, the standard moment estimate for finite-dimensional linear SDEs gives, for each $m$,
\begin{equation}\label{eq:galerkin-fourth}
 \mathbb E\sup_{t\leq T}\bigl(\|\rho_m(t)\|_{H^1}^4+\|j_m(t)\|_{H^1}^4\bigr)<\infty.
\end{equation}
This estimate follows by applying the usual finite-dimensional SDE moment bound with exponent $4$: the drift has linear growth with a deterministic bound (depending on $m$), while the additive diffusion is bounded. We use \eqref{eq:galerkin-fourth} only to justify the expectation of the martingale below. Note that its constant does not need to be uniform in $m$.

Define the weighted  energy
\begin{equation}\label{eq:Em}
    \mathcal E_m(t)
    :=
    \overline v(t)^2\|\rho_m(t)\|_{L^2}^2
    +\|j_m(t)\|_{L^2}^2
    +\overline v(t)^2\|\partial_x\rho_m(t)\|_{L^2}^2
    +\|\partial_xj_m(t)\|_{L^2}^2.
\end{equation}
By $v_*\leq|\overline v|\leq v^*$, there are deterministic $c,C>0$ depending only on $v_*,v^*$ such that
\begin{equation}\label{eq:energy-equivalence}
 c(\|\rho_m\|_{H^1}^2+\|j_m\|_{H^1}^2)
 \leq\mathcal E_m\leq
 C(\|\rho_m\|_{H^1}^2+\|j_m\|_{H^1}^2).
\end{equation}
In the energy estimates below, $C$ denotes a deterministic constant that may change from line to line and may depend on $C_a,v_*,v^*,b_*$ (and on fixed model parameters), but not on $m$.

\textbf{Zero order term} Since $d\rho_m=-\partial_xj_m\,dt $, integration by parts on the torus yields 
\begin{equation}\label{eq:rho-L2}
    d\|\rho_m\|_{L^2}^2
    =-2\langle\rho_m,\partial_xj_m\rangle\,dt
    =2\langle\partial_x\rho_m,j_m\rangle\,dt.
\end{equation}
It\^o's formula gives
\[
    d(\overline v^2)
    =
    2\overline v\sum_{z\geq 1}\overline q_z\,d\beta_z
    +
    \sum_{z\geq 1}|\overline q_z|^2\,dt.
\]
Moreover, because $\|\rho_m\|_{L^2}^2$ has finite variation, its quadratic covariation with $\overline{v}^2$ vanishes. Consequently, 
\begin{align}\label{eq:vrho-L2}
    d\left(\overline v^2\|\rho_m\|_{L^2}^2\right)
    &=
    2\overline v^2
    \langle\partial_x\rho_m,j_m\rangle\,dt \notag\\
    &\quad
    +
    \|\rho_m\|_{L^2}^2
    \sum_z|\overline q_z|^2\,dt
    +
    2\overline v\|\rho_m\|_{L^2}^2
    \sum_z\overline q_z\,d\beta_z.
\end{align}

Likewise,
\begin{align}\label{eq:j-L2}
 d\|j_m\|_{L^2}^2
 &=2C_a\langle j_m,\overline v\rho_m-j_m\rangle\,dt
 -2\overline v^2\langle j_m,\partial_x\rho_m\rangle\,dt\notag\\
 &\quad+\sum_z\|P_mq_z\|_{L^2}^2\,dt
 +2\sum_z\langle j_m,P_mq_z\rangle\,d\beta_z.
\end{align}
Adding \eqref{eq:vrho-L2} and \eqref{eq:j-L2}, the principal terms cancel:
$2\overline v^2\langle\partial_x\rho_m,j_m\rangle-2\overline v^2\langle j_m,\partial_x\rho_m\rangle=0$. Hence
\begin{align*}
 d(\overline v^2\|\rho_m\|_{L^2}^2+\|j_m\|_{L^2}^2)
 &=2C_a\langle j_m,\overline v\rho_m-j_m\rangle\,dt
 +\|\rho_m\|_{L^2}^2\sum_z|\overline q_z|^2\,dt\\
 &\quad+\sum_z\|P_mq_z\|_{L^2}^2\,dt+dM_{m,0}(t),
\end{align*}
where
\[
 M_{m,0}(t)=2\sum_z\int_0^t\left[
 \overline v\|\rho_m\|_{L^2}^2\overline q_z
 +\langle j_m,P_mq_z\rangle\right]d\beta_z.
\]
Young's inequality gives
\[
 |2C_a\langle j_m,\overline v\rho_m-j_m\rangle|
 \leq C(\overline v^2\|\rho_m\|_{L^2}^2+\|j_m\|_{L^2}^2)
 \leq C\mathcal E_m.
\]
Also,
\[
 \|\rho_m\|_{L^2}^2\sum_z|\overline q_z|^2
 \leq v_*^{-2}\left(\sum_z|\overline q_z|^2\right)\mathcal E_m.
\]
Therefore
\begin{equation}\label{eq:L2-estimate}
 d(\overline v^2\|\rho_m\|_{L^2}^2+\|j_m\|_{L^2}^2)
 \leq C(1+\sum_z|\overline q_z|^2)\mathcal E_m\,dt
 +\sum_z\|P_mq_z\|_{L^2}^2\,dt+dM_{m,0}(t).
\end{equation}

\textbf{First order term.} 
Differentiating the Galerkin equations in $x$ gives
\begin{align*}
 d(\partial_x\rho_m)&=-\partial_x^2j_m\,dt,\\
 d(\partial_xj_m)&=\{C_a(\overline v\partial_x\rho_m-\partial_xj_m)
 -\overline v^2\partial_x^2\rho_m\}\,dt
 +\sum_z\partial_xP_mq_z\,d\beta_z.
\end{align*}

As before, $\|\partial_x\rho_m\|_{L^2}^2$ has finite variation, so
\begin{align}\label{eq:vrho-deriv}
    d\left(
       \overline v^2\|\partial_x\rho_m\|_{L^2}^2
    \right)
    &=
    -2\overline v^2
    \langle\partial_x\rho_m,\partial_x^2j_m\rangle\,dt
    +
    \|\partial_x\rho_m\|_{L^2}^2
    \sum_z|\overline q_z|^2\,dt
    \notag\\
    &\quad
    +
    2\overline v\|\partial_x\rho_m\|_{L^2}^2
    \sum_z\overline q_z\,d\beta_z.
\end{align}

Similarly,
\begin{align}\label{eq:j-deriv}
    d\norm{\partial_xj_m}_{L^2}^2
    &=2C_a\ip{\partial_xj_m}{\overline{v} \partial_x\rho_m-\partial_xj_m}\,dt 
    -2\vbar^2\ip{\partial_xj_m}{\partial_x^2\rho_m}\,dt \notag\\
    &\quad
    +\sum_z\norm{\partial_xP_mq_z}_{L^2}^2\,dt
    +2\sum_z\ip{\partial_xj_m}{\partial_xP_mq_z}\,d\beta_z.
\end{align}

The second-order terms cancel by integration by parts:
\[
 -2\overline v^2\langle\partial_x\rho_m,\partial_x^2j_m\rangle
 -2\overline v^2\langle\partial_xj_m,\partial_x^2\rho_m\rangle=0.
\]
Estimating the remaining terms exactly as in the zero-order calculation gives
\begin{equation}\label{eq:H1-estimate-derivative}
 d(\overline v^2\|\partial_x\rho_m\|_{L^2}^2+\|\partial_xj_m\|_{L^2}^2)
 \leq C(1+\sum_z|\overline q_z|^2)\mathcal E_m\,dt
 +\sum_z\|\partial_xP_mq_z\|_{L^2}^2\,dt+dM_{m,1}(t),
\end{equation}
where
\[
 M_{m,1}(t)=2\sum_z\int_0^t\left[
 \overline v\|\partial_x\rho_m\|_{L^2}^2\overline q_z
 +\langle\partial_xj_m,\partial_xP_mq_z\rangle\right]d\beta_z.
\]

\textbf{Uniform energy estimate.} 
Combining \eqref{eq:L2-estimate} and \eqref{eq:H1-estimate-derivative}, and using the deterministic bound \eqref{eq:qbar-bound}, we obtain
\begin{equation}\label{eq:Em-simple}
 d\mathcal E_m(t)\leq C\mathcal E_m(t)\,dt
 +C\sum_z\|P_mq_z(t)\|_{H^1}^2\,dt+dM_m(t),
\end{equation}
where $M_m=M_{m,0}+M_{m,1}$ is given by
\[
 M_m(t)=2\sum_z\int_0^t\left[
 \overline v(s)\|\rho_m(s)\|_{H^1}^2\overline q_z(s)
 +\langle j_m(s),P_mq_z(s)\rangle_{H^1}
 \right]d\beta_z(s).
\]
Writing $A_m(t)=\sum_z\|P_mq_z(t)\|_{H^1}^2$, its quadratic variation satisfies
\begin{align*}
 d[M_m]_t
 &\leq C\overline v(t)^2\|\rho_m(t)\|_{H^1}^4
 \sum_z|\overline q_z(t)|^2\,dt
 +C\|j_m(t)\|_{H^1}^2A_m(t)\,dt\\
 &\leq C\bigl(b_*\mathcal E_m(t)^2+\mathcal E_m(t)A_m(t)\bigr)\,dt.
\end{align*}
Here we used $\overline v^2\|\rho_m\|_{H^1}^4\leq v_*^{-2}\mathcal E_m^2$ and $\|j_m\|_{H^1}^2\leq\mathcal E_m$. Since $A_m\leq A_*$ and \eqref{eq:galerkin-fourth} holds, the quadratic variation has finite expectation for each fixed $m$. Thus $M_m$ is a square-integrable martingale and BDG can be applied directly, without introducing an energy stopping time.

Let $F_m(t)=\mathbb E\sup_{0\leq s\leq t}\mathcal E_m(s)$. By BDG and the preceding estimate,
\begin{align*}
 \mathbb E\sup_{s\leq t}|M_m(s)|
 &\leq C\mathbb E\left(\int_0^t\mathcal E_m(s)A_m(s)\,ds\right)^{1/2}
 +C\sqrt{b_*}\,\mathbb E\left(\int_0^t\mathcal E_m(s)^2\,ds\right)^{1/2}\\
 &\leq \varepsilon F_m(t)
 +C_\varepsilon\mathbb E\int_0^t A_m(s)\,ds
 +C_\varepsilon\int_0^tF_m(s)\,ds.
\end{align*}
For the first term, we factor out $(\sup_{r\leq t}\mathcal E_m(r))^{1/2}$ and use Young's inequality. For the second, use
\[
 \left(\int_0^t\mathcal E_m(s)^2\,ds\right)^{1/2}
 \leq\left(\sup_{r\leq t}\mathcal E_m(r)\right)^{1/2}
 \left(\int_0^t\mathcal E_m(s)\,ds\right)^{1/2}
\]
and then Young's inequality. Choosing $\varepsilon$ small enough, integrating \eqref{eq:Em-simple}, and applying Gronwall gives
\begin{equation}\label{eq:sup-estimate}
 \mathbb E\sup_{t\leq T}\mathcal E_m(t)
 \leq C_T\left[\mathbb E\mathcal E_m(0)
 +\mathbb E\int_0^T\sum_z\|P_mq_z(s)\|_{H^1}^2\,ds\right].
\end{equation}
Crucially, $C_T$ does not depend on $m$: the auxiliary fourth-moment bound was used only to justify taking expectations, not in the final inequality. Since $P_m$ is an $H^1$ contraction and \eqref{eq:energy-equivalence} holds,
\begin{equation}\label{eq:uniform-H1}
 \mathbb E\sup_{t\leq T}\bigl(\|\rho_m(t)\|_{H^1}^2+\|j_m(t)\|_{H^1}^2\bigr)
 \leq C_T\left[\|\rho_0\|_{H^1}^2+\|j_0\|_{H^1}^2
 +\mathbb E\int_0^T\sum_z\|q_z(s)\|_{H^1}^2\,ds\right].
\end{equation}

By linearity of the Galerkin system, the differences \((\rho_m-\rho_n,j_m-j_n)\) satisfy the same type of equations, with initial data and noise coefficients given by the corresponding differences of Fourier projections. Applying the same energy estimate and using the strong convergence of \(P_m\to I\) in \(H^1(\mathbb T)\), together with dominated convergence and \eqref{bound_q_zs}, shows that \((\rho_m,j_m)\) is Cauchy in \(L^2(\Omega;C([0,T];H^1(\mathbb T)^2))\). Hence the Galerkin approximations converge strongly in this space to an adapted limit \((\rho,j)\).

\textbf{Identification of the limit.}
The Galerkin formulation is
\[
 \rho_m(t)=P_m\rho_0-\int_0^t\partial_xj_m(s)\,ds,
\]
and
\begin{align*}
 j_m(t)&=P_mj_0+\int_0^t\{C_a(\overline v(s)\rho_m(s)-j_m(s))
 -\overline v(s)^2\partial_x\rho_m(s)\}\,ds\\
 &\quad+\sum_z\int_0^tP_mq_z(s)\,d\beta_z(s).
\end{align*}
Since $P_m\rho_0\to\rho_0$ and $P_mj_0\to j_0$ in $H^1$, and $(\rho_m,j_m)\to(\rho,j)$ in $L^2(\Omega;C([0,T];H^1)^2)$, the drift terms converge in $L^2(\Omega;L^2(0,T;L^2))$, here we use the deterministic bound on $|\overline v|$. For the stochastic integral, the Hilbert-space BDG inequality yields
\[
 \mathbb E\sup_{t\leq T}\left\|\sum_z\int_0^t(P_m-I)q_z(s)\,d\beta_z(s)\right\|_{H^1}^2
 \leq C\mathbb E\int_0^T\sum_z\|(P_m-I)q_z(s)\|_{H^1}^2\,ds\longrightarrow0,
\]
by dominated convergence. Passing to the limit gives, as identities in $L^2(\mathbb T)$, almost surely for every $t\in[0,T]$,
\begin{align*}
 \rho(t)&=\rho_0-\int_0^t\partial_xj(s)\,ds,\\
 j(t)&=j_0+\int_0^t\{C_a(\overline v(s)\rho(s)-j(s))
 -\overline v(s)^2\partial_x\rho(s)\}\,ds
 +\sum_z\int_0^tq_z(s)\,d\beta_z(s).
\end{align*}
Thus $(\rho,j)$ has $H^1$-continuous trajectories and solves the system in the indicated integral sense. Passing to the limit in \eqref{eq:uniform-H1} gives
\[
 \mathbb E\sup_{t\leq T}\bigl(\|\rho(t)\|_{H^1}^2+\|j(t)\|_{H^1}^2\bigr)
 \leq C_T\bigl(\|\rho_0\|_{H^1}^2+\|j_0\|_{H^1}^2+TA_*\bigr)<\infty.
\]

\textbf{Uniqueness.}
The closed subsystem for $(V,\overline v)$ is pathwise unique, and $k$ is uniquely determined by \eqref{eq:k-explicit}. For two solutions driven by the same Brownian motions and with the same initial data, set $R=\rho^1-\rho^2$ and $J=j^1-j^2$. Their stochastic forcing cancels, and
\[
 dR=-\partial_xJ\,dt,\qquad
 dJ=\{C_a(\overline vR-J)-\overline v^2\partial_xR\}\,dt,
 \qquad R(0)=J(0)=0.
\]
In fact, uniqueness can be proved without another stochastic energy estimate. We first apply the Fourier projection $P_\ell$ to the difference equations. Since $P_\ell$ commutes with $\partial_x$ and the coefficients depend only on time, the pair $(P_\ell R,P_\ell J)$ satisfies the finite-dimensional homogeneous linear system
\[
 d(P_\ell R)=-\partial_x(P_\ell J)\,dt,\qquad
 d(P_\ell J)=\{C_a(\overline v P_\ell R-P_\ell J)-\overline v^2\partial_x(P_\ell R)\}\,dt
\]
with zero initial data. For almost every sample path the coefficients are continuous and bounded on $[0,T]$, so uniqueness for finite-dimensional linear ordinary differential equations gives $P_\ell R=P_\ell J=0$ identically. Letting $\ell\to\infty$ in $L^2(\mathbb T)$ yields $R=J=0$ almost surely on $[0,T]$. This proves pathwise uniqueness without differentiating an $H^1$ solution or using a further localization argument.
\end{proof}

\begin{remark}
\label{rem:non-constant-kernel}
The well-posedness argument for a general kernel is
more delicate than in the constant-kernel case.

First, for a constant kernel, the $k$-equation reduces
to a closed transport-reaction equation admitting an
explicit solution. For a general kernel, the
$k$-equation is coupled to the remaining equations,
and such an explicit representation is no longer
available.

Second, the system for $(\rho,j,k)$ is fully coupled
and nonlinear. In particular, one must control
terms of the form
\[
    \rho(a\ast j)-j(a\ast\rho),
    \qquad
    (j-\vbar\rho)\bigl(a\ast(j-\vbar\rho)\bigr),
    \qquad
    (a\ast\rho)k.
\]

Third, the Galerkin Cauchy argument is less
straightforward, since the nonlinear terms introduce
additional differences of the form
\[
    P_mF(U_m)-P_nF(U_n),
\]
which require further stability and projection
estimates in $H^1$.

A possible approach is therefore to use truncation
and a local Picard argument based on $H^1$ estimates.
Under suitable regularity assumptions, this is
expected to yield local $H^1$ well-posedness for
the regularised coefficient $b_\varepsilon(k)$.
For the original coefficient $\sqrt{k}$, the
corresponding argument applies as long as $k$
remains uniformly strictly positive, up to the
maximal $H^1$ existence time.
\end{remark}

 \section{SPDE flocking in the case of i.i.d. noise}\label{sec:proofs_iid}

\subsection{Proof of Theorem \ref{thm: constant-kernel-flocking-iid}: SPDE flocking estimate for constant kernel $a$} 

We break the proof up into several steps. We sometimes omit time dependencies for notational convenience. Furthermore, $\delta > 0$ is a constant which may change from line to line, and which we do not wish to explicitly keep track of. 

\emph{Step 1: exponential decay of norms of $k$}. Testing \eqref{const_reduced_CS_k_iid_B} with $k$ with respect to the $H^\beta$-norm, for any $\beta\geq 0$, and relying on integration by parts, one easily gets 
$d\|k\|^2_{H^\beta} = -2\gamma \|k\|^2_{H^\beta}dt$,
or equivalently $\|k(t)\|^2_{H^\beta} = e^{- 2 \gamma t}\|k_0\|^2_{H^\beta}$, where $\gamma>0$ due to Assumption \ref{ass_smallness_iid}.

\emph{Step 2: cross-variation bounds for noises}. We use the symbol $\llangle \cdot , \cdot \rrangle$ to denote cross-variation. Using the definition of the coloured noise $\xi$ and the Parseval identity in $L^2(\mathbb{T})$, we obtain
\begin{align}
    \llangle \sqrt{k}d\xi \rrangle & = \llangle \sum_{z}{\sqrt{\lambda_z}}\sqrt{k}e_z(x)d\beta_z(t)\rrangle ^2 = \sum_{z}{\lambda_z\|\sqrt{k}e_z\|^2}\dd t \lesssim \Tr(\xi)\|k\|_{L^1} \dd t\label{noise_l2_cross}\\
    \left\llangle \int_{x}\sqrt{k}d\xi \right\rrangle & = \left \llangle \sum_{z}{\sqrt{\lambda_z}}\int_{\mathbb{T}}\sqrt{k}e_z(x)d\beta_z(t) \right\rrangle =
    \sum_{z}{\lambda_z}\left(\int_{\mathbb{T}}\sqrt{k}e_z(x)\dd x\right)^2 \dd t \lesssim \|k\|_{L^1}\dd t.\label{noise_v_cross}
\end{align}

\emph{Step 3: statistical properties of $\overline{v}(t)$}. 
 
Using the representation \eqref{evol_average_j} adapted for coloured noise (see Remark \ref{rem_coloured_noise_v_av}), \eqref{noise_v_cross}, \emph{Step 1}, and the Burkholder-Davis-Gundy inequality \cite[Theorem 3.2]{burkholder1973distribution}, we get 
\begin{align}\label{bound_fluct_v_bar}
\mathbb{E}\left[\sup_{t\in[0,T]}|\overline{v}(t)-\overline{v}(0)|^p\right] & = \mathbb{E}\left[\sup_{t\in[0,T]}\left|\int_{0}^{t}{\sigma N^{-1/2}\int_{\mathbb{T}}\sqrt{k}\dd x \dd\xi}\right|^p\right] \nonumber\\
& \lesssim p^{p} \mathbb{E}\left[ \left(\sigma^2 N^{-1} \int_{0}^{T}{\|k\|_{L^1}dt} \right)^{p/2}\right] \nonumber \\
& \lesssim p^{p} \left( \sigma^2 N^{-1} \frac{1}{ \gamma} \right)^{p/2} \propto p^{p}(\sigma^2 N^{-1}\gamma^{-1} )^{p/2}.
\end{align}
By continuity of probability, we may replace $T$ with $+\infty$ in the left-hand-side of \eqref{bound_fluct_v_bar}.  
Plugging \eqref{bound_fluct_v_bar} into the Markov's inequality $\mathbb{P}(\sup_{t\geq 0}|\overline{v}(t)-\overline{v}(0)|>\overline{v}(0)/2) \leq \mathbb{E}[\sup_{t\geq 0}|\overline{v}(t)-\overline{v}(0)|^p]/(\overline{v}(0)/2)^p$, and performing an optimisation in $p$, we deduce
\begin{align}\label{prob_set_A}
\mathbb{P}(\sup_{t\geq 0}|\overline{v}(t)-\overline{v}(0)|>\overline{v}(0)/2) \lesssim \exp\left\{ - C(\sigma^2 N^{-1} \gamma^{-1} \overline{v}^2(0))^{-\alpha} \right\}
\end{align}
for some $C>0,\alpha>0$. This means that, if $\omega \in A^c$, where $A := \{\omega: |\sup_{t\geq 0}|\overline{v}(t)-\overline{v}(0)|>\overline{v}(0)/2\}$, then $\overline{v}(t)\propto \overline{v}(0)$, allowing for $\overline{v}(t)$ to be handled `as if it were a constant'.

\emph{Step 4: adapting the stopping bound definition}. Define $\tau_\rho := \inf \cbrack{t \geq 0 : \norm{(\rho(t), j(t))} \geq \overline{v}(0)\sqrt{\Tr(\xi)}}$ and $\tau_v := \inf \cbrack{t \geq 0 : |\overline{v}(t)-\overline{v}(0)| \geq \overline{v}(0)/2}$ and the stopping time
\begin{align}
    \tau :=
    \begin{cases}\label{defn_stop_time_iid}
        \min\{
        \tau_\rho; \tau_v\}, \quad &\text{if either is finite}, \\
        \infty, \quad &\text{otherwise}.
    \end{cases}
\end{align}

\emph{Step 5: second moment of flocking quantity $X_t = \|j- \overline{v}(t)\rho\|^2$}. It\^o's formula gives
\begin{align}
    dX_t & = 2\langle j- \overline{v}(t)\rho, \left\{ C_a[\overline{v}(t)\rho  - j] - \overline{v}^2(t) \nabla \rho\right\}dt + \sigma N^{-1/2}\sqrt{k}\,d\xi + \overline{v}(t)\nabla \cdot j \dd t - d(\overline{v}(t))\rho \rangle \nonumber \\
    & \quad \quad + \frac{1}{2}2 \llangle d(j-\overline{v}(t)\rho)\rrangle \dd t\nonumber \\
    & = -2 C_a \|j- \overline{v}(t)\rho\|^2 \dd t + 2\langle j- \overline{v}(t)\rho, \sigma N^{-1/2}\sqrt{k}\,d\xi\rangle - 2\langle j - \overline{v}(t)\rho, (d\overline{v}(t))\rho \rangle \nonumber \\
    & \quad \quad + \frac{1}{2}2 \sigma^2 N^{-1}\left \llangle\sqrt{k}d\xi - \left(\int_{\mathbb{T}}{\sqrt{k}d\xi dx}\right)\rho\right\rrangle
    \label{pre_gronwall_X_2} \\
    & \lesssim -2 C_a \|j- \overline{v}(t)\rho\|^2 \dd t + 2\langle j- \overline{v}(t)\rho, \sigma N^{-1/2}\sqrt{k}\,d\xi\rangle - 2\langle j - \overline{v}(t)\rho, (d\overline{v}(t))\rho \rangle \nonumber \\
    & \quad \quad + \frac{1}{2}2 \sigma^2 N^{-1}\left(\|k\|_{H^1}\mathtt{Tr}(\xi) + \|k\|_{H^1}\|\rho\|^2\right) \dd t\label{d_flock_2}
    \end{align}
where we have used integration by parts in the first equality, and the cross-variation bounds \eqref{noise_l2_cross}--\eqref{noise_v_cross} together with $H^1\subset L^1$ in the last inequality.

In what follows, under the smallness conditions in Assumption \ref{ass_smallness_iid}, we will combine \eqref{pre_gronwall_X_2}--\eqref{d_flock_2} with 
$
d \|k\|_{H^1} = -\gamma \|k\|_{H^1}
$
from \emph{Step 1} 
and deduce flocking properties for the combined quantity
\begin{align}\label{total_floc_quantity}
Y_t := \|j-\overline{v}(t)\rho\|^2 + \|k\|_{H^1}
\end{align}    
up to the stopping time \eqref{defn_stop_time_iid}.

\emph{Step 6: higher moments of auxiliary flocking quantity $Y_t$ \eqref{total_floc_quantity}}. For $\mathbb{N} \ni p\geq 2$, we use It\^o's formula in combination with \eqref{pre_gronwall_X_2} and \emph{Step 1} to deduce
\begin{align}\label{pre_gronwall_X_p}
    dY_t^{p/2} & = \frac{p}{2}Y^{p/2-1}dY_t + \frac{1}{2}\frac{p}{2}
    \left(\frac{p}{2} -1 \right)Y_t^{p/2-2}\llangle dY_t \rrangle \dd t \nonumber\\
    & \stackrel{\eqref{pre_gronwall_X_2}}{=} \frac{p}{2}Y^{p/2-1} \left\{-2 C_a \|j- \overline{v}(t)\rho\|^2 \dd t + 2\langle j- \overline{v}(t)\rho, \sigma N^{-1/2}\sqrt{k}\,d\xi\rangle - 2\langle j - \overline{v}(t)\rho, (d\overline{v}(t))\rho \rangle \nonumber \right. \nonumber \\
    & \left. \quad \quad + \frac{1}{2}2 \sigma^2 N^{-1}\left\llangle \sqrt{k}d\xi - \left(\int_{x}{\sqrt{k}d\xi dx}\right)\rho\right\rrangle -\gamma \|k\|_{H^1}\right\} \nonumber\\
    & \quad \quad + \frac{1}{2}\frac{p}{2}
    \left(\frac{p}{2} -1 \right)Y_t^{p/2-2}\llangle 2\langle j- \overline{v}(t)\rho, \sigma N^{-1/2}\sqrt{k}\,d\xi\rangle - 2\langle j - \overline{v}(t)\rho, (d\overline{v}(t))\rho \rangle \rrangle \dd t.
\end{align}
In light of \eqref{noise_l2_cross}--\eqref{noise_v_cross}, we get 
\begin{align}\label{pre_gronwall_X_p_2nd} 
& \llangle 2\langle j- \overline{v}(t)\rho, \sigma N^{-1/2}\sqrt{k}\,d\xi\rangle - 2\langle j - \overline{v}(t)\rho, (d\overline{v}(t))\rho \rangle \rrangle \nonumber \\ 
& \quad \lesssim 
\sum_{z}4\sigma^2 N^{-1}\langle j-\overline{v}(t)\rho, \sqrt{k}e_z \rangle^2 dt + \sum_{z}4\sigma^2 N^{-1}\langle j-\overline{v}(t)\rho, \left(\int_{\mathbb{T}}{\sqrt{k}e_z \dd x}\right)\rho \rangle^2 \nonumber \\
& \quad \lesssim 4\sigma^2 N^{-1} \left\{ \|(j-\overline{v}(t)\rho)\sqrt{k}\|^2 + \|j-\overline{v}(t)\rho\|^2\|\rho\|^2\|k\|_{L^1} \right\}.
\end{align}
Using \eqref{pre_gronwall_X_p}--\eqref{pre_gronwall_X_p_2nd} we may continue and get  
\begin{align}
    dY_{t}^{p/2} & \lesssim \frac{p}{2}Y_t^{p/2-1} \left\{-2 C_a \|j- \overline{v}(t)\rho\|^2 \dd t + 2\langle j- \overline{v}(t)\rho, \sigma N^{-1/2}\sqrt{k}\,d\xi\rangle - 2\langle j - \overline{v}(t)\rho, (d\overline{v}(t))\rho \rangle \nonumber \right. \\
    & \left. \quad \quad + \frac{1}{2}2 \sigma^2 N^{-1} \left\llangle \sqrt{k}d\xi - \left(\int_{x}{\sqrt{k}d\xi dx}\right)\rho \right\rrangle^2 - \gamma \|k\|_{H^1}\dd t \right\} \nonumber \\
    & \quad \quad + \frac{1}{2}\frac{p}{2}
    \left(\frac{p}{2} -1 \right)Y_t^{p/2-2}\cdot 4 \cdot \sigma^2 N^{-1} \left\{ \|(j-\overline{v}(t)\rho)\sqrt{k}\|^2 + \|j-\overline{v}(t)\rho\|^2\|\rho\|^2\|k\|_{L^1} \right\} dt \label{requires_l_infty_bound} \\
    & \lesssim \frac{p}{2}Y_t^{p/2-1} \left\{-2 C_a \|j- \overline{v}(t)\rho\|^2 \dd t + 2\langle j- \overline{v}(t)\rho, \sigma N^{-1/2}\sqrt{k}\,d\xi\rangle - 2\langle j - \overline{v}(t)\rho, (d\overline{v}(t))\rho \rangle \nonumber \right. \\
    & \left. \quad \quad + \frac{1}{2}2 \sigma^2 N^{-1}\left(\|k\|_{H^1}\mathtt{Tr}(\xi) + \|k\|_{H^1}\|\rho\|^2\right)\dd t - \gamma \|k\|_{H^1} \dd t \right\} \nonumber \\
    & \quad \quad + \frac{1}{2}\frac{p}{2}
    \left(\frac{p}{2} -1 \right)Y_t^{p/2-2}\cdot 4 \cdot N^{-1}\sigma^2\|k\|_{H^1}\left\| j- \overline{v}(t)\rho\right\|^2(1+\|\rho\|^2) dt \label{final_y_p_2}
\end{align}
where we have also used the 1D embedding $H^1\subset L^\infty$ in \eqref{final_y_p_2}. 
Using the definition of the stopping time $\tau$ \eqref{defn_stop_time_iid} and $\|k\|_{H^1}\left\| j- \overline{v}(t)\rho\right\|^2 \leq Y^2_t$, we continue and get

\begin{align}
    & dY^{p/2}_{t\wedge \tau} \\
    & \lesssim \frac{p}{2}Y_{t \wedge \tau}^{p/2-1} \left\{-2 C_a \|j- \overline{v}(t \wedge \tau)\rho\|^2 \dd t + 2\langle j- \overline{v}(t \wedge \tau)\rho, \sigma N^{-1/2}\sqrt{k}\,d\xi\rangle \nonumber \right. \\
    & \left. \quad \quad \quad \quad - 2\langle j - \overline{v}(t \wedge \tau)\rho, (d\overline{v}(t \wedge \tau))\rho \rangle
    + \frac{1}{2}2 \sigma^2 N^{-1}\left(\|k\|_{H^1}\mathtt{Tr}(\xi) + \|k\|_{H^1}\|\rho\|^2\right)\dd t - \gamma \|k\|_{H^1}\dd t \right\} \nonumber \\
    & \quad \quad + \frac{1}{2}\frac{p}{2}
    \left(\frac{p}{2} -1 \right)Y_{t \wedge \tau}^{p/2-2}\cdot 4 \cdot N^{-1}\sigma^2 \Tr(\xi) Y^2_{t \wedge \tau} dt \label{final_y_higher_mom} \\
     & \lesssim \frac{p}{2}Y^{p/2}_{t\wedge \tau}\left( -2C_a + \sigma^2 N^{-1}\Tr(\xi) \right) \dd t \nonumber \\
    & \quad + \frac{p}{2}Y_{t \wedge \tau}^{p/2-1} \left\{2\langle j- \overline{v}(t \wedge \tau)\rho, \sigma N^{-1/2}\sqrt{k}\,d\xi\rangle - 2\langle j - \overline{v}(t \wedge \tau)\rho, (d\overline{v}(t \wedge \tau))\rho \rangle \nonumber \right\} \nonumber \\
    & \quad + \frac{1}{2}\frac{p}{2}\left(\frac{p}{2}-1\right)4\sigma^2 N^{-1}\Tr(\xi)Y^{p/2}_{t\wedge \tau}\dd t. \nonumber
\end{align}

Using Assumptions \ref{ass_smallness_iid}, and imposing the bound
\begin{align}\label{bound_p}
p\leq c\cdot C_a(N^{-1/2}\sigma \Tr^{1/2}(\xi))^{-2}
\end{align}
for some small enough $c>0$, and using Gronwall's Lemma that, for some $\delta > 0$,
\begin{align}\label{flocking_Y_pth_moment}
    \mathbb{E}[e^{\frac{p}{2}\delta C_a(t\wedge \tau)}Y^{p/2}_{t\wedge \tau}] \leq \mathbb{E}[Y_0^{p/2}].
\end{align}

\emph{Step 7: higher-order bounds on the energy $\eta := \|j\|^2 + \overline{v}^2(t)\|\rho\|^2 + \|k\|_{H^1}$}.
Using It\^o's formula once again, we obtain
\begin{align}
    d \eta & = 2\langle j , \left\{C_a[\overline{v}(t)\rho  - j] - \overline{v}^2(t) \nabla \rho\right\}dt + \sigma N^{-1/2}\sqrt{k}\,d\xi \rangle 
        + \frac{1}{2}\cdot 2\cdot \llangle dj \rrangle \nonumber \\
        & \quad + 2\overline{v}(t)d \overline{v}(t)\|\rho\|^2 + \frac{1}{2}\cdot 2\cdot \llangle d\overline{v}(t)\rrangle \|\rho\|^2 + \overline{v}^2(t)\cdot 2\langle \rho, -\nabla \cdot j\rangle \dd t - \gamma \|k\|_{H^1}\dd t \nonumber\\
        & = 2\langle j , \sigma N^{-1/2}\sqrt{k}\,d\xi \rangle + 2\overline{v}(t)d \overline{v}(t)\|\rho\|^2 \nonumber\\ 
        & \quad + 2\langle j , C_a[\overline{v}(t)\rho  - j]\rangle \dd t  + \sigma^2 N^{-1}\left\llangle \int_{x}{\sqrt{k}d\xi \dd x } \right\rrangle \|\rho\|^2
        +  \sigma^2 N^{-1} \llangle \sqrt{k}d\xi \rrangle - \gamma\|k\|_{H^1} \dd t\label{2nd_mom_energy_eta}. 
\end{align}

As for $d\eta^{p/2}$ with $p\geq 2$, we get
\begin{align}\label{pre_gronwall_eta_p}
d\eta^{p/2} & = \frac{p}{2}\eta^{p/2-1}d\eta + \frac{1}{2}\cdot\frac{p}{2}\cdot\left(\frac{p}{2}-1\right)\cdot \eta^{p/2-2}\cdot\llangle d\eta \rrangle \nonumber \\
& = \frac{p}{2}\eta^{p/2-1}\left\{2\langle j , \sigma N^{-1/2}\sqrt{k}\,d\xi \rangle + 2\overline{v}(t)d \overline{v}(t)\|\rho\|^2 \right. \nonumber\\ 
& \left.\quad\quad\quad + 2\langle j , C_a[\overline{v}(t)\rho  - j]\rangle \dd t  + \sigma^2 N^{-1}\left\llangle \int_{x}{\sqrt{k}d\xi} \right\rrangle \|\rho\|^2
        +  \sigma^2 N^{-1} \llangle \sqrt{k}d\xi \rrangle  - \gamma\|k\|_{H^1} \dd t \right\} \nonumber \\
& \quad + \frac{1}{2}\cdot\frac{p}{2}\cdot\left(\frac{p}{2}-1\right)\cdot \eta^{p/2-2}\cdot \llangle 2\langle j , \sigma N^{-1/2}\sqrt{k}\,d\xi \rangle + 2\overline{v}(t)d \overline{v}(t)\|\rho\|^2 \rrangle \\
&  \lesssim \frac{p}{2}\eta^{p/2-1}\left\{2\langle j , \sigma N^{-1/2}\sqrt{k}\,d\xi \rangle + 2\overline{v}(t)d \overline{v}(t)\|\rho\|^2 \right. \nonumber\\ 
& \left.\quad\quad\quad + 2\langle j , C_a[\overline{v}(t)\rho  - j]\rangle \dd t  + \sigma^2 N^{-1}\left\llangle \int_{x}{\sqrt{k}d\xi} \right\rrangle^2 \|\rho\|^2
        +  \sigma^2 N^{-1} \left\llangle \sqrt{k}d\xi \right\rrangle^2\right\} \nonumber \\
& \quad + \frac{1}{2}\cdot\frac{p}{2}\cdot\left(\frac{p}{2}-1\right)\cdot \eta^{p/2-2}\cdot \left[ 4 \sigma^2 N^{-1}\left\|j\sqrt{k}\right\|^2 + 4\overline{v}^2(t)\sigma^2 N^{-1} \|k\|_{H^1}\|\rho\|^4 \right]. \label{pth_mom_energy_eta}
\end{align}
Using the stopping time $\tau$ \eqref{defn_stop_time_iid} to bound two of the four powers in $\|\rho\|^4$ in the right-hand-side of \eqref{pth_mom_energy_eta}, 
reconstructing $\eta$ using the remaining two powers as well as $\overline{v}^2(t)$, 
noticing that $Y_{t \wedge \tau} = Y^{1/2}_{t \wedge \tau} Y^{1/2}_{t \wedge \tau} \leq  Y^{1/2}_{t \wedge \tau} \eta^{1/2}_{t \wedge \tau}$, 
using Assumption \ref{ass_smallness_iid}, as well as the 1D embedding $H^1\subset L^\infty$ to deal with the term $\norm{j \sqrt{k}}^2$,
we deduce that 
\begin{align}\label{pth_eta_compact}
    d\eta_{t\wedge \tau}^{p/2} & \lesssim p \cdot \eta_{t\wedge \tau}^{p/2-1}\left(C_a \eta_{t\wedge \tau}^{1/2}Y_{t\wedge \tau}^{1/2} + \sigma^2 N^{-1}\Tr(\xi)Y_{t\wedge \tau}\right) + p^2 \cdot \eta_{t\wedge \tau}^{p/2-2} \sigma^2 N^{-1} \Tr(\xi) \eta_{t\wedge \tau} Y_{t\wedge \tau}  \nonumber \\
    & \quad \quad + p\cdot \eta_{t\wedge \tau}^{p/2-1}\left\{2\langle j , \sigma N^{-1/2}\sqrt{k}\,d\xi \rangle + 2\overline{v}(t\wedge \tau)d \overline{v}(t \wedge \tau)\|\rho\|^2 \right\} \nonumber \\
    & \lesssim p \cdot C_a \eta_{t\wedge \tau}^{p/2-1/2}Y_{t\wedge \tau}^{1/2} 
    + p^2 \cdot \sigma^2 N^{-1}\Tr(\xi)\eta_{t\wedge \tau}^{p/2-1/2}Y^{1/2}_{t\wedge \tau} \nonumber \\
    & \quad \quad + p\cdot \eta_{t\wedge \tau}^{p/2-1}\left\{2\langle j , \sigma N^{-1/2}\sqrt{k}\,d\xi \rangle + 2\overline{v}(t\wedge \tau)d \overline{v}(t \wedge \tau)\|\rho\|^2 \right\}.
\end{align}

Looking at \eqref{pth_eta_compact}, we realise that it is structurally identical to \eqref{pth_eta_before_expectation} in the proof of Lemma \ref{lem:general-p-energy-estimate}, only with the size of the noise being $\sigma N^{-1/2}\sqrt{\Tr(\xi)}$ rather than $\sigma$. We can thus follow the same logic of the proof of Lemma \ref{lem:general-p-energy-estimate} and get to a bound similar to \eqref{pth_eta_final} (note that, even though we do not have pathwise bounds for $Y$, this does not affect the argument, as the application of the necessary H\"older can be carried out with the knowledge of the bounds \eqref{flocking_Y_pth_moment} only). Taking into account \eqref{bound_p} (in particular, with the smallness of the constant $c$ therein) we deduce that
 \begin{align}\label{bound_eta_iid}
    \E{\eta^{p/2}(t \wedge \tau)} \lesssim \sbrack{\brack{\E{\eta^{p/2}_0}}^{1/p} + \brack{\E{Y_0^{p/2}}}^{1/p} }^p.
 \end{align}

 We may now proceed in analogy to the proof of the Theorem \ref{thm: stopping-time-exp-decay}: using \eqref{bound_eta_iid} and \eqref{prob_set_A}, we deduce that 
 \begin{align*}
 \P{\tau <\infty} & = \lim_{t\rightarrow \infty}\P{0<\tau < t}  \leq  \P{\tau_v < \infty} + \lim_{t\rightarrow \infty} \P{\{0 < \tau_\rho < t\}\cap\{\tau_\rho \leq \tau_v\}}\\
 & \leq \exp\left\{ - C(\sigma^2 N^{-1} \gamma^{-1} \overline{v}^2(0))^{-\alpha} \right\} + \frac{\sbrack{\brack{\E{\eta^{p/2}_0}}^{1/p} + \brack{\E{Y_0^{p/2}}}^{1/p} }^p}{(\overline{v}(0)\sqrt{\mathtt{Tr}(\xi)})^p}.
 \end{align*}

 Upon smallness of the initial data $\eta_0$ and $Y_0$ with respect to $\overline{v}(0)\sqrt{\Tr(\xi)}$ stated in Assumption \ref{ass_smallness_iid} and the upper bound \eqref{bound_p}, we deduce 
\begin{align*}
    \mathbb{P}(\tau < \infty) \propto \exp\left\{ - C (\sigma^2 N^{-1} \gamma^{-1} \overline{v}^2(0))^{-\alpha} \right\} + \exp\{ - C(\sigma N^{-1/2}\Tr^{1/2}(\xi))^{-2}\},
\end{align*}
for some $C>0, \alpha > 0$ 
(this is \eqref{stop_time_bound_iid}), thus implying \eqref{flocking_est_iid} also using \eqref{flocking_Y_pth_moment}. The proof of Theorem \ref{thm: constant-kernel-flocking-iid} is concluded.
\begin{remark}
    Note that, in contrast to the common noise setting, due to the lack of the Geometric Brownian motion nature of the noise, we could not use the pathwise comparison principles for SDEs to get the bound \eqref{flocking_Y_pth_moment}. As a result, we only obtain SPDE flocking in expectation and not in the pathwise sense.
\end{remark}

\subsection{Remark \ref{thm: general-kernel-flocking-iid-exp}: SPDE flocking estimate for general kernel $a$} 

We choose to only provide a sketch of the methodology for Remark \ref{thm: general-kernel-flocking-iid-exp}. 

\begin{proof}[Sketch of proof for Remark \ref{thm: general-kernel-flocking-iid-exp}]
Replacing the continuous $L^2$-norm $\|\cdot\|$ with its discrete counterpart $\|\cdot\|_h$ (see paragraph before the statement of Remark \ref{thm: general-kernel-flocking-iid-exp} for the necessary notation) does not alter the algebraic structure of the It\^o formula, or the applications of Parseval's identities. However, in comparison to the proof of Theorem \ref{thm: constant-kernel-flocking-iid}, we face a few additional complications. These are all tied to the evolution equation for $k_h$ being no longer self-contained, due to the presence of $\rho_h$ and $j_h$. In particular, this prevents us from obtaining estimates for $k_h$ in the $H^1$-norm in a reasonable way. Since we do not have access to the $H^1$-norm bounds for $k_h$, this has a few consequences:
    \begin{itemize}
        \item we have to redefine the flocking quantity $Y_t$.
        \item we can not use the $L^1$-$L^\infty$ splitting of terms in \eqref{pth_mom_energy_eta} and \eqref{final_y_p_2}.
    \end{itemize}
As for the definition of $Y_t$, we choose $Y_t := \|j_h-\overline{v}(t)\rho_h\|^2_h + \|k_h\|_h$, i.e., we downgrade the $H^1$ norm on $k_h$ to a $L^2$ norm. This allows to obtain flocking bounds for $Y_t$ in the following way. Define the stopping time
\begin{align}
    \tau_{\rho} :=
    \begin{cases}\label{preliminary_stopping_time_exp}
        \inf \cbrack{t \geq 0 : \norm{\rho_h(t)}_h \geq \min\left\{\frac{C_a}{4\theta \|g\|_h}; \sqrt{\Tr(\xi_h)}\right\}}, \quad &\text{if infimum taken over non-empty set}, \\
        \infty, \quad &\text{otherwise},
    \end{cases}
\end{align}
for now avoiding keeping track of whether $\overline{v}(t)$ deviates significantly from $\overline{v}(0)$. We start by noticing that 
\begin{align}
    d k_h & = 
    2(j_h-\overline{v}(t)\rho_h)a\ast_h (j_h-\overline{v}(t)\rho_h)dt + \left(\sigma^2(1-2N^{-1})- 2 a\ast_h \rho_h\right)k_h dt - \overline{v}(t)\nabla_h k_h dt. \nonumber \\
    & = 
    2(j_h-\overline{v}(t)\rho_h) \theta g \ast_h (j_h-\overline{v}(t)\rho_h)dt + \left(\sigma^2(1-2N^{-1})- 2 C_a - 2 \theta g\ast_h \rho_h\right)k_h dt - \overline{v}(t)\nabla_h k_h dt. \label{reduced_CS_k_iid_B_expanded_exp}
\end{align}

Upon smallness of $\sigma$, and up to the stopping time $\tau_\rho$ (this allows us to bound the term $\theta g\ast_h \rho_h$), we deduce existence of $\gamma > 0$ such that
\begin{align}
    d\|k_h\|_h & = d\sqrt{\|k_h\|_h^2} \lesssim \frac{1}{2}\left(\|k_h\|_h^2\right)^{-1/2}\left( \langle 2(j_h-\overline{v}(t)\rho_h) \theta g \ast_h (j_h-\overline{v}(t)\rho_h) , k_h \rangle_h dt - \gamma \|k_h\|_h^2 dt \right) \nonumber\\
    & \lesssim \frac{1}{2}\left(\|k_h\|_h^2\right)^{-1/2}\left( \left\|2(j_h-\overline{v}(t)\rho_h) \theta g \ast_h (j_h-\overline{v}(t)\rho_h)\right\|_h\|k_h\|_h - \gamma \|k_h\|_h^2 \right)dt \nonumber \\
    & \lesssim \frac{1}{2}\left(\|k_h\|_h^2\right)^{-1/2}\left( X_t \|k_h\|_h - \gamma \|k_h\|_h^2 \right)dt \nonumber\\
    & \lesssim  \left( X_t - \gamma \|k_h\|_h \right)dt.\label{ev_norm_k_exp_1}
\end{align}
The possible degeneracy for $\|k_h\|_h = 0$ in \eqref{ev_norm_k_exp_1} has been dealt with by taking a non-negative $C^1$ approximation $f_\epsilon$ of $x\mapsto \sqrt{x}$ such that $f_\epsilon$ agrees with $x\mapsto \sqrt{x}$ in the region $[\epsilon,\infty)$, and such that $f'_\epsilon(x)\leq \frac{1}{2\sqrt{x}}$, and then sending $\epsilon\rightarrow 0$. We have omitted these details. 
Next, using the stopping time $\tau_\rho$ to bound the quadratic non-linearity it is straightforward to adapt \eqref{d_flock_2} to obtain
\begin{align}
    & dX_{t\wedge \tau_\rho} \nonumber \\
    & \lesssim - 2 C_a \|j_h- \overline{v}(t \wedge \tau_\rho)\rho_h\|_h^2 
    \nonumber \\
    & \quad \quad + 2\langle j_h- \overline{v}(t \wedge \tau_\rho)\rho_h, \sigma N^{-1/2}\sqrt{k_h}\,d\xi_h \rangle_h - 2\langle j_h - \overline{v}(t \wedge \tau_\rho)\rho_h, (d\overline{v}(t \wedge \tau_\rho))\rho_h \rangle_h \nonumber \\
    & \quad \quad + \frac{1}{2}2 \sigma^2 N^{-1}\left(\|k_h\|_h\mathtt{Tr}(\xi_h) + \|k_h\|_h\|\rho_h\|_h^2\right) \\
    & \lesssim - 2 C_a \|j_h- \overline{v}(t)\rho_h\|_h^2 + 2\langle j_h- \overline{v}(t \wedge \tau_\rho)\rho_h, \sigma N^{-1/2}\sqrt{k_h}\,d\xi_h \rangle_h - 2\langle j_h - \overline{v}(t)\rho_h, (d\overline{v}(t \wedge \tau_\rho))\rho_h \rangle_h \nonumber \\
    & \quad \quad + \frac{1}{2}2 \sigma^2 N^{-1}\Tr(\xi_h)\|k_h\|_h.\label{d_flock_2_general}
\end{align}
where we have used the embedding $\|k_h\|_{h,1} \lesssim \|k_h\|_{h}$ as opposed to the embedding $\|k_h\|_{h,1} \lesssim \|k_h\|_{H^1}$.  
Upon smallness of parameters (cfr. Assumption \ref{ass_smallness_iid}) we can add \eqref{d_flock_2_general} and  \eqref{ev_norm_k_exp_1} to obtain the exponentially decaying bound  $\E{e^{\delta C_a t\wedge \tau_\rho}Y(t\wedge \tau_{\rho})} \leq \E{Y_0}$. Furthermore, we can apply the It\^o formula to $Y_{t\wedge \tau}^{p/2}$ and get an expression analogous to those discussed in \emph{Step 6} of proof of Theorem \ref{thm: constant-kernel-flocking-iid}. The only difference now is that the the problematic term $\|(j_h - \overline{v}(t)\rho_h)\sqrt{k_h}\|^2_h$ is estimated as
\begin{align}\label{l4_embed}
 \|(j_h - \overline{v}(t)\rho_h)\sqrt{k_h}\|_h^2 & \leq \|j_h - \overline{v}(t)\rho_h\|_{4,h}^2  \|k_h\|_h \lesssim h^{-1/2} \|j_h - \overline{v}(t)\rho_h\|_{h}^2  \|k_h\|_h,
\end{align} 
where we have used the discrete embedding $\|f_h\|_{4,h} \lesssim h^{-1/4}\|f_h\|_h$. This estimate aligns the norms to those of the proof of Theorem \ref{thm: constant-kernel-flocking-iid} at the cost of introducing the factor $h^{-1/2}$. However, this is not an issue, as the term $\|(j_h - \overline{v}(t)\rho_h)\sqrt{k_h}\|^2$ is only multiplied by $\sigma^2 N^{-1}$ to begin with, so the overall constant $\sigma^2 N^{-1} h^{-1/2}$ is dominated by the term $\sigma^2 N^{-1} \Tr(\xi_h)$ by assumption. Overall, this allows us to reuse content from \emph{Step 6} from proof of Theorem \ref{thm: constant-kernel-flocking-iid} and get, for some $\delta > 0$,
\begin{align}\label{flocking_Y_pth_moment_exp}
    \mathbb{E}[e^{\frac{p}{2}\delta C_a(t\wedge \tau_\rho)}Y^{p/2}_{t\wedge \tau_\rho}] \leq \mathbb{E}[Y_0^{p/2}].
\end{align}

Next, we need to account for the deviations of $\overline{v}(t)$ from $\overline{v}(0)$, which we haven't needed so far for the estimates on $Y^{p/2}$, but which play into the final result. We can write 
\begin{align}
    \mathbb{E}\left[\sup_{t\in[0,T]}|\overline{v}(t \wedge \tau_\rho)-\overline{v}(0)|^p\right] & = \mathbb{E}\left[\sup_{t\in[0,T]}\left|\int_{0}^{t \wedge \tau_\rho}{\sigma N^{-1/2}(\sqrt{k_h},1)_h\dd x \dd\xi_h}\right|^p\right], \nonumber \\
    &\lesssim p^{p} \mathbb{E}\left[ \left(\sigma^2 N^{-1} \int_{0}^{T \wedge \tau_\rho}{\|k(t)\|_{h,1}dt} \right)^{p/2}\right], \nonumber \\
    &\lesssim p^p\E{\brack{\sigma^2 N^{-1}\int_0^T \mathbf{1}_{\{t \leq \tau_\rho\}} Y_t \dd t}^{p/2}}, \nonumber \\
    &\lesssim p^p (\sigma^2 N^{-1})^{p/2} \brack{\int_0^T \E{\mathbf{1}_{\{t \leq \tau_\rho\}}Y_t^{p/2}}^{2/p} \dd t}^{p/2} \label{bockner} \\
    & \lesssim p^{p} \left( \sigma^2 N^{-1} \frac{1}{ \delta C_a} \right)^{p/2} \propto 
    p^{p}\left(\sigma^2 N^{-1}\frac{1}{ \delta C_a} \right)^{p/2}\label{bound_fluct_v_bar_exp},
\end{align}
where we have used properties of the Bochner integral in \eqref{bockner} and \eqref{flocking_Y_pth_moment_exp} in \eqref{bound_fluct_v_bar_exp}.  
Similar arguments as in the proof of Theorem \ref{thm: constant-kernel-flocking-iid} grant 
$$
\P{\sup_{t\geq 0}{|\overline{v}(t \wedge \tau_\rho) - \overline{v}(0)|} > \overline{v}(0)/2} \lesssim \exp\{-C(\sigma^{-2} N^{}\delta C_a \overline{v}^2(0))^{\alpha} \},
$$
for some $C, \alpha > 0$. Now, we set $\tau_v := \inf \cbrack{t \geq 0 : |\overline{v}(t \wedge \tau_\rho)-\overline{v}(0)| \geq \overline{v}(0)/2}$, and define
\begin{align}\label{defn_stop_time_iid_general_exp}
    \tau := \min\{ \tau_\rho ; \tau_v \}.
\end{align}
Dealing with powers of $\eta := \|j_h\|_h^2 + \overline{v}^2(t)\|\rho_h\|^2_h + \|k_h\|_h$ up to the stopping time $\tau$ is handled similarly to the constant kernel case (with the caveat of introducing the -- non consequenatial -- term $h^{-1/2}$ in the estimate of the term $\|j_h\sqrt{k_h}\|^2_h$, same discussion as in \eqref{l4_embed}). Following the same logic as in the constant kernel case, this ultimately implies
\begin{align*}
\P{\tau < \infty} 
& \lesssim \exp\{-(\sigma^{-2} N^{}\delta C_a \overline{v}^2(0))^{\alpha} \} + \exp\{ - C(\sigma N^{-1/2}\Tr^{1/2}(\xi_h))^{-2}\},
\end{align*}
and the sketch of the proof is complete. 
\end{proof} 

 \section{Error proof}\label{app: error}

In this section we give an error proof measured in a weighted $H^{-2} \times H^{-2}$ norm between the particle system and reduced SPDE in the common noise setting in $d=1$.

We first need a flocking result for the particle system.

\begin{theorem}\label{thm: particle-flocking}
    Assume $v_i(0) \in L^2(\Omega; \mathbb{R})$ are $\mathcal{F}_0$-adapted and independent of the Brownian motion $B(t)$. Suppose we are in $d=1$ with positions on the torus. Then we have the following flocking result:
    $$
    \E{\sum_{i=1}^N (v_i - \vbar)^2}^{1/2} \leq N^{1/2}R_v \exp \brack{\frac{1}{2}(\sigma^2 - 2a_{\mathrm{min}})t},
    $$
    where $a_{\mathrm{min}} \coloneqq \inf_{x \in \mathbb{T}}a(x)$ and $R_v^2 \coloneqq \max_{1 \leq i \leq N} \E{(v_i(0) - \vbar)^2}$.
\end{theorem}

\begin{proof}
    Define $u_i \coloneqq v_i - \vbar$. Then
    $$
    \dd u_i = N^{-1}\sum_{j=1}^N a(x_i - x_j)(u_j - u_i) \dd t + \sigma u_i \dd B(t).
    $$
    We now apply It\^o's lemma using $\varphi(x) = x^2$ to get
    $$
    \sum_{i=1}^N \dd u_i^2 = \sum_{i=1}^N \brack{2N^{-1}\sum_{j=1}^N a(x_i - x_j)(u_j - u_i) \dd t + 2\sigma u_i \dd B(t)}u_i + \sum_{i=1}^N \sigma^2 u_i^2 \dd t.
    $$
    Defining $w \coloneqq \sum_{i=1}^N u_i^2$, we rewrite the above as
    $$
    \dd w = \sigma^2 w \dd t + 2 \sigma w \dd B(t) + \mathcal{S} \dd t,
    $$
    where
    $\mathcal{S} \coloneqq 2N^{-1}\sum_{i, j=1}^N a(x_i - x_j)(u_j - u_i)u_i$.
    By symmetry, we can swap the dummy indices $(i, j)$ in the definition of $\mathcal{S}$ and add these expressions together to obtain
    \begin{align*}
    \mathcal{S} &= N^{-1}\sum_{i, j=1}^N a(x_i - x_j)\brack{u_i(u_j - u_i) + u_j(u_i - u_j)}, \\
    &= - N^{-1}\sum_{i, j=1}^N a(x_i-x_j)(u_i - u_j)^2, \\
    &\leq -a_{\mathrm{min}}N^{-1}\sum_{i, j=1}^N (u_i - u_j)^2,
    \end{align*}
    where $a_{\mathrm{min}}\coloneqq \inf_{x \in \mathbb{T}}a(x)$ is attained since we are on the torus. Then note that the cross terms in the sum vanish by definition of $u_i$ giving
    $$
    \mathcal{S} \leq -2a_{\mathrm{min}} \sum_{i=1}^N u_i^2 = -2a_{\mathrm{min}}w.
    $$
    Altogether, we have
    $$
    \dd w \leq (\sigma^2 - 2a_{\mathrm{min}})w \dd t + 2\sigma w \dd B(t).
    $$
    Appealing to a stochastic comparison principle we have $w \leq \hat{w}$ $\mathbb{P}$-a.s. where $\hat{w}$ satisfies
    $$
    \dd \hat{w} = (\sigma^2 - 2a_{\mathrm{min}})\hat{w}\dd t + 2 \sigma \hat{w} \dd B(t).
    $$
    Taking expectations and using the independence of the initial velocities and the Brownian motion, it follows that
    $$
    \E{w}^{1/2} = \E{\sum_{i=1}^N (v_i - \vbar)^2}^{1/2} \leq \E{\sum_{i=1}^N (v_i(0) - \vbar)^2}^{1/2} \exp \brack{\frac{1}{2}(\sigma^2 - 2a_{\mathrm{min}})t}.
    $$
    Moreover, if we further assume that the initial velocities $v_i(0) \in L^2(\Omega; \mathbb{R})$ for all $i=1, \ldots, N$ then for any $N \in \mathbb{N}$ we can define
    $$
    R_v^2 \coloneqq \max_{1 \leq i \leq N} \E{(v_i(0) - \vbar)^2}.
    $$
    It follows that
    $$
    \E{\sum_{i=1}^N (v_i(0) - \vbar)^2} \leq NR_v^2,
    $$
    which gives
    $$
    \E{\sum_{i=1}^N (v_i - \vbar)^2}^{1/2} \leq N^{1/2}R_v \exp \brack{\frac{1}{2}(\sigma^2 - 2a_{\mathrm{min}})t}.
    $$
\end{proof}

We then need an analogue of \cite[Proposition 5.1]{zimper2025reduced} which estimates the convolutional term in the (S)PDE when the Dirac delta is replaced with a von Mises kernel.

\begin{proposition}\label{prop: mollification-error}
    Assume the interaction potential $a$ has bounded first derivative. Then
    \begin{equation}\label{eq: replace-dirac-with-vm-conv}
    N^{-1}\sum_{i=1}^N \delta_\ve(x-x_i) N^{-1}\sum_{j=1}^N a(x_i - x_j)(v_j - v_i) = \rho_\ve(a*j_\ve) - j_\ve(a*\rho_\ve) + R_\ve,
    \end{equation}
    where
    \begin{multline}\label{eq: pathwise-Reps}
        R_\ve \lesssim \absv{\rho_\ve N^{-1}\sum_{j=1}^N (v_j - \vbar)r_{\ve, j}} + N^{-1}\sum_{i=1}^N \delta_\ve(x-x_i)\absv{x-x_i}\brack{N^{-1}\sum_{j=1}^N \absv{v_j - \vbar}}, \\
        + \absv{(j_\ve - \vbar \rho_\ve)N^{-1}\sum_{j=1}^N r_{\ve, j}} + N^{-1}\sum_{i=1}^N \absv{v_i - \vbar}\delta_\ve(x-x_i)\absv{x-x_i} =: \sum_{k=1}^4 I_k,
    \end{multline}
    where $\norm{r_{\ve, j}}_{L^\infty} \lesssim \sqrt{\ve}$ for each $j=1, \ldots, N$. Furthermore, we have
    $$
    \E{\norm{R_\ve}_{H^{-2}}} \lesssim C(a) R_v \exp \brack{\frac{1}{2}(\sigma^2 - 2a_{\mathrm{min}})t}\brack{\ve + \sqrt{\ve}}.
    $$
\end{proposition}

\begin{proof}
    Obtaining \eqref{eq: replace-dirac-with-vm-conv} with the estimate \eqref{eq: pathwise-Reps} is identical to \cite[Proposition 5.1]{zimper2025reduced}. It remains to bound each term in \eqref{eq: pathwise-Reps} in expectation. We first recall that the $H^{-2}$ norm is defined by
    $$
    \norm{f}_{H^{-2}} \coloneqq \sup_{\varphi \in H^2:\norm{\varphi}_{H^2} = 1} \absv{\int f \varphi},
    $$
    and the inequality
    $$
    \sum_{i=1}^N \absv{b_i} \leq N^{1/2}\brack{\sum_{i=1}^N \absv{b_i}^2}^{1/2}.
    $$
    Moreover, we have the embeddings $H^2 \hookrightarrow H^1 \hookrightarrow L^\infty$, the bound $\norm{r_{\ve, j}}_{L^\infty} \lesssim \sqrt{\ve}$ and Gaussian moment bounds granting
    $$
    \int \delta_\ve(x-x_i)\absv{x-x_i} \dd x \lesssim \ve.
    $$
    Using these ingredients we obtain
    $$
    \absv{\int_{\mathbb{T}}I_1 \varphi \dd x} = \absv{\int_{\mathbb{T}} \rho_\ve(x) N^{-1}\sum_{j=1}^N (v_j-\vbar)r_{\ve, j}(x) \varphi(x) \dd x} \lesssim \norm{\varphi}_{H^2}\sqrt{\ve}N^{-1} N^{1/2} \brack{\sum_{j=1}^N(v_j - \vbar)^2}^{1/2}.
    $$
    Taking expectations, applying Jensen's inequality, using the flocking estimate in Theorem \ref{thm: particle-flocking} and then using the characterisation of the $H^{-2}$ norm, as well as the one-dimensional Sobolev embedding $H^2\subset L^\infty$, gives
    $$
    \E{\norm{I_1}_{H^{-2}}} \lesssim R_v \exp \brack{\frac{1}{2}(\sigma^2 - 2a_{\mathrm{min}})t}\sqrt{\ve}.
    $$
    Similar procedures give analogous bounds for $I_2, I_3, I_4$ with either a $\ve$ or $\sqrt{\ve}$ factor (arising from estimating $\int \delta_\ve(x-x_i)\absv{x-x_i} \dd x)$ or $r_{\ve, j}$ respectively) thereby giving the result.
\end{proof}

We now have the required ingredients to prove the following theorem that quantifies the error between the particle system and reduced SPDE. We use the weighted norm $\hmtwo$ defined by $\norm{(f, g)}_{\hmtwo}^2 \coloneqq \vbar^2 \norm{f}^2_{H^{-2}} + \norm{g}^2_{H^{-2}}$.

\begin{proof}[Proof of Theorem \ref{thm: error}] 

    \textbf{Step 1: functional setup.} Fix $T>0$ and some $t \in [0, T]$. Note that the linear differential operator $(-\partial_x j, -\vbar^2 \partial_x \rho)$ generates a $C^0$ semigroup of contractions in $\hmtwo$. Therefore, for any $s \leq t$ we have
    \begin{multline*}
    \xpde(s \wedge \tau) = S(s \wedge \tau)\xpde(0) + \int_0^{s \wedge \tau} S(s \wedge \tau - r) F(\xpde(r)) \dd r \\ 
    + \int_0^{s \wedge \tau} S(s \wedge \tau - r) B(\xpde(r)) \dd W(r),
    \end{multline*}
    and
    \begin{multline*}
    \xcs(s \wedge \tau) = S(s \wedge \tau)\xcs(0) + \int_0^{s \wedge \tau} S(s \wedge \tau - r) F(\xcs(r)) \dd r \\ 
    + \int_0^{s \wedge \tau}S(s \wedge \tau - r)R(r) \dd r + \int_0^{s \wedge \tau} S(s \wedge \tau - r) B(\xcs(r)) \dd W(r),
    \end{multline*}
    where $R(r) \coloneqq R_{\mathrm{closing}}(r) + R_\ve(r)$. The term $R_{\ve}(r)$ is bounded as in Proposition \ref{prop: mollification-error} and 
    $$
    R_{\mathrm{closing}}(r) \coloneqq \brack{0, -N^{-1}\sum_{i=1}^N (v_i^2(r) - \vbar^2) \delta'_\ve(x-x_i(r))}.
    $$

  Subtracting the two mild formulations, using the linearity of $B$, taking the supremum over times followed by the expectation, we get 
    \begin{multline*}
        \E{\sup_{s \in [0, t]}\norm{\xpde(s \wedge \tau) - \xcs(s \wedge \tau)}_{\hmtwo}} \lesssim \E{\sup_{s \in [0, t]}\norm{S(s \wedge \tau)(\xpde(0) - \xcs(0))}_{\hmtwo}} \\
        + \E{\sup_{s \in [0, t]}\norm{\int_0^{s \wedge \tau}S(s \wedge \tau - r) \brack{F(\xpde(r)) - F(\xcs(r))} \dd r}_{\hmtwo}} \\
        + \E{\sup_{s \in [0, t]}\norm{\int_0^{s \wedge \tau}S(s \wedge \tau - r) \brack{B(\xpde(r) - \xcs(r))} \dd W(r)}_{\hmtwo}} \\
        + \E{\sup_{s \in [0, t]}\norm{\int_0^{s \wedge \tau}S(s \wedge \tau - r)R(r) \dd r}_{\hmtwo}} =: \sum_{k=1}^4 J_k.
    \end{multline*}

\textbf{Step 2: dealing with $J_1$.} For $J_1$, we simply use contractivity of the semigroup to write
$$
J_1 \lesssim \E{\norm{\xpde(0) - \xcs(0)}_{\hmtwo}}.
$$

\textbf{Step 3: dealing with $J_2$.} We first recall from \cite[Proof of Theorem 3.2]{zimper2025reduced} that we have the following estimate:
\begin{multline*}
    \norm{F(\xpde) - F(\xcs)}_{\hmtwo} \lesssim C(a)\left(\norm{\rho - \rho_\ve}_{H^{-2}}\norm{j-\vbar \rho}_{H^{-2}} \right . \\
    + \left . \norm{\rho_\ve}_{H^{-2}}\norm{(j-\vbar \rho) - (j_\ve - \vbar\rho_\ve)}_{H^{-2}}\right ).
\end{multline*}
We then write
\begin{align*}
J_2 &= \E{\sup_{s \in [0, t]}\int_0^{s \wedge \tau}\norm{S(s \wedge \tau - r)\mathbf{1}_{[0, \tau]}(r) (F(\xpde(r)) - F(\xcs(r)))}_{\hmtwo} \dd r} \\
&\leq \int_0^t \E{\mathbf{1}_{[0, \tau]}(r) \norm{F(\xpde(r)) - F(\xcs(r)))}_{\hmtwo}} \dd r,
\end{align*}
where we have used contractivity of the semigroup. We now need to estimate the terms in $\norm{F(\xpde - F(\xcs)}_{\hmtwo}$ in order to get something we can apply a Gr\"onwall inequality to. We first notice that $\norm{\rho_\ve}_{H^{-2}} \lesssim 1$ due to the embedding $H^2 \hookrightarrow L^\infty$. We then write
$$
\E{\mathbf{1}_{[0, \tau]}(r) \norm{(j-\vbar \rho) - (j_\ve - \vbar \rho_\ve)}_{H^{-2}}} \lesssim \E{\mathbf{1}_{[0, \tau]}(r) \norm{j-\vbar \rho}} + \E{\norm{j_\ve - \vbar \rho_\ve}_{H^{-2}}}.
$$
For the first term we note that for $r \leq \tau$ we have 
$$
\norm{j-\vbar \rho} \leq \norm{j_0 - \vbar \rho_0}\exp \brack{\frac{1}{2}(\sigma^2 - C_a)r}\sqrt{M_r},
$$
where $M_r$ is an exponential martingale with mean $1$ (see Lemma \ref{lem: p=2-flocking-bound}). It follows that 
\begin{align*}
\E{\mathbf{1}_{[0, \tau]}(r) \norm{j-\vbar \rho}} &\lesssim \E{\norm{j_0 - \vbar \rho_0}\exp \brack{\frac{1}{2}(\sigma^2 - C_a)r}\sqrt{M_r}} \\
&\lesssim \E{\norm{j_0 - \vbar \rho_0}^2}^{1/2}\exp \brack{\frac{1}{2}(\sigma^2 - C_a)r},
\end{align*}
where we have used Cauchy-Schwarz and the independence of the initial data and the Brownian motion. For the second term, using the characterisation of the $H^{-2}$ norm and Theorem \ref{thm: particle-flocking}, it follows that
$$
\E{\norm{j_\ve(r) - \vbar \rho_\ve(r)}_{H^{-2}}} \lesssim N^{-1/2}\brack{\sum_{i=1}^N (v_i(r) - \vbar)^2}^{1/2} \lesssim R_v \exp \brack{\frac{1}{2}(\sigma^2 - 2a_{\mathrm{min}})r}.
$$
Bounding the other contribution to $\norm{F(\xpde - F(\xcs)}_{\hmtwo}$ is more subtle since we need to recover (a power $\in (0,1)$ of) the object we want to use for Gr\"onwall. We write 
\begin{multline*}
\E{\mathbf{1}_{[0, \tau]}(r) \norm{\rho-\rho_\ve}_{H^{-2}} \norm{j-\vbar \rho}_{H^{-2}}} \\ 
\lesssim \E{\mathbf{1}_{[0, \tau]}(r) \norm{\rho-\rho_\ve}_{H^{-2}}\cdot\norm{\rho-\rho_\ve}_{H^{-2}}}^{1/2}\E{\mathbf{1}_{[0, \tau]}(r) \norm{j-\vbar \rho}^2}^{1/2}.
\end{multline*}
For the first expectation, we bound each instance of $\norm{\rho-\rho_\ve}_{H^{-2}}$ separately. First note that $\norm{\rho - \rho_\ve}_{H^{-2}} \lesssim C(\vbar) \norm{\xpde(r) - \xcs(r)}_{\hmtwo}$. Moreover, for $r \leq \tau$ we have
$$
\norm{\rho(r) - \rho_\ve(r)}_{H^{-2}} \lesssim \norm{\rho(r)} + \norm{\rho_\ve(r)}_{H^{-2}} \lesssim \frac{C_a}{4 \theta \norm{g}} + 1.  
$$
The second expectation is bounded using the flocking estimate which holds for $r \leq \tau$ giving
$$
\E{\mathbf{1}_{[0, \tau]}(r) \norm{j-\vbar \rho}^2}^{1/2} \lesssim \E{\norm{j_0 - \vbar \rho_0}^2}^{1/2}\exp \brack{\frac{1}{2}(\sigma^2 - C_a)r}.
$$
Put together, it follows that
\begin{multline*}
\E{\mathbf{1}_{[0, \tau]}(r) \norm{\rho-\rho_\ve}_{H^{-2}} \norm{j-\vbar \rho}_{H^{-2}}} \\
\lesssim C(\vbar) \brack{1+ \frac{C_a}{4 \theta \norm{g}}}\E{\norm{\xpde(r \wedge \tau) - \xcs(r \wedge \tau)}_{H^{-2}}}^{1/2}\E{\norm{j_0 - \vbar \rho_0}^2}^{1/2}\exp \brack{\frac{1}{2}(\sigma^2 - C_a)r}.
\end{multline*}
Gathering all of these estimates gives
\begin{multline*}
    J_2 \lesssim \E{\norm{j_0 - \vbar \rho_0}^2}^{1/2} \int_0^t \exp \brack{\frac{1}{2}(\sigma^2-C_a)r} \dd r + R_v \int_0^t \exp \brack{\frac{1}{2}(\sigma^2 - 2a_{\mathrm{min}})r} \dd r \\
    + \biggl [ \biggr . C(\vbar)\brack{1+\frac{C_a}{4\theta \norm{g}}}\E{\norm{j_0 - \vbar \rho_0}^2}^{1/2} \\
    \times \int_0^t \exp \brack{\frac{1}{2}(\sigma^2 - C_a)r}\E{\norm{\xpde(r \wedge \tau) - \xcs(r \wedge \tau)}_{H^{-2}}}^{1/2} \dd r \biggl . \biggr ].
\end{multline*}
Integrating the first two terms in time and smuggling a supremum inside the latter term gives
\begin{multline*}
J_2 \lesssim \frac{2\E{\norm{j_0 - \vbar \rho_0}^2}^{1/2}}{C_a -  \sigma^2} + \frac{2R_v}{2a_{\mathrm{min}}-\sigma^2} + \biggl [ \biggr . C(\vbar)\brack{1+\frac{C_a}{4\theta \norm{g}}}\E{\norm{j_0 - \vbar \rho_0}^2}^{1/2} \\
\times \int_0^t \exp \brack{\frac{1}{2}(\sigma^2 - C_a)r}\E{\sup_{u \in [0, r]}\norm{\xpde(u \wedge \tau) - \xcs(u \wedge \tau)}_{H^{-2}}}^{1/2} \dd r \biggl . \biggr ].
\end{multline*}

\textbf{Step 4: dealing with $J_3$.} First observe that 
$$
J_3 \leq \E{\sup_{q \in [0, t]}\norm{\int_0^q S(q-r) \mathbf{1}_{[0, \tau]}(r) B(\xpde(r) - \xcs(r)) \dd W(r)}_{\hmtwo}}.
$$
We now apply a maximal estimate for stochastic convolutions \cite{van2011maximal} and use Jensen's inequality to get
\begin{align*}
J_3 &\lesssim \brack{\E{\int_0^t \mathbf{1}_{[0, \tau]}(r) \norm{B(\xpde(r) - \xcs(r))}_{L_2^0(\hmtwo)}^2} \dd r}^{1/2} \\
&\lesssim C(\sigma)\brack{\E{\int_0^t \mathbf{1}_{[0, \tau]}(r) \norm{(j-\vbar \rho) - (j_\ve - \vbar \rho_\ve)}_{H^{-2}}^2} \dd r}^{1/2} \\
&\lesssim C(\sigma) \brack{\int_0^t \E{\mathbf{1}_{[0, \tau]}(r) \norm{j-\vbar \rho}^2} + \E{\norm{j_\ve - \vbar \rho_\ve}_{H^{-2}}^2} \dd r}^{1/2}.
\end{align*}
The integrands can then be handled similarly to the previous step to give
$$
\int_0^t \E{\mathbf{1}_{[0, \tau]}(r) \norm{j-\vbar \rho}^2} \dd r \lesssim \frac{\E{\norm{j_0-\vbar \rho_0}^2}}{C_a - \sigma^2}
$$
and
$$
\int_0^t \E{\norm{j_\ve - \vbar \rho_\ve}_{H^{-2}}^2} \dd r \lesssim \frac{R_v^2}{2a_{\mathrm{min}}-\sigma^2}.
$$
Therefore,
$$
J_3 \lesssim C(\sigma) \brack{\frac{\E{\norm{j_0-\vbar \rho_0}^2}}{C_a - \sigma^2} + \frac{R_v^2}{2a_{\mathrm{min}}-\sigma^2}}^{1/2}.
$$

\textbf{Step 5: dealing with $J_4$.} Write
\begin{align*}
    J_4 &= \E{\sup_{s \in [0, t]} \norm{\int_0^{s \wedge \tau} S(s \wedge \tau - r)\mathbf{1}_{[0, \tau]}(r) R(r)}_{\hmtwo} \dd r} \\
    &\leq \E{\sup_{s \in [0, t]} \int_0^s \mathbf{1}_{[0, \tau]}(r) \norm{R(r)}_{\hmtwo}} \\
    &\leq \int_0^t \E{\norm{R(r)}_{\hmtwo}} \dd r \\
    &\leq \int_0^t \E{\norm{R_\ve(r)}_{\hmtwo}} + \E{\norm{R_{\mathrm{closing}}(r)}_{\hmtwo}} \dd r.
\end{align*}
Then $\E{\norm{R_\ve(r)}_{\hmtwo}}$ is bounded as in Proposition \ref{prop: mollification-error}. Moreover, similarly to the proof of Proposition \ref{prop: mollification-error} but using the one-dimensional embedding $H^2 \hookrightarrow W^{1, \infty}$, we can estimate
$$
\E{\norm{R_{\mathrm{closing}}(r)}_{\hmtwo}} \lesssim R_v \exp \brack{(\sigma^2 - 2a_{\mathrm{min}})r}.
$$
Inserting these estimates then integrating in time gives 
$$
J_4 \lesssim \frac{2R_v(1+C(a)(\ve + \sqrt{\ve}))}{2a_{\mathrm{min}}-\sigma^2}.
$$

\textbf{Step 6: writing the estimate in a form suitable for a Gr\"onwall inequality.}
Define the function
$$
\Psi(t) \coloneqq \E{\sup_{s \in [0, t]}\norm{\xpde(s \wedge \tau) - \xcs(s \wedge \tau)}_{\hmtwo}}.
$$
Then combining the results of Steps 2--5 gives
$$
\Psi(t) \lesssim c + \int_0^t b(r) \Psi^{1/2}(r) \dd r,
$$
where
$$
b(r) \coloneqq C(\vbar)\brack{1+\frac{C_a}{4\theta \norm{g}}}\E{\norm{j_0 - \vbar \rho_0}^2}^{1/2} \exp \brack{\frac{1}{2}(\sigma^2 - C_a)r} \geq 0
$$
and
\begin{multline*}
c \coloneqq \E{\norm{\xpde(0) - \xcs(0)}_{\hmtwo}} + \frac{2\E{\norm{j_0 - \vbar \rho_0}^2}^{1/2}}{C_a -  \sigma^2} \\
+ \frac{2R_v}{2a_{\mathrm{min}}-\sigma^2} + C(\sigma) \brack{\frac{\E{\norm{j_0-\vbar \rho_0}^2}}{C_a - \sigma^2} + \frac{R_v^2}{2a_{\mathrm{min}}-\sigma^2}}^{1/2} + \frac{2R_v(1+C(a)(\ve + \sqrt{\ve}))}{2a_{\mathrm{min}}-\sigma^2} \geq 0.
\end{multline*}

\textbf{Step 7: Application of Gr\"onwall and tidying up.}
We can now apply a non-linear generalisation of Gr\"onwall's inequality \cite[Theorem 21]{dragomir2003applications} to get
$$
\Psi(t) \lesssim \brack{c^{1/2} + \frac{1}{2}\int_0^t b(r) \dd r}^2.
$$
Note that
$$
\frac{1}{2}\int_0^t b(r) \dd r \leq C(\vbar)\brack{1+\frac{C_a}{4\theta \norm{g}}}\frac{\E{\norm{j_0 - \vbar \rho_0}^2}^{1/2}}{C_a - \sigma^2}.
$$
We then use that $(a+b)^2 \leq 2a^2 + 2b^2$ and note that $t \in [0, T]$ was arbitrary so we can extend the supremum up to the final time. Finally, since the right-hand side is independent of $T$ we can extend the supremum over all times $t \geq 0$ which gives the claimed result. 
\end{proof} 
\end{document}